\documentclass[11pt]{article}

\usepackage[a4paper,margin=1in]{geometry}

\usepackage[utf8]{inputenc}

\usepackage[ngerman,british]{babel}

\usepackage[T1]{fontenc}
\usepackage{mathpazo,tgcursor,tgpagella}
\usepackage{textcomp}

\usepackage[bookmarksnumbered,bookmarksopen,unicode]{hyperref}
\hypersetup{
	pdftitle={An Online-Learning Approach to Inverse Optimization},
	pdfauthor={Andreas Bärmann,Alexander Martin,Sebastian Pokutta,Oskar Schneider},
	pdfkeywords={},
	pdfsubject={MSC2000 Primary ;Secondary} 
}

\usepackage{amsthm}

\usepackage[authoryear]{natbib}

\usepackage{subfig}
\usepackage{authblk}
\usepackage{mathtools}

\usepackage{amsmath,amssymb}
\usepackage{eurosym}
\usepackage{color}
\usepackage{algorithm,eqparbox,array}
\usepackage{algorithmic}

\usepackage{xspace}

\usepackage{float}
\usepackage{booktabs}

\usepackage{graphicx}
\usepackage{color}
\usepackage{epstopdf}
\DeclareGraphicsExtensions{.pdf,.eps,.png}
\usepackage{stackengine}
\usepackage{tikz}
\usetikzlibrary{shapes.misc}

\usepackage{chngcntr}
\counterwithout{equation}{section}

\usepackage{paralist}
\usepackage{enumitem}

\usepackage{bbm}

\usepackage{calc}

\usepackage{todonotes}
\presetkeys{todonotes}{inline}{}

\newtheorem{theorem}{Theorem}[section]
\newtheorem{corollary}[theorem]{Corollary}
\newtheorem{definition}[theorem]{Definition}
\newtheorem{remark}[theorem]{Remark}

\newtheorem{proposition}[theorem]{Proposition}
\newtheorem{lemma}[theorem]{Lemma}

\newcommand{\F}{\{0, 1\}}
\newcommand{\FR}{[0, 1]}
\newcommand{\N}{\mathbbm{N}}
\newcommand{\Z}{\mathbbm{Z}}

\newcommand{\R}{\mathbbm{R}}

\newcommand{\E}{\mathbbm{E}}

\newcommand{\OO}{\mathcal{O}}

\newcommand{\card}[1]{\lvert #1 \rvert}

\newcommand{\abs}[1]{\lvert #1 \rvert}
\newcommand{\norm}[1]{\lVert #1 \rVert}

\newcommand{\choice}[1]{\left\{\begin{array}{rl} #1 \end{array}\right.}

\renewcommand{\sp}[2]{#1^\top #2}
\newcommand{\tp}[1]{#1^\top}

\newcommand{\sol}[1]{\bar{#1}}
\newcommand{\arb}[1]{\hat{#1}}

\newcommand{\set}[1]{\{#1\}}
\newcommand{\lrset}[1]{\left\{#1\right\}}

\newcommand{\qm}[1]{``#1''}

\newcommand{\ie}{i.e.\ }

\newcommand{\cf}{cf.\ }
\newcommand{\eg}{e.g.\ }

\newcommand{\mycomment}[1]{}

\newcommand{\st}{\mathrm{s.t.}}

\newlength\mysinglespace
\newlength\objspace
\newlength\conspace
\newlength\cconspace
\newenvironment{constr}[1]{\begin{array}[t]{#1}}{\end{array}} 

\newenvironment{opt*}[3]{\begin{equation*}\begin{array}{rl}#1 & #2 \\[\objspace] \st & \begin{constr}{#3}}{\end{constr}\end{array}\end{equation*}\\[0pt]}

\newenvironment{mip}{\alignat{4}}{\endalignat}
\newcommand{\objective}[2]{#1\>\>\>\> & #2 \span\span\span\span\span\span \nonumber}
\newcommand{\stconstraint}[4]{\st\>\>\>\> && #1 && \;\;\;\, #2 & \;\;\;\, #3 && \quad \begin{array}{l} #4 \end{array} \nonumber}
\newcommand{\constraint}[4]{&& #1 && \;\;\;\, #2 & \;\;\;\, #3 && \quad \begin{array}{l} #4 \end{array} \nonumber}

\newenvironment{eqns*}[1]{\begin{equation*}\begin{array}[t]{#1}}{\end{array}\end{equation*}\\[0pt]}

\newtheorem{example}[theorem]{Example}

\newcommand{\OPTp}{$ \textrm{OPT}(p) $\ }
\newcommand{\OPTpp}{$ \textrm{OPT}(p) $}

\newcommand{\ctrue}{c_\text{true}}

\newcommand{\btrue}{b_\text{true}}

\newcommand{\argmin}[1]{\operatorname{argmin}{\set{#1}}}
\newcommand{\lrargmin}[1]{\operatorname{argmin}{\lrset{#1}}}
\newcommand{\argmax}[1]{\operatorname{argmax}{\set{#1}}}

\newcommand{\diam}{\operatorname{diam}}
\renewcommand{\vec}{\operatorname{vec}}

\newcommand{\algorithmicinput}{\textbf{Input:}}
\newcommand{\algorithmicoutput}{\textbf{Output:}}
\newcommand{\INPUT}{\item[\algorithmicinput]}
\newcommand{\OUTPUT}{\item[\algorithmicoutput]}

\newcommand{\graphicpathknp}{Graphics/KNP}
\newcommand{\graphicpathsp}{Graphics/SP}
\newcommand{\graphicpathtsp}{Graphics/TSP}

\title{An Online-Learning Approach to Inverse Optimization}
\author[1]{Andreas~Bärmann}
\author[1]{Alexander Martin}
\author[2]{Sebastian Pokutta}
\author[3]{Oskar Schneider\vspace{2\baselineskip}}
\affil[1]{
\url{Andreas.Baermann@math.uni-erlangen.de}

\url{Alexander.Martin@math.uni-erlangen.de}

Lehrstuhl für Wirtschaftsmathematik,

Department Mathematik,

Friedrich-Alexander-Universität Erlangen-Nürnberg,

Cauerstraße 11, 91058 Erlangen, Germany
\vspace{\baselineskip}
}
\affil[2]{
\url{Pokutta@zib.de}

Arbeitsgruppe Diskrete und Algorithmische Mathematik,

Institut für Mathematik,

Technische Universität Berlin and Zuse-Institut Berlin,

Takustraße 7, 10623 Berlin, Germany
\vspace{\baselineskip}
}
\affil[3]{
\url{Oskar.Schneider@fau.de}

Gruppe Data Science and Optimization 

Fraunhofer Arbeitsgruppe für Supply-Chain Services SCS,

Fraunhofer Institut für Integrierte Schaltungen IIS,

Nordostpark 93, 90411 Nürnberg, Germany
}

\begin{document}

\maketitle

\begin{abstract}
In this paper, we demonstrate how to learn
the objective function of a decision-maker
while only observing the problem input data
and the decision-maker's corresponding decisions
over multiple rounds.
We present exact algorithms for this online version
of inverse optimization which
converge at a rate of $ \OO(1/\sqrt{T}) $
in the number of observations~$T$
and compare their further properties.
Especially, they all allow taking decisions
which are essentially as good
as those of the observed decision-maker
already after relatively few iterations,
but are suited best for different settings each.
Our approach is based on online learning
and works for linear objectives over arbitrary feasible sets
for which we have a linear optimization oracle.
As such, it generalizes previous approaches
based on KKT-system decomposition
and dualization.
We also introduce several generalizations,
such as the approximate learning of non-linear objective functions,
dynamically changing as well as parameterized objectives
and the case of suboptimal observed decisions.
When applied to the stochastic offline case,
our algorithms are able to give guarantees
on the quality of the learned objectives in expectation.
Finally, we show the effectiveness and possible applications
of our methods in indicative computational experiments.

\end{abstract}

\section{Introduction}
\label{Sec:Introduction}
Human decision-makers are very good at taking decisions under rather
imprecise specification of the decision-making problem, both in terms
of constraints as well as objective.
One might argue that the human decision-maker
can pretty reliably learn from observed previous
decisions -- a traditional learning-by-example setup.
At the same time, when we try to turn these decision-making problems
into actual optimization problems,
we often run into all types of issues in terms
of specifying the model.
In an optimal world, we would be able to \emph{infer or learn}
the optimization problem from previously
observed decisions taken by an expert.

This problem naturally occurs in many settings
where we do not have direct access
to the decision-maker's preference or objective function
but can observe the resulting behaviour, 
and where the learner as well as the decision-maker
have access to the same information.
Natural examples are as diverse as making recommendations
based on user history and strategic planning problems,
where the agent's preferences are unknown
but the system is observable.
Other examples include knowledge transfer
from a human planner into a decision support system:
often human operators have arrived
at finely-tuned \qm{objective functions}
through many years of experience,
and in many cases it is desirable
to replicate the decision-making process
both for scaling up and also for potentially including it
in large-scale scenario analysis and simulation
to explore responses under varying conditions. 

Here, we consider the learning of
preferences or objectives from an expert
by means of observing the actions taken. 
More precisely, we observe a set
of input parameters and corresponding decisions of the form
$ \set{(p_1, x_1), \dots, (p_T, x_T)} $.
They are such that $ p_t \in P $ with $ t = 1, \ldots, T $
is a certain realization of problem parameters
from a given set~$ P \subseteq \R^k $.
Furthermore, $ x_t $ is an optimal solution
to the optimization problem 
\begin{equation*}
	\max\set{\sp{\ctrue}{x} \mid x \in X(p_t)},
\end{equation*}
where $ \ctrue \in \R^n $ is the expert's true but unknown objective
and $ X(p_t) \subseteq \R^n $ for some (fixed)~$n$.
We assume that we have
full information on the feasible set $ X(p_t) $
and that we can compute
$ \argmax{\sp{c}{x} \mid x \in X(p_t)} $
for any candidate objective $ c \in \R^n $ and $ t = 1, \dots, T $.
We present an online-learning framework
that allows us to learn a strategy $ (c_1, \ldots, c_T) $
of subsequent objective function choices
with the following guarantee:
if we optimize according to the surrogate objective function~$ c_t $
instead of the actual unknown objective function $ \ctrue $
in response to parameter realization~$ p_t $,
we obtain a sequence of optimal decisions (w.r.t.\ to each~$ c_t $)
given by
\begin{equation*}
	\sol{x}_t = \argmax{\sp{c_t}{x} \mid x \in X(p_t)}
\end{equation*}
that are essentially as good
as the decisions~$ x_t $ taken by the expert on average.
To this end, we interpret the observations
of parameters and expert solutions
as revealed over multiple rounds
such that in each round~$t$
we are shown the parameters~$ p_t $ first,
then take our optimal decision~$ \sol{x}_t $
according to our objective function~$ c_t $,
afterwards we are shown the solution~$ x_t $ chosen by the expert,
and finally we are allowed to update $ c_t $ for the next round.
For this setup, we will be able to show that
our online-learning algorithms are able to produce a strategy
of objective functions and corresponding solutions
for which both the deviations in true cost
$ \sp{\ctrue}{(x_t - \sol{x}_t)} \geq 0 $
as well as the deviations in surrogate cost
$ \sp{c_t}{(\sol{x}_t - x_t)} \geq 0 $
can be made arbitrarily small on average
under very general assumptions.
This way, our approach can be understood
as an online, multi-sample version of inverse optimization,
where the typical task is to find an objective function
which is close to a known target obective
and makes a given solution optimal.

Our results show that linear objective functions
over general feasible sets
can be learned from relatively few observations
of historical optimal parameter-solution pairs.
We will derive various extensions of our scheme,
such as approximately learning non-linear objective functions
and learning from suboptimal decisions.
We will also, briefly, discuss the case
where the objective~$ \ctrue $ is known,
but some linear constraints are unknown in this paper.

\subsection*{Literature Overview}

The idea of learning or inferring parts
of an optimization model from data
is a reasonably well-studied problem
under many different assumptions and applications
and has gained significant attention
in the optimization community over the last few years,
as discussed for example in \citet{denHertogPostek2016},
\citet{Lodi2016} or \citet{Simchi-Levi2014}.
These works argue that there would be significant benefits
in combining traditional optimization models
with data-derived components.
Most approaches in the literature focus on deriving
the objective function of an expert decision-maker in a static fashion,
based on past observations of input data
and the decisions he took in each instance.
In almost all cases, the objective functions are learned
by considering the KKT-conditions or the dual
of the (parameterized) optimization problem,
and as such convexity for both the feasible region
and the objective function is inherently assumed;
examples include \citet{KeshavarzWangBoyd2011, Li2016, ThaiBayen2018}.
Especially, in \citet{ESHK2018, TrouttTadisinaSohnBrandyberry2005, TrouttPangHou2006}
the authors study the minimization of the same regret function
that we consider in the present paper.
However, as their exact solution approaches are based on duality,
they do not extend to the integer case like the ideas presented here.
Sometimes also distributional assumptions
regarding the observations are made.
In \citet{troutt2011some},
the authors study experimental setups for learning
objectives under various stochastic assumptions,
focussing on maximum likelihood estimation,
which is generally the case for their line of work;
we make no such assumptions.
Especially, the observations can be chosen by a fully-adaptive adversary
in our framework.

Closely related to learning optimization models from observed data is
the subject of inverse optimization.
Here the goal is to find an objective function
that renders the observed solutions optimal with
respect to the concurrently observed parameter realizations.
Approaches in this field mostly stem from convex optimization,
and they are used for inverse optimal control
(\citet{IyengarKang2005,PancheaRamdani2015,MolloyTsaiFordPerez2016}),
inverse combinatorial optimization
(\citet{BurtonPulleyblankToint1997,BurtonToint1994,BurtonToint1992,SokkalingamAhujaOrlin1999,AhujaOrlin2000}),
integer inverse optimization (\citet{Schaefer2009})
and inverse optimization in the presence of noisy data,
such as observed decisions that were suboptimal
(\citet{AswaniShenSiddiq2018,ChanLeeTerekhov2019}).
In this paper, we actually do inverse optimization
by means of online-learning algorithms.

We remark that our work is also related to inverse reinforcement learning
and apprenticeship learning, where the reward function
is the target to be learned.
However, in this case the underlying problem is modelled as a
Markov decision process (MDP);
see, for example, the results in \citet{SyedSchapire2007} and
\citet{ratia2012performance}.
Typically, the obtained guarantees are therefore of a different form.

Applications of such approaches are as diverse as energy systems
(\citet{RatliffDongOhlssonSastry2014,KonstantakopoulosRatliffJinSpanosSastry2016}),
robot motion
(\citet{PapadopoulosBacettaFerretti2016,YangZeilingerTomlin2014}),
medicine (\citet{SayreRuan2014})
and revenue management
(\citet{KallusUdell2015,QiangBayati2016,ChenOwenPixtonSimchi-Levi2015,KallusUdell2016,BertsimasKallus2016});
including situations where the observed decisions
were not necessarily optimal (\citet{NielsenJensen2004}).


\subsection*{Contribution}
To the best of the authors' knowledge,
this is the first attempt
to learn the objective function of an optimization model 
from data using an online-learning approach.  
All previous approaches heavily rely on duality
and thus require convexity assumptions
both for the feasible region as well as the objectives.
As such, they cannot deal with more complex,
possibly non-convex decision domains.
This in particular includes the important case
of integer-valued decisions
(such as yes/no-decisions or, more generally, mixed-integer programming)
and also many other non-convex setups
(several of which admit efficient linear optimization algorithms).
Previously, this was only possible
when the structure of the feasible set
could be beneficially exploited.
In contrast, our approach does not make any such assumptions
and only requires access to a \emph{linear optimization oracle}
(in short: \emph{LP oracle}) for the feasible region~$X$.
Such an oracle is defined as a method which,
given a vector $ c \in \R^n $,
returns $ \argmax{\sp{c}{x} \mid x \in X} $.
Thus, we do not explicitly analyse the KKT-system or the dual program
(in the case of linear programs (LPs); see Remark~\ref{rem:polyApproach}).
Actually, one might consider our approach
as an algorithmic analogue of the KKT-system (or dual program)
in the convex case.

Classical inverse optimization
only treats the offline case,
where all observations are present in advance.
Our new approach is based on online learning algorithms,
especially the \emph{Multiplicative Weights Update (MWU)} algorithm
developed in \citet{LittlestoneWarmuth1994}, \citet{Vovk1990} as well as
\citet{FreundSchapire1997} (see \citet{AroraHazanKale2012,ocoBook}
for a comprehensive introduction; see also \citet{audibert2013regret}).
A similar algorithm was used in \citet{PlotkinShmoysTardos1995}
for solving fractional packing and covering problems.
To generalize the applicability of our approach,
we also derive a second algorithm based on
\emph{Online Gradient Descent (OGD)}
due to Zinkevich (see \citet{zinkevich2003online})
and study approaches based on follow-the-leader schemes
(see \citet{ocoBook} for an introduction).
We remark that our approaches are equivalent to \emph{Mirror Descent}
\citet{nemirovski1979efficient}
for suitable domains and loss function,
however, the interpretation here is different.
We finally point out that our feedback is stronger than bandit feedback.
This requirement is not unexpected
as the costs chosen by the \qm{adversary} depend on
our decision; as such the bandit model (see, for example,
\citet{DaniHayesKakade2008, Abbasi-YadkoriPalSzepevari2011})
does not readily apply.



In an indicative set of computational experiments,
we demonstrate the effectiveness
and wide applicability of our algorithmic approach.
To this end, we investigate its use for learning the objective functions
of several combinatorial optimization problems
that are relevant in practice
and are able to show, among other things,
that the learned objective
generalizes to previously unseen data samples very well.

The present paper is the full version of a contribution
to the International Conference on Machine Learning (ICML) 2017,
see \citet{pmlr-v70-barmann17a}.
The follow-up works \citet{WardMasterBambos2019}
and \cite{DongChenZeng2018}
by different authors
which have been published in the meantime
study a specialization of our approach
to projective cone scheduling
and implicit update rules respectively.
The authors of \cite{SessaBogunovicKamgarpourKrause2020}
treat robust optimization
under an uncertain objective function
where a bandit approach adjusts the level of conservativeness.
In the appendix,
we give further analyses
on the results presented here.


\section{Problem Setting}
\label{Sec:Problem}
We consider the following family
of optimization problems~$ (\textrm{OPT}(p))_p $,
which depend on parameters $ p \in P \subseteq \R^k $
for some $ k \in \N $:
\begin{equation*}
	\max\set{\sp{\ctrue}{x} \mid x\in X(p)},
\end{equation*}
where $ \ctrue \in \R^n $ is the objective function
and $ X(p) \subseteq \R^n $ is the feasible region,
which depends on the parameters~$p$.
Of particular interest to us will be feasible regions
that arise as polyhedra defined by linear constraints
and their intersections with integer lattices,
\ie the cases of LPs and MIPs:
\begin{equation*}
	X(p) = \set{x \in \Z^{n - l} \times \R^l \mid A(p) x \leq b(p)}
\end{equation*}
with $ A(p) \in \R^{m \times n} $ and $ b(p) \in \R^m $.
However, our approach can also readily be applied
in the case of more complex feasible regions, such as
mixed-integer sets bounded by a convex function
$ G\colon P \times \Z^{n - l} \times \R^l \to \R $:
\begin{equation*}
	X(p) = \set{x \in \Z^{n - l} \times \R^l \mid G(p, x) \leq 0},
\end{equation*}
or even more general settings.
In fact, for any possible choice of model
for the sets of feasible decisions,
we only require the availability of a linear optimization oracle,
\ie an algorithm which is able to determine $ \argmax{\sp{c}{x} \mid x \in X(p)} $ for any $ c \in \R^n $ and $ p \in P $.
We call a decision $ x \in \R^n $ \emph{optimal} for $p$
if it is an optimal solution to \OPTpp.

We assume that Problem~\OPTp
models a parameterized optimization
problem which has to be solved repeatedly 
for various input parameter realizations~$p$.
Our task is to learn the fixed objective function~$ \ctrue $
from given observations of the parameters~$p$
and a corresponding optimal solution~$x$ to \OPTpp.
To this end, we further assume
that we are given a series of observations~$ ((p_t, x_t))_t $
of parameter realizations $ p_t \in P $
together with an optimal solution~$ x_t $ to $ \textrm{OPT}(p_t) $
computed by the expert for $ t = 1, \ldots, T$;
these observations are revealed over time in an online fashion:
in round~$t$, we obtain a parameter setting~$ p_t $
and compute an optimal solution~$ \sol{x}_t \in X(p_t) $
with respect to an objective function $ c_t $
based on what we have learned about $ \ctrue $ so far.
Then we are shown the solution~$ x_t $
the expert with knowledge of $ \ctrue $
would have taken and can use this information
to update our inferred objective function for the next round.
In the end, we would like to be able
to use our inferred objective function to take decisions
that are essentially as good as those chosen by the expert
in an appropriate aggregation measure such as, for example,
\qm{on average} or \qm{with high probability}.
The quality of the inferred objective is measured
in terms of cost deviation between our solutions $ \sol{x}_t $
and the solutions $ x_t $ obtained by the expert --
details of which will be given in the next section.

To fix some useful notations,
let $ v(i) $ denote the $i$-th component of a vector~$v$ throughout,
and let $ [n] \coloneqq \set{1, \dots, n} $ for any natural number~$n$.
Furthermore, let $ \mathbbm{1}^n \coloneqq \tp{(1, \ldots, 1)} $ denote
the all-ones vector in $ \R^n$.
Finally, we need a suitable measure for the diameter of a given set.
\begin{definition}
The $ \ell_p $-diameter of a set $ S \subseteq \R^n $,
denoted by $ \diam_p(S) $, is the largest distance
between any two points $ x_1, x_2 \in S $,
measured in the $p$-norm, $ 1 \leq p \leq \infty $, \ie
$ \diam_p(S) \coloneqq \max_{x_1, x_2 \in S} \norm{x_1 - x_2}_p $.
\end{definition}
As a technical assumption,
we further demand that $ \ctrue \in F $
for some convex, compact and non-empty subset $ F \subseteq \R^n $
which is known beforehand.
This is no actual restriction,
as $F$ could be chosen to be any ball
according to some $p$-norm, $ 1 \leq p \leq \infty $, for example.
In particular, this ensures that we do not have to deal with issues
that arise when rescaling our objective.


\section{Learning Objectives}
\label{Sec:Objective}
One approach to find a candidate
for the true objective function~$ \ctrue $
is to solve the following optimization problem:
\begin{equation}
	\min_{c \in F}
		\sum_{t \in [T]} \left(\left(\max_{x \in X(p_t)} \sp{c}{x}\right) - \sp{c}{x_t}\right),
\label{Eq:OPT}
\end{equation}
where $ \norm{\cdot}\colon \R^n \to \R_+ $
is an arbitrary norm on $ \R^n $
and $ x_t \in X(p) $ is the \emph{optimal} decision
taken by the expert in round $t$.
The true objective function~$ \ctrue $
is an optimal solution to Problem~\eqref{Eq:OPT}
with objective value~$0$.
This is because any solution~$ \arb{c} \in F $
is feasible and produces non-negative summands
$ \left(\max_{x \in X(p_t)} \sp{\arb{c}}{x}\right) - \sp{\arb{c}}{x_t} $
for $ t \in [T] $, as we assume $ x_t \in X(p_t) $
to be optimal for $ p_t $
with respect to $ \ctrue $.

When solving Problem~\eqref{Eq:OPT},
we are interested in an objective function vector~$ c \in F $
that delivers a consistent explanation
for why the expert chose $ x_t $ as his response
to the parameters~$ p_t $ in round~$ t \in [T] $.
We call an objective function~$ c \in F $ \emph{consistent}
with the observations $ (p_t, x_t) $, $ t \in [T] $,
if it is optimal for Problem~\eqref{Eq:OPT}.
The following result shows that Problem~\eqref{Eq:OPT} is convex
and allows for an explicit statement of subgradients:
\begin{proposition}
	For each $ t \in [T] $, the function
	$ f_t(c) \coloneqq \max_{x \in X(p_t)} \sp{c}{x} - \sp{c}{x_t} $
	is convex.
	Furthermore, for any $ \arb{c} \in F $
	and $ \sol{x}_t \coloneqq \argmax{\sp{\arb{c}}{x} \mid x \in X(p_t)} $
	the vector $ (\sol{x}_t - x_t) $
	is a subgradient in~$ \arb{c} $.
\end{proposition}
\begin{proof}
	Consider $ t \in [T] $, $ c_1, c_2 \in F $ and $ \lambda \in [0, 1] $.
	Then we find
	\begin{align*}
		f_t(\lambda c_1 + (1 - \lambda) c_2)
			&= \max_{x \in X(p_t)} \sp{(\lambda c_1 + (1 - \lambda) c_2)}{x} - \sp{(\lambda c_1 + (1 - \lambda) c_2)}{x_t}\\
			&\leq \lambda \max_{x \in X(p_t)} \sp{c_1}{x} - \sp{c_1}{x}
				+ (1 - \lambda) \max_{x \in X(p_t)} \sp{c_2}{x} - \sp{c_2}{x}\\
			&= \lambda f_t(c_1) + (1 - \lambda) f_t(c_2).
	\end{align*}
	We also see that for any $ c \in F $ and a fixed $ \arb{c} \in F $,
	we have
	\begin{align*}
		f_t(c) - f_t(\arb{c})
			&= \left(\max_{x \in X(p_t)} \sp{c}{x}\right) - \sp{c}{x_t} 
				-\left(\max_{x \in X(p_t)} \sp{\arb{c}}{x}\right) + \sp{\arb{c}}{x_t}\\
			&\geq \sp{c}{\sol{x}_t} - \sp{c}{x_t} 
				-\sp{\arb{c}}{\sol{x}_t} + \sp{\arb{c}}{x_t}
			= \sp{(\sol{x}_t - x_t)}{(c - \arb{c})}.
	\end{align*}
\end{proof}
As the objective function of Problem~\eqref{Eq:OPT}
is the sum of the individual convex functions $ f_t(c) $,
it is convex as well,
and a subgradient in any $ \arb{c} \in F $
is given by $ \sum_{t \in [T]} (\sol{x}_t - x_t) $.
Thus, we can determine the true objective function $ \ctrue $
via a subgradient method in which we compute the subgradients
by making calls to the linear optimization oracle
in order to solve the $T$~instances of the maximization subproblem
\begin{equation}
	\max\set{\sp{c}{x} \mid x\in X(p_t)}.
	\label{Eq:SUB}
\end{equation}
The \emph{offline case}, where all feasible sets
and the corresponding expert solutions are known at once,
is therefore solvable at a sublinear rate of convergence.
\begin{remark}
\label{rem:polyApproach}
Note that in the case of polyhedral feasible regions,
\ie observations $ p_t = (A_t, b_t) \in \R^{m \times n} \times \R^m $
and $ X(p_t) = \set{x \in \R^n \mid A_t x \leq b_t} $
for $ t \in [T] $,
as well as a polyhedral region $ F = \set{c \in \R^n \mid Bc \leq d} $,
Problem~\eqref{Eq:OPT} can be reformulated as a linear program
by dualizing the $T$~instances of Subproblem~\eqref{Eq:SUB}.
This yields
\begin{equation}
\min \left\{ \sum_{t \in [T]} (\sp{b_t}{y_t} - \sp{c}{x_t}) \  \colon \ \tp{A_t} y_t = c, y_t  \geq 0, (\forall t \in [T]), Bc \leq d \right\},
\label{Eq:LP_Reformulation}
\end{equation}
where the $ y_t $ are the corresponding dual variables
and the $ x_t $ are the observed decisions from the expert
(\ie the latter are part of the \emph{input data}).
This problem asks for a primal objective function vector~$c$
that minimizes the total duality gap
summed over all primal-dual pairs $ (x_t, y_t) $
while all $ y_t $'s shall be dual feasible, 
which makes the $ x_t $'s the respective primal optimal solutions.
Thus, Problem~\eqref{Eq:OPT} can be seen as a direct generalization
of the linear primal-dual optimization problem.
In fact, our approach also covers non-convex cases,
\eg mixed-integer linear programs.
\end{remark}

For our main contribution in this paper,
the consideration of the \emph{online case},
we replace Problem~\eqref{Eq:OPT} by a game
over $T$~rounds between a player who guesses an objective
function~$ c_t $ in round~$ t \in [T] $ and a player who knows the
true objective function~$ \ctrue $ and chooses the observations
$ (p_t, x_t) $ in a potentially adversarial way.
The payoff of the latter player in each round~$t$
is equal to $ \sp{c_t}{(\sol{x}_t - x_t)} \geq 0 $,
\ie the difference in cost between the player's solution $ \sol{x}_t $
and the expert's solution $ x_t $ as given by our guessed
objective function~$c_t$.

To this end, we will design online-learning algorithms that,
rather than finding an optimal objective~$c$, 
find a \emph{strategy}
of objective functions~$ (c_1, c_2, \ldots, c_T) $
to play in each round
whose error in solution quality
as compared to the true objective function
is as small as possible.
Our aim will then be to give a quality guarantee for this strategy
in terms of the number of observations.
The framework we present here
will provide even stronger guarantees in some cases,
such as the one described in Section~\ref{sec:stable-case},
showing that we can replicate
the decision-making behaviour of the expert.


From a meta perspective, our approach works as outlined in Algorithm~\ref{Alg:OOFL}.
\begin{algorithm}[htb]
	\caption{Online Objective Function Learning}
	\label{Alg:OOFL}
	\begin{algorithmic}[1]
		\INPUT Observations $ (p_t, x_t)$ for $ t \in [T] $
		\OUTPUT Sequence of objectives $ (c_1, c_2, \ldots, c_T) $
		\STATE Choose initial objective function $ c_1 \in F $
		\FOR{$ t = 1, \ldots, T $}
			\STATE Observe parameters $ p_t $
			\STATE Compute $ \sol{x} \in X(p_t) $
				as an optimum to Subproblem~\eqref{Eq:SUB}
				with objective~$ c_t $
			\STATE Observe expert solution $ x_t \in X(p_t) $
			\STATE Compute an updated objective $ c_{t + 1} \in F $
		\ENDFOR
		\STATE {\bfseries return} $ (c_1, c_2, \ldots, c_T) $.
\end{algorithmic}
\end{algorithm}
It chooses an arbitrary objective in the first round,
as there is no better indication of what to do at this point.
Then, in each round $ t \in [T] $, 
it computes an optimal solution over $ X(p_t) $
with respect to the current guess of objective function~$ c_t $.
Upon the following observation of the expert's solution,
it updates its guess of objective function to use it in the next round.

Clearly, the accumulated objective value
of a strategy $ (c_1, \ldots, c_T) $
over $T$~rounds is given by $ \sum_{t \in [T]} \sp{c_t}{\sol{x}_t} $,
while that of $ \ctrue $
would be~$ \sum_{t \in [T]} \sp{\ctrue}{x_t} $.
Via the proposed scheme, it would be overly ambitious to demand
$ \lim_{T \to \infty} c_T = \ctrue $,
or even $ \lim_{T \to \infty} (\sp{c_T}{\sol{x}_T} - \sp{\ctrue}{x_T}) = 0 $,
as the following example shows.
\begin{example}
\label{Ex:AlternativeObjective}
Consider the case $ \ctrue = \tp{(0, 1)} $
and $ X(p_t) \subset \set{x \in \R^2 \mid x(1) = 0} $
compact for $ t \in [T] $.
If the first player chooses $ c_t = \tp{(1 - \varepsilon, \varepsilon)} $
for any $ 0 < \varepsilon \leq 1 $
as his objective function guess in each round $ t \in [T] $,
he will obtain optimal solutions~$ \sol{x}_t $
with respect to $ \ctrue $.
However, both the objective functions $ c_t $
and the objective values $ \sp{c_t}{\sol{x}_t} $
will be far off.
Indeed, when taking the $1$-norm, for example, we have
$ \norm{\ctrue - c_t}_1 = \norm{\tp{(\varepsilon - 1, 1 - \varepsilon)}}_1 \to 2 $
for $ \varepsilon \to 0 $.
And if $ X(p_t) = \set{\tp{(0, 1)}} $ for all $ t \in [T] $,
we additionally have
$ \sp{\ctrue}{\tp{(0, 1)}} = 1 $,
but $ \sp{c_t}{\tp{(0, 1)}} = \varepsilon \to 0 $
for $ \varepsilon \to 0 $.
\end{example}
Altogether, we cannot expect to approximate
the true objective function $ \ctrue $
or the true optimal values $ \sp{\ctrue}{x_t} $ in general.
Neither can we expect to approximate the solutions $ x_t $,
because even if we have the correct objective function $ c_t = \ctrue $
in each round, the optima do not not necessarily have to be unique.

As a more appropriate measure of quality,
we will show that our algorithms based on online learning
produce strategies $ (c_1, \ldots, c_T) $ with
\begin{equation}
	\label{Eq:TotalError}
	\lim_{T \to \infty}
		\frac1T \sum_{t \in [T]} \sp{(c_t - \ctrue)}{(\sol{x}_t - x_t)} = 0,
\end{equation}
of which we will see that it directly implies both
\begin{equation}
	\label{Eq:ObjectiveSolutionError}
	\lim_{T \to \infty}
		\frac1T \sum_{t \in [T]} \sp{c_t}{(\sol{x}_t - x_t)} = 0
	\quad
	\text{and}
	\quad
	\lim_{T \to \infty}
		\frac1T \sum_{t \in [T]} \sp{\ctrue}{(x_t - \sol{x}_t)} = 0,
\end{equation}
with non-negative summands for all rounds~$t$ in all three expressions.
The \emph{objective error} given by $ \sum_{t \in [T]} \sp{c_t}{(\sol{x}_t - x_t)} $
is the objective function of Problem~\eqref{Eq:OPT}
when relaxing the requirement
to play the same objective function in each round
and instead passing to a strategy of objective functions.
Equation~\eqref{Eq:ObjectiveSolutionError} states
that the average objective error
over all observations converges to zero
with the number of observations going to infinity.
The same holds for the average \emph{solution error}
$ \sum_{t \in [T]} \sp{\ctrue}{(x_t - \sol{x}_t)} $,
which is the cumulative suboptimality of the solutions~$ \sol{x}_t $
compared to the optimal solutions~$ x_t $
with respect to the true objective function.
This means it is possible to take decisions~$ \sol{x}_t $
which are essentially as good as the decisions $ x_t $ of the expert
with respect to $ \ctrue $ over the long run.
We will call the sum of these two errors the \emph{total error}.

This total error is derived from the notion of \emph{regret},
which is commonly used in online learning to characterize
the quality of a learning algorithm:
given an algorithm~$A$ which plays solutions $ c_t \in F $
for some decision set~$F$
in response to loss functions $ f_t\colon F \to R $
observed from an adversary over rounds $ t \in [T] $,
it is given by
$ R(A) \coloneqq \sum_{t \in [T]} f_t(c_t) - \min_{c \in F} f_t(c) $.
Minimizing the regret of a sequence of decisions
thus aims to find a strategy that perfoms at least as good
as the \emph{best fixed decision in hindsight},
\ie the best static solution that can be played
with full advance-knowledge of the loss functions the adversary will play.
See e.g.\ \citet{ocoBook} for a broad introduction
to regret minimization in online learning.

In our approach, a learning algorithms chooses an objective~$ c_t $
from the set of possible objective functions~$F$
in each round $ t \in [T] $, which is evalutated
using the loss function $ f(c_t) \coloneqq \sp{(\sol{x}_t - x_t)}{c_t} $.
The average regret against $ \ctrue $
is then given by $ 1/T \sum_{t \in [T]} \sp{(c_t - \ctrue)}{(\sol{x}_t - x_t)} $,
and Equation~\eqref{Eq:TotalError} states that it tends to zero
as the number of observations increases.
Note that~$ \ctrue $ is not necessarily
the best fixed objective in hindsight --
the latter would be given by a standard unit vector $ e_i $,
where $ i \in \argmin{i \in {1, \ldots, n \mid \sum_{t \in [T]} (\sol{x}_t(i) - x_t(i))}} $,
which is rather meaningless here.

In the following, we derive online-learning algorithms
for which Equation~\eqref{Eq:TotalError} holds provably.
Furthermore, we study their convergence properties
under different assumptions as well as possible applications
and draw connections between the online and the offline case.

\subsection{An Algorithm based on Multiplicative Weights Updates}
\label{Sec:Objective.MWU}

A classical algorithm in online learning
is the multiplicative weights update (MWU) algorithm,
which solves the following problem:
given a set of $n$~decisions,
a player is required to choose one of these decisions
in each round $ t \in [T] $.
Each time, after the player has chosen a decision, 
an adversary reveals the costs $ m_t \in [-1, 1]^n $ 
of all decisions in the current round.
The objective of the player
is to minimize the overall cost $ \sum_{t \in [T]} \sp{m_t}{w_t} $ 
over the time horizon~$T$.
The MWU algorithm solves this problem by maintaining
weights~$ w_t \in \R_+^n $ which are updated from round to round,
starting with the initial weights $ w_1 = \mathbbm{1}^n $.
These weights are used to derive
a probability distribution~$ p_t \coloneqq w_t / \norm{p_t} $.
In round~$t$, the player samples a decision~$i$ from $ \set{1, \ldots, n} $
according to $ p_t $.
Upon observation of the costs~$ m_t $,
the player updates his weights
according to the update rule $w_{t + 1} = w_t - \eta (w_t \odot m_t) $.
Here, $ 0 < \eta < \frac12 $ is a suitably chosen step size,
in online learning also called \emph{learning rate},
and $ a \odot b \coloneqq (a_1 \cdot b_1, \ldots, a_n \cdot b_n) $
denotes the componentwise multiplication
of two vectors $ a, b \in \R^n $.
The expected cost of the player in round~$t$
is then given by $ \sp{m_t}{p_t} $,
and the total expected cost
is given by $ \sum_{t \in [T]} \sp{m_t}{p_t} $.
MWU attains the following regret bound against
any fixed distribution:
\begin{lemma}[{\citet[Corollary 2.2]{AroraHazanKale2012}}]
\label{Lem:Guarantee_MWU}
The MWU algorithm guarantees that after $T$~rounds,
for any distribution~$p$ on the decisions, we have
\begin{equation*}
	\sum_{t \in [T]} \sp{m_t}{p_t}
		\leq \sum_{t \in [T]} \sp{(m_t + \eta \abs{m_t})}{p} + \frac{\ln n}{\eta},
\end{equation*}
where the $ \abs{m_t} $ is to be understood componentwise.
\end{lemma}

We will now reinterpret the distributions $ p_t $
as well as the cost vectors~$ m_t $ in MWU in a way
that will allow us to learn an objective function
from observed solutions.
Namely, we will identify the distributions $ p_t $
with the objective functions~$ c_t $ in the strategy of the player
and the distribution~$p$ on the right-hand side
with the actual objective function~$ \ctrue $.
The (normalized) difference between the optimal solution~$ \bar{x}_t $
computed by the player and the optimal solution~$ x_t $ of the expert
will then act as the cost vector~$ m_t $.
Naturally, this limits us
to $ F = \Delta_n \coloneqq \set{c \in \R_+^n \mid \norm{c}_1 = 1} $,
\ie the objective functions have to lie in the positive orthant
(while normalization is without loss of generality).
However, whenever this restriction applies,
we obtain a very lightweight method for learning the objective function
of an optimization problem.
In Section~\ref{Sec:Objective.OGD}, we will present an algorithm
which works without this assumption on $F$.

Our application of MWU to learning the objective function
of an optimization problem
proceeds as outlined in Algorithm~\ref{Alg:MWU}.
\begin{algorithm}[t]
	\caption{Objective Function Learning via Multiplicative Weights Updates}
	\label{Alg:MWU}
	\begin{algorithmic}[1]
		\INPUT Observations $ (p_t, x_t) $ for $ t = 1, \ldots, T $
		\OUTPUT Sequence of objectives $ (c_1, c_2, \ldots, c_T) $
		\STATE $ \eta \leftarrow \sqrt{\frac{\ln n}{T}} $
                \COMMENT{Set learning rate}
			\STATE $ w_1 \leftarrow \mathbbm{1}^n $
				\COMMENT{Initialize weights}
		\FOR{$ t = 1, \ldots, T $}
			\STATE $ c_t \leftarrow \frac{w_t}{\norm{w_t}_1} $
				\COMMENT{Normalize weights}
			\STATE Observe parameters $ p_t $
			\STATE $ \sol{x}_t \leftarrow
				\argmax{\sp{c_t}{x} \mid X(p_t)} $
				\COMMENT{Solve Subproblem~\eqref{Eq:SUB}}
			\STATE Observe expert solution $ x_t $
			\IF{$ \sol{x}_t = x_t $}
				\STATE $ y_t \leftarrow 0 $
			\ELSE
				\STATE $ y_t \leftarrow
					\frac{\sol{x}_t - x_t}
						{\norm{\sol{x}_t - x_t}_\infty} $ \label{Alg:MWU:CostFunction}
			\ENDIF
			\STATE $ w_{t + 1} \leftarrow
				w_t - \eta (w_t \odot y_t) $ \COMMENT{Update weights} \label{Alg:MWU:Update}
		\ENDFOR
		\STATE {\bfseries return} $ (c_1, c_2, \ldots, c_T) $.
\end{algorithmic}
\end{algorithm}
For the series of objective functions $ (c_1, c_2, \ldots, c_T) $
it returns, we can establish the following guarantee:
\begin{theorem}
\label{Thm:Convergence}
Let $ K \geq 0 $ with $ \diam_\infty X(p_t) \leq K $
for all $ t \in [T] $.
Then we have
\begin{equation*}
	0 \leq \frac1T \sum_{t \in [T]} \sp{(c_t - \ctrue)}{(\sol{x}_t - x_t)}
	\leq 2 K \sqrt{\frac{\ln n}{T}},
\end{equation*}
and in particular it also holds:
\begin{enumerate}
	\item $ 0 \leq \frac1T \sum_{t \in [T]} \sp{c_t}{(\sol{x}_t - x_t)}
		\leq 2 K \sqrt{\frac{\ln n}{T}} $,
	\item $ 0 \leq \frac1T \sum_{t \in [T]} \sp{\ctrue}{(x_t - \sol{x}_t)}
		\leq 2 K \sqrt{\frac{\ln n}{T}} $.
	\end{enumerate}
\end{theorem}
\begin{proof}
According to the standard performance guarantee of MWU
as stated in Lemma~\ref{Lem:Guarantee_MWU},
Algorithm~\ref{Alg:MWU} attains the following
bound on the average total cost
of the secuence $ (c_1, c_2, \ldots, c_T) $
compared to $ \ctrue $ with respect to the cost vectors $ y_t $:
\begin{equation*}
	\frac1T \sum_{t \in [T]} \sp{c_t}{y_t}
		\leq \frac1T \sum_{t \in [T]} \sp{\ctrue}{(y_t + \eta \abs{y_t})} + \frac{\ln n}{\eta T}.
\end{equation*}
Using that each entry of $ \abs{y_t} $ is at most~$1$
and that $ \ctrue \in F $, we can conclude
\begin{equation*}
	\frac1T \sum_{t \in [T]} \sp{c_t}{y_t} - \frac1T \sum_{t \in [T]} \sp{\ctrue}{y_t}
		\leq \eta + \frac{\ln n}{\eta T}.
\end{equation*}
The right-hand side attains its minimum for
$ \eta = \sqrt{\frac{\ln n}{T}} $,
which yields the bound
\begin{equation*}
	\frac1T \sum_{t \in [T]} \sp{c_t}{y_t} - \frac1T \sum_{t \in [T]} \sp{\ctrue}{y_t}
		\leq 2\sqrt{\frac{\ln n}{T}}.
\end{equation*}
Substituting back for the $ y_t $'s and using
\begin{equation*}
	\max_{t \in [T]} \norm{\sol{x}_t - x_t}_\infty
		\leq \max_{t \in [T]} \diam\nolimits_\infty(X(p_t)) \leq K,
\end{equation*}
we obtain
\begin{equation*}
	\frac1T \sum_{t \in [T]} \sp{c_t}{(\sol{x}_t - x_t)} + \frac1T \sum_{t \in [T]} \sp{\ctrue}{(x_t - \sol{x}_t)}
		\leq 2 K \sqrt{\frac{\ln n}{T}}.
		\label{Eq:Bound1}
\end{equation*}
Observe that for each summand $ t \in [T] $
we have $ \sp{c_t}{(\sol{x}_t - x_t)} \geq 0$
as $ \sol{x}_t, x_t \in X(p_t) $
and $ \sol{x}_t $ is the maximum over this set
with respect to $ c_t $.
With a similar argument, we see that $ \sp{\ctrue}{(x_t - \sol{x}_t)} \geq 0$
for all $ t \in [T] $.
Thus, we have
\begin{equation*}
	0 \leq \frac1T \sum_{t \in [T]} \sp{(c_t - \ctrue)}{(\sol{x}_t - x_t)}
		\leq 2 K \sqrt{\frac{\ln n}{T}},
		\label{Eq:Bound2}
\end{equation*}
and, in consequence, the same for the separate terms.
This establishes the claim.
\end{proof}
Note that by using exponential updates of the form
\begin{equation*}
	w_{t + 1}(i) \leftarrow w_t(i) e^{-\eta y_t(i)}
\end{equation*}
in Line~\ref{Alg:MWU:Update} of the algorithm,
we could attain essentially the same bound,
cf. \cite[Theorem 2.3]{AroraHazanKale2012}.
Secondly, we remark that our choice of the learning rate~$ \eta $
requires the number of rounds~$T$ to be known beforehand;
if this is not the case,
we can use the standard doubling trick (see \citet{Cesa-BianchiLugosi2006})
or use an anytime variant of MWU.

From the above theorem, we can conclude that the average error
over all observations $ (p_t, x_t) $ for $ t \in [T] $
when choosing objective function~$ c_t $ in iteration~$t$
of Algorithm~\ref{Alg:MWU} instead of $ \ctrue $
converges to $0$ with an increasing number of observations~$T$
at a rate of roughly $ \OO(1/\sqrt{T})$:

\begin{corollary}
Let $ K \geq 0 $ with $ \diam_\infty X(p_t) \leq K $
for all $ t \in [T] $.
This implies
\begin{enumerate}
	\item $ \lim_{T \to \infty}
		\frac1T \sum_{t \in [T]} \sp{c_t}{(\sol{x}_t - x_t)} = 0 $
		\quad and
	\item $ \lim_{T \to \infty}
		\frac1T \sum_{t \in [T]} \sp{\ctrue}{(x_t - \sol{x}_t)} = 0 $.
\end{enumerate}
\end{corollary}
In other words, both the average error incurred from replacing the
actual objective function~$ \ctrue $ by the estimation $ c_t $ as well
as the average error in solution quality with respect to $ \ctrue $ tend
to $0$ as $T$ grows.

Moreover, using Markov's inequality
we also obtain the following quantitative bound
on the deviation by more than a given $ \varepsilon > 0 $
from the average cost:
\begin{corollary}
	Let $ \varepsilon > 0 $.
	Then, for any $ 0 < p < 1 $, we have that after
    \begin{equation*}
    	T \geq \ln n \left(\frac{(1 - p) 2 K}{p \varepsilon}\right)^2
    \end{equation*}
    observations at most the fraction~$p$
    of observations $ x_t $ has a cost
    \begin{equation*}
		\sp{\ctrue}{(x_t - \sol{x}_t)} \geq \frac{\varepsilon}{1 - p} \geq 2 K
			\sqrt{\frac{\ln n}{T}} + \varepsilon.
	\end{equation*}
\end{corollary}
\begin{proof}
	Markov's inequality states
	$ \card{\set{x \in X \mid f(x) \geq a}}
		\leq \frac1a \sum_{x \in X} \abs{f(x)} $
	for a finite set $X$, a function~$ f\colon X \to \R $ and $ a > 0 $.
	With $ X = [T] $, $ f(t) = \sp{\ctrue}{(x_t - \sol{x}_t)} $
	for $ t \in [T] $
	as well as $ a = 2 K \sqrt{(\ln n)/T} + \varepsilon $,
	we obtain the desired upper bound on the fraction of high deviations.
	The second part follows from solving
	\begin{equation*}
		1 - \frac{\varepsilon}{2 K \sqrt{\frac{\ln n}{T}} + \varepsilon} \leq p
	\end{equation*}
	for $T$ and plugging in values.
\end{proof}

\begin{remark}
\label{Rem:BasisFunctions}
It is straightforward to extend the result from Theorem~\ref{Thm:Convergence}
to a more general setup, namely the learning of an objective function
which is linearly composed from a set of basis functions.
To this end, we consider the problem
\begin{equation*}
\max \left\{ \sp{\ctrue}{f(x)} \ \colon \ x \in X(p) \right\},
\end{equation*}
where $ \ctrue \in \R_+^n $ with $ \norm{\ctrue}_1 = 1 $,
$ f\colon D \to \R^m $ on $ D \subset \R^n $ compact
and $ X(p) $ parameterized in $ p \in P $ as above.
In order to apply Theorem~\ref{Thm:Convergence} to this case,
the $ \ell_\infty $-diameter of the image of $f$
additionally needs to be finite,
which is naturally the case, for example,
if $f$ is Lipschitz continuous with respect to the maximum norm
with Lipschitz constant~$L$.
Then we can change the cost function in Line~\ref{Alg:MWU:CostFunction}
of Algorithm~\ref{Alg:MWU} to
\begin{equation*}
	y_t = \frac{f(\bar{x}_t) - f(x_t)}{\norm{f(\bar{x}_t) - f(x_t)}_{\infty}}.
\end{equation*}
For $ K = L \cdot \max_{1, \ldots, T} \diam_{\infty}(X(p_t)) $,
this yields a guarantee of 
\begin{equation*}
	0 \leq \frac1T \sum_{t \in [T]} \sp{(c_t - \ctrue)}{(f(\sol{x}_t) - f(x_t))}
	\leq 2 K \sqrt{\frac{\ln n}{T}}.
	\end{equation*}
\end{remark}
We would like to point out that the requirement 
to observe optimal solutions~$ x_t $
to learn the objective function $ \ctrue $ which produced them
can be relaxed in all the above considerations.
Assume that we observe \emph{$ (1 - \varepsilon) $-optimal}
solutions $ \arb{x}_t \in X(p_t) $ instead,
\ie they satisfy $ \sp{\ctrue}{\arb{x}_t} \geq (1 - \varepsilon) \sp{\ctrue}{x_t} $
for all $ t \in [T] $ and some $ \varepsilon \geq 0 $.
In this case, the upper bound
\begin{equation*}
	\frac1T \sum_{t \in [T]} \sp{(c_t - \ctrue)}{(\sol{x}_t - \arb{x}_t)} \leq 2K\sqrt{\frac{\ln n}{T}},
\end{equation*}
which is analoguous to what we derived in Theorem~\ref{Thm:Convergence},
still holds, as it does not depend on the optimality
of the observed solutions.
On the other hand, we have
\begin{equation*}
	\frac1T \sum_{t \in[T]} \sp{(c_t - \ctrue)}{(\sol{x}_t - \arb{x}_t)}
		\geq \frac1T \sum_{t \in [T]} \sp{\ctrue}{(\arb{x}_t - \sol{x}_t)}
		\geq \frac1T \sum_{t \in [T]} \sp{\ctrue}{((1 - \varepsilon) x_t - \sol{x}_t)}
\end{equation*}
due to the optimality of the $ \sol{x}_t $'s
with respect to the $ c_t $'s
and the $ (1 - \varepsilon) $-optimality of the $ \arb{x}_t $'s.
Altogether, this yields
\begin{equation*}
	\frac1T \sum_{t \in [T]} \sp{\ctrue}{\sol{x}_t}
		\geq \frac1T \sum_{t \in [T]} (1 - \varepsilon) \sp{\ctrue}{x_t}
			- 2K\sqrt{\frac{\ln n}{T}},
\end{equation*}
such that in the limit, our solutions $ x_t $
become $ (1 - \varepsilon) $-optimal on average.
Note that a similar result can be obtained
if we assume an additive error in the observed solutions~$ \arb{x}_t $
instead of a multiplicative one.

\subsection{The Stable Case}
\label{sec:stable-case}

While in most applications it is sufficient
to be able to produce solutions via the surrogate objectives
that are essentially equivalent to those for the true objective,
we will show now that under slightly strengthened assumptions
we can obtain significantly stronger guarantees
for the convergence of the solutions:
we will show that in the long run
we learn to emulate the true optimal solutions
provided that the problems have unique solutions
as we will make precise now.

We say that the sequence of feasible regions $ (X(p_t))_t $
is \emph{$ \Delta $-stable for $ \ctrue $} for some $ \Delta > 0$
if for any $ t \in [T] $, $ c \in \R^n $ with $ \norm{c}_1 = 1 $,
$ c \neq \ctrue $ and
$ \sol{x}_t \coloneqq \argmin{\sp{c}{x} \mid x \in X(p_t)} $
so that for $ x_t \neq \sol{x}_t$ we have
\begin{equation*}
	\sp{\ctrue}{(x_t - \sol{x}_t)} \geq \Delta,
\end{equation*}
\ie either the two optimal solutions coincide
or they differ by at least $ \Delta $ with respect to $ \ctrue $.
In particular, optimizing $ \ctrue $ over $ X(p_t) $
leads to a unique optimal solution
for all $ p_t $ with $ t \in [T] $.
This condition --
which is well known as the \emph{sharpness} of a minimizer
in convex optimization --
is, for example, trivially satisfied
for the important case where $ X(p_t) $ with $ t \in [T] $
is a polytope with vertices in $ \set{0, 1} $
and $ \ctrue $ is a rational vector.
In this case, write $ \ctrue = d/\norm{d}_1 $
with $ d \in \Z^n_+ $
and observe that the minimum change in objective value
between any two vertices $ x, y $ of the 0/1-polytope
with $ \sp{\ctrue}{x} \neq \sp{\ctrue}{y} $ is bounded by
$ \lvert \sp{\ctrue}{(x - y)} \rvert \geq 1/\norm{d}_1 $,
so that $ \Delta $-stability
with $ \Delta \coloneqq 1/\norm{d}_1 $
holds in this case.
The same argument works for more general polytopes
via bounding the minimum non-zero change in objective function value
via the encoding length.
We obtain the following simple corollary of
Theorem~\ref{Thm:Convergence}. 
\begin{corollary}
\label{cor:learnDelta}
Let $ K \geq 0 $ with $ \diam_\infty X(p_t) \leq K $
for all $ t \in [T] $,
let $ (X(p_t))_t $ be $ \Delta $-stable
for some $ \Delta > 0 $,
and let $ N_T \coloneqq \set{t \in [T] \mid \sol{x}_t \neq x_t} $.
Then
\begin{equation*}
	\card{N_T} \leq 2 K \sqrt{\frac{T \ln n}{\Delta}}.
\end{equation*}
\end{corollary}
\begin{proof}
Via the guarantee on the solution error
from Theorem~\ref{Thm:Convergence}, we have
\begin{equation*}
	0 \leq \frac1T \sum_{t \in [T]} \sp{\ctrue}{(x_t - \sol{x}_t)}
		= \frac1T\sum_{t \in N_T} \sp{\ctrue}{(x_t - \sol{x}_t)}
		\leq 2 K \sqrt{\frac{\ln n}{T}}.
\end{equation*}
Observe that $ \Delta \leq \sp{\ctrue}{(x_t - \sol{x}_t)} $,
as $ x_t $ was optimal for $ \ctrue $
together with $ \Delta $-stability.
We thus obtain
$ 
	1/T\card{N_T} \Delta \leq 2 K \sqrt{(\ln n)/T}
$, 
which yields
$ 
	\card{N_T} \leq 2 K \sqrt{T (\ln n)/\Delta}
$. 
\end{proof}
From the above corollary, we obtain in particular
that in the $ \Delta $-stable case we have
$ \frac{1}{T} \card{N_T} \leq 2 K \sqrt{(\ln n)/(T\Delta)} $,
\ie the average number of times
that $ \sol{x}_t $ deviates from $ x_t $
tends to $0$ in the long run.
We hasten to stress, however,
that the convergence implied by this bound can potentially be slow
as it is exponential in the actual encoding length of $ \ctrue $;
this is to be expected given the convergence rates
of our algorithm and online-learning algorithms in general.

\subsection{An Algorithm based on Online Gradient Descent}
\label{Sec:Objective.OGD}

The algorithm based on MWU introduced in Section~\ref{Sec:Objective.MWU}
has the limitation that it is only applicable
for learning non-negative objectives.
In addition, it cannot make use of any prior knowledge
about the structure of $ \ctrue $
other than coming from the positive orthant.
To lift these limitations,
we will extend our approach
using online gradient descent (OGD)
which is an online-learning algorithm
applicable to the following game over $T$~rounds:
in each round $ t \in [T] $,
the player chooses a solution $ c_t $
from a convex, compact and non-empty feasible set $ F \subset \R^n $.
Then the adversary reveals 
a convex objective function $ f_t\colon \R^n \to \R $,
and the player incurs a cost of~$ f_t(x_t) $.
OGD proceeds by choosing an arbitrary $ c_1 \in F $
in the first round and updates this choice after observing~$ f_t $ via
\begin{equation*}
	c_{t + 1} = P(c_t - \eta_t \nabla f_t(x_t)),
\end{equation*}
where $P$ is the projection onto the set $F$
and $ \eta_t $ is the learning rate.
With the abbreviations $ D \coloneqq \diam_2(F) $
and $ G \coloneqq \sup_{c \in F, t \in [T]} \norm{\nabla f_t(x)}_2^2 $,
the regret of the player can then be bounded as follows.
\begin{lemma}[{\citet[Theorem~1]{zinkevich2003online}}]
	\label{Lem:Guarantee_OGD}
	For $ \eta_t \coloneqq 1/\sqrt{t} $, $ t \in [T] $, we have
	\begin{equation*}
		\sum_{t \in [T]} f_t(c_t) - \min_{c \in F} \sum_{t \in [T]} f_t(x)
			\leq \frac{D^2 \sqrt{T}}{2}
				+ \left(\sqrt{T} - \frac12\right) G^2.
	\end{equation*}
\end{lemma}
Concerning the choice of learning rate,
there are a couple of things to note.
Firstly, the learning rate $ \eta_t = 1/\sqrt{t} $ in round~$t$
does not depend on the total number of rounds~$T$ of the game.
This means that the resulting version of OGD
works without prior knowledge of~$T$.
It is even possible to improve slightly on the above result:
by choosing the learning rate $ \eta_t = D/(G\sqrt{t}) $ in round $t$,
the regret bound after $T$~rounds becomes $ (3/2) DG \sqrt{T} $,
thus exhibiting smaller constant factors
(see, for example, \citet[Theorem~3.1]{ocoBook}).
In the case of prior knowledge of~$T$,
it is possible to choose
the constant learning rate $ \eta_t = D/(G\sqrt{T}) $
in each round~$t$, which in our computational experiments
leads to a smoother convergence especially in the first iterations
and again marginally improves the regret bound;
it is then possible to bound it by $ DG \sqrt{T} $
(\cf the proof of \citet[Theorem~1]{zinkevich2003online}).

As before, we can reinterpret the underlying game of OGD
in the context of learning objective functions:
the player now plays linear objective functions~$ c_t \in F $
and the adversary answers
with the linear loss function $ f(c_t) = \sp{(\sol{x}_t - x_t)}{c_t} $.
This leads to the learning scheme described in Algorithm~\ref{Alg:OGD}.
\begin{algorithm}
	\caption{Objective Function Learning via Online Gradient Descent}
	\label{Alg:OGD}
	\begin{algorithmic}[1]
		\INPUT observations $ (p_t, x_t) $
			and learning rates $ \eta_t = \frac{D}{G\sqrt{t}} $
			with $ t \in [T] $
		\OUTPUT sequence of objectives $ c_1, c_2, \ldots, c_T $
		\STATE choose $ y_1 \in \R^n $ arbitrarily
		\FOR{$ t = 1, \ldots, T $}
			\STATE $ c_t \leftarrow \argmin{\norm{y_t - c}_2 \mid c \in F} $
				\COMMENT{Project onto $F$}
			\STATE Observe parameters $ p_t $
			\STATE $ \sol{x}_t \leftarrow \argmin{\sp{c_t}{x} \mid x \in X(p_t)} $
				\COMMENT{Solve Subproblem~\eqref{Eq:SUB}}
			\STATE Observe expert solution $ x_t $
			\STATE $ y_t \leftarrow c_t - \eta_t (\sol{x}_t - x_t) $
				\COMMENT{Perform gradient descent step}	
		\ENDFOR
		\RETURN $ c_1, c_2, \ldots, c_T $.
	\end{algorithmic}
\end{algorithm}

Obviously, this second algorithm is more general than the first one --
we can now learn objective functions with arbitrary coefficients --
but is also more computationally involved
due to the projection step in Line~3.
For suitably bounded sets~$F$ and~$ X(p_t) $,
it yields the following performance guarantee
which follows directly from Lemma~\ref{Lem:Guarantee_OGD}
and the subsequent discussion on the choice of the learning rates.
\begin{theorem}
	\label{Thm:OGD}
	If $ \diam_2(F) \leq L $ for some $ L \geq 0 $
	and $ \diam_2(X_t) \leq K $, $ t \in [T] $, for some $ K \geq 0 $,
	then Algorithm~\ref{Alg:OGD} produces a series of objective functions
	$ (c_1, \ldots, c_T) $ with
	\begin{equation*}
		0 \leq \frac1T \sum_{t \in [T]} \sp{(c_t - \ctrue)}{(\sol{x}_t - x_t)}
			\leq \frac{3LK}{2\sqrt{T}}.
	\end{equation*}
\end{theorem}
Via this result, it is not only possible
to learn objective functions with arbitrary coefficients,
and incorporate prior knowledge of its stucture,
but we can also consider more general setups
for learning objective functions
as we demonstrate in the following.

Using the above theorem along the lines of Remark~\ref{Rem:BasisFunctions},
we can now learn a best-possible approximation
of an arbitrary objective function~$f$
via a piecewise-defined function
over a given triangulation of the feasible domain of~$f$.
As an example, we will consider learning a piecewise-linear approximation
of a continuous objective function $ f\colon [a, b] \to \R $
with $n$~breakpoints.
Let the breakpoints $ d_i $ with $ i = 1, \ldots n $
be such that $ a = d_1 \leq \ldots \leq d_n = b $.
Then we can choose piecewise-defined basis functions of the form
\begin{equation*}
	g_{ij}(x) = \begin{cases} 
		x^j, & \text{if } x \in [d_i, d_{i + 1}], \quad \quad 0, \quad \text{otherwise,}
		\end{cases}
\end{equation*}
with $ i = 1, \ldots n - 1 $ and $ j = 0, \ldots, j_{\max} $
for some desired maximal order $ j_{\max} $.
Our approximation of~$f$
will then be of the form $ \sum_{i = 1}^n \sum_{j = 1}^{j_{\max}} c_{\text{true}, ij} g_{ij}(x) $.
The set~$F$ from which $ \ctrue $ is assumed to originate can accordingly be chosen
such that it models the boundary conditions for the continuity of the approximation
via linear equations:
\begin{equation*}
	F \coloneqq \left\{c \in R^{n \cdot (j_{\max} + 1)} \colon
		\sum_{j = 1}^{j_{\max}} c_{ij} g_{ij}(d_i)
			= \sum_{j = 1}^{j_{\max}} c_{i + 1, j} g_{i + 1, j}(d_i),
				\quad i = 2, \ldots, n - 1\right\}.
\end{equation*}
This approach naturally generalizes to piecewise-defined functions
in higher dimensions
and higher orders of smoothness.

Using Theorem~\ref{Thm:OGD}, it is also possible to learn
linearly parameterized objective functions.
To this end, we generalize~$ (OPT(p))_p $
by considering the family of problems $ (OPT2(q, p))_{q, p} $ given by
\begin{equation}
	\max\lrset{\sp{\ctrue(q)}{x} : x \in X(p)},
	\label{Eq:OptP}
\end{equation}
where $ \ctrue $ is now a linear function $ \ctrue\colon Q \to \R^n $,
$ q \mapsto M_\text{true} q $
which depends on parameters $ q \in Q \subset \R^m $
via multiplication with some matrix~$ M_\text{true} \in \R^{n \times m}$.
The task is then to infer the matrix~$ M_\text{true} $
from the observed optimal solutions
(again assuming that the model~$ X(p) $
of the feasible region is known).

First, observe that $ \ctrue(q) = M_\text{true} q $ is equivalent to
$ \ctrue(q) = \sum_{i = 1}^m q(i) c_{i, \text{true}} $
for some basic objective functions $ c_{i, \text{true}} \in F $,
$ i = 1, \ldots, m $.
Defining
$ \vec(M) \coloneqq \tp{(c_{1, \text{true}}, \ldots, c_{m, \text{true}})} \in F^m $
as the vector that arises by stacking the columns of $ M_\text{true} $,
and similarly defining
$ \vec(q, x) \coloneqq \tp{(q(1) x, \ldots, q(m) x)} \in \R^{mn} $
for $ x \in X(p_t) $,
the objective function $ \sp{\ctrue(q)}{x} $ of Problem~\eqref{Eq:OptP}
can also be written as $ \sp{\vec(M_\text{true})}{\vec(q, x)} $.
We now assume that in each round $ t \in [T] $,
in addition to the parameter realizations~$ p_t $
determining the feasible region,
we observe parameter realizations~$ q_t $
determining the objective function
according to the above construction.
A direct application of Algorithm~\ref{Alg:OGD}
then allows us to learn all the $ c_{i, \text{true}} $'s simultaneously,
yielding the following approximation guarantee:
\begin{corollary}
	Let the sets $F$ and $ X(p_t) $, $ t \in [T] $, be as in Theorem~\ref{Thm:OGD}.
	If there is an $ N \geq 0 $ with $ \norm{q}_2 \leq N $ for all $ q \in Q $,
	Algorithm~\ref{Alg:OGD} produces a sequence of matrices
	$ (\vec(M_1), \ldots, \vec(M_T)) $ with
	\begin{equation*}
		0 \leq \frac1T \sum_{t \in [T]}
			\sp{(\vec(M_t) - \vec(M_\text{true}))}
				{(\vec(q_t, \sol{x}_t) - \vec(q_t, x_t))}
			\leq \frac{3\sqrt{m}NLK}{2\sqrt{T}}.
	\end{equation*}
\end{corollary}
\begin{proof}
	This result directly follows from $ \diam_2(F^m) = \sqrt{m}\diam_2(F) \leq \sqrt{m}L $
	together with $ \diam_2(\bigtimes_{i = 1}^m q_i X(p_t)) \leq N \diam_2(X(p_t)) \leq NK $ for  $ t \in [T] $.
\end{proof}

A further extension of our approach
is that of learning a dynamic objective function
where there are no parameters known
which determine how it changes from round to round.
Naturally, this is only possible
if the change in the true objective is suitably bounded.
The following result for OGD to learn a dynamic strategy
is the basis for an approximation guarantee:
\begin{lemma}[{\citet[Theorem 2]{zinkevich2003online}}]
	\label{Lem:Guarantee_OGD_Dynamic}
	Let $ \sum_{t \in [T - 1]} \norm{x_{t + 1} - x_{t}}_2 $
	be the \emph{path length} of a sequence $ (x_1, \ldots, x_T) $
	with $ x_t \in F $, $ t \in [T] $,
	and let $ \mathcal{X}(F, T, N) $
	be the set of of all sequences of $T$~vectors in $F$
	with path length at most~$R \geq 0 $.
	Under the same assumptions as for Lemma~\ref{Lem:Guarantee_OGD}
	and some fixed learning rate $ \eta \in \R_+ $, we have
	\begin{equation*}
		\sum_{t \in [T]} c_t(x_t) - \min_{(\hat{x}_1, \ldots, \hat{x}_T) \in \mathcal{X}(F, T, N)} \sum_{t \in [T]} c_t(\hat{x}_t)
			\leq \frac{7R^2}{4\eta} + \frac{RD^2}{\eta} + \frac{T\eta G}{2}.
	\end{equation*}
\end{lemma}
Using the fixed learning rate $ \eta = D/(G\sqrt{T}) $,
we obtain an upper bound of order $ \OO(R\sqrt{T}) $
in the above lemma,
such that the average error vanishes if the path length
grows slower asymptotically than $ \sqrt{T} $.
This directly translates into a guarantee for the regret
when learning an dynamic objective function whose path length in~$F$
is bounded by some constant~$ R \geq 0 $.
A~similar parametrization approach based on MWU
can be found in \citet{WardMasterBambos2019}.

We have seen that our algorithms based on MWU and OGD
enjoy similar convergence speeds of $ \OO(T) $.
The MWU-based algorithm only works for positive objectives,
but there it has an exponential speed up
with respect to the problem dimension~$n$.
In Theorem~\ref{Thm:Convergence}, the upper error bound
for MWU is $ 2 K \sqrt{(\ln n)/T} $.
For OGD, if we set $ L = \sqrt{n} $
for the diameter of the simplex in Theorem~\ref{Thm:OGD},
we arrive at $ 3/2 K \sqrt{n}/T $ as the upper bound.
This shows that, theoretically, for positive objectives
and for large~$n$, the MWU-based algorithm should perform better. 

\subsection{Connections between the Online and the Offline Case}

As laid out before, we can learn objectives of optimization problems
in the offline case by solving Problem~\eqref{Eq:OPT}.
This is generally possible via a subgradient method
as long as we can solve Subproblem~\eqref{Eq:SUB};
for polyhedral feasible regions~$ X(p_t) $
and a polyhedral region~$F$ for the possible choices of~$c$,
this can even be done by solving the LP~\eqref{Eq:LP_Reformulation},
which minimizes the total duality gap.
In cases where the observations $ (p_t, x_t) $ are not given all at once
but are revealed sequentially,
it is possible to solve Problem~\eqref{Eq:OPT}
only for the data that has been revealed so far,
and to use the resulting vector~$c$
as the objective function in the next round,
see Algorithm~\ref{Alg:FTL}.
This strategy might be a useful heuristic for the online case,
especially if in the end a consistent objective
for all observations is required.
\begin{algorithm}
	\caption{Heuristic Objective Function Learning via Follow-the-Leader}
	\label{Alg:FTL}
	\begin{algorithmic}[1]
		\INPUT Observations $ (p_t, x_t) $
			with $ t \in [T] $
		\OUTPUT Sequence of objectives $ (c_1, c_2, \ldots, c_T) $
		\FOR{$ t = 1, \ldots, T $}
			\STATE $ c_t \leftarrow
				\lrargmin{\sum_{\tau \in [t]} \left(\left(\max_{x \in X(p_\tau)} \sp{c}{x}\right) - \sp{c}{x_\tau}\right): c \in F} $
			\STATE
				\COMMENT{Solve Problem~\eqref{Eq:OPT}}
			\STATE Observe parameters $ p_t $
			\STATE $ \sol{x}_t \leftarrow \argmin{\sp{c_t}{x} \mid x \in X(p_t)} $
				\COMMENT{Solve Subproblem~\eqref{Eq:SUB}}
			\STATE Observe expert solution $ x_t $
		\ENDFOR
		\RETURN $ (c_1, c_2, \ldots, c_T) $.
	\end{algorithmic}
\end{algorithm}

Algorithm~\ref{Alg:FTL} is in fact an adaption
of the well-known follow-the-leader (FTL) algorithm
(see \citet{KalaiVempala2005,Hannan1957}) for our specific setting.
In each round~$t$, the current objective
is chosen as $ c_t \in \argmin{\sum_{\tau \in [t - 1]} g_{\tau} (c) \mid
	c \in F} $,
where the loss function is $ g\colon F \to \R $,
$ g_t(c) = \max_{x \in X(p_t)} \sp{c}{x} - \sp{c}{x_t} $
(which means that in the first round an abitrary $ c_1 $ is chosen).
For adversarially chosen $ p_t $'s or expert solutions $ x_t $,
the average regret of Algorithm~\ref{Alg:FTL}
does not necessarily converge to~$0$,
as counterexample \ref{exp:ggbsp}
in the appendix shows.
It is well known that the regularized follow-the-leader algorithm
(FTRL, for short), however, ensures similar regret bounds as MWU and OGD.
One possible way to achieve this is by taking
a strongly convex regularization function $ R \colon F \to \R $
and to change the update rule
to $ c_t \in \argmin{\sum_{\tau \in [t - 1]} g_{\tau}(c) + R(c) \mid
	c \in F} $.
This leads to a regret bound of $ 2GD \sqrt{2T} $, for example
(see \citet{ocoBook}).
FTRL has the drawback that may be more computationally involved;
for example, in the case of polyhedral $ X(p_t) $ and $F$,
we need to solve a more general convex problem
instead of a linear problem each round.
Note also that for the loss function~$ g_t $,
we only obtain vanishing average regret for the objective error,
as the true objective~$ \ctrue $ produces a loss of zero.
In order for the average total error to vanish,
we need to choose the loss function~$ f_t $
with $ f_t(c) = \sp{(\bar{x}_t - x_t)}{c} $ again,
which, however, does not necessarily lead
to learned objectives consistent with Problem~\eqref{Eq:OPT}. 

Via so-called \emph{online-to-batch conversions},
we can use our previously developed online-learning algorithms
to solve stochastic variants of the offline case.
For the loss function~$ g_t $ from above,
we can use Corollary~5.2 from \citet{shalev2012online}
to bound the expected error
produced by the objective $ \bar{c} \coloneqq 1/T \sum_{t \in [T]} c_t $,
\ie the average of the objectives played in each round.
\begin{theorem}
	\label{Thm:Stochastic}
	Let $ \pi \coloneqq ((p_1, x_1), (p_2, x_2), \ldots, (p_T, x_T)) $
	be a sequence of $T$~independently sampled observations
	according to some distribution~$D$
	over the set of possible observations.
	For the sequence of objectives $ (c_1, c_2, \ldots, c_T) $
	produced by either one of Algorithms~\ref{Alg:MWU}, \ref{Alg:OGD}
	or a strongly-convex regularized version of~\ref{Alg:FTL}
	in response to these samples
	and using the loss functions~$ g_t $,
	we obtain the following guarantee
	for the average objective $ \bar{c} $:
	\begin{equation*}
		0 \leq \E_{\pi}\left[\E_{(p, x) \sim D}\left[
			\left(\max_{\bar{x} \in X(p)} \sp{\bar{c}}{\bar{x}}\right) - \sp{\bar{c}}{x}\right]\right]
				\leq \E_{\pi}\left[\frac1T \sum_{t \in [T]}
					\sp{c_t}{(\bar{x}_t - x_t)}\right]
				\in \OO\left(\frac1{\sqrt{T}}\right).
	\end{equation*}
\end{theorem}
This means that in the sense of Thereom~\ref{Thm:Stochastic},
our online algorithms can also be used to solve the offline case.
Especially, it is well justified to use them as a heuristic
for Problem~\eqref{Eq:OPT} when the observations
can be assumed to be independently sampled from a common distribution.
Note that the OGD-based algorithm behaves the same
under both loss functions $ f_t $ and $ g_t $
and can accordingly be used in both the online and the offline case.

\subsection{Remarks on Learning Constraints}
\label{Sec:Constraints}

A natural question that arises is
if the same methodology we have used
to learn the objective of an optimization problem
can be used to learn constraints as well.
We will only briefly address this case here
to indicate where some obstacles lie.
We consider the family of optimization
problems~$ (\textrm{OPT3}(p))_p $, $ p \in P \subseteq \R^k $,
given by
\begin{equation*}
	\max\lrset{\sp{c(p)}{x} : A x \leq \btrue, x  \geq 0}, 
\end{equation*}
where the objective function $ c(p) \in \R^n $
depends on the parameters~$p$,
$ A \in \R^{m \times n} $ is the constraint matrix
and $ \btrue \in \R^m $ is the right-hand side. 
Again we assume that the learner
observes pairs of parameter realizations
and corresponding optimal solutions $ (p_t, x_t) $
in each round $ t \in [T] $.
Furthermore, we assume that the objective functions~$ c(p_t) $
are known to both the learner and the expert.
The same can be assumed for~$A$ without loss of generality
by standard arguments.
The right-hand side~$ \btrue $ is only known to the expert
and to be learned from the observations.

The most natural approach
for solving this learning problem
is to apply Algorithm~\ref{Alg:MWU}
to the dual of $ \textrm{OPT2}(p_t) $,
\begin{equation*}
\min \lrset{\sp{\btrue}{y} : \tp{A} y \geq c(p_t), y  \geq 0}, 
\end{equation*}
where $y$ are the dual variables for the linear constraints.  In the
dual problem, $ \btrue $ is the unknown objective function
($ \btrue \geq 0 $ without loss of generality), while the constraints
to be optimized over in each round are known --
the same setting as before.
It is important to note though
that the learner has to observe the \emph{dual} optimal solutions~$ y_t $
and the guarantee will be that the \emph{dual regret} tends to $0$.
In addition, it remains open
whether this scheme can directly be extended
to also have the primal regret converge to~$0$;
we suspect the answer to be in the negative in general.



\section{Applications}
\label{Sec:Applications}
We will now show several example applications
of our framework for learning objective functions
from observed decisions.
These are the learning of customer preferences from observed purchases,
the learning of travel times in a road network
and the learning of optimal delivery routes.
In each case, we will study different assumptions
for the nature of the objective function to learn
in order to demonstrate the flexibility of our approach.

Our computational experiments
have been conducted on
a server comprising Intel Xeon E5-2690 3.00~GHz
computers with 25~MB cache and 128~GB~RAM.
We have implemented our framework using the Python-API of
\emph{Gurobi~8.0.1} (see \citet{gurobi}).
In the appendix,
we present additional analyses to the results shown here.

\subsection{Learning Customer Preferences}
\label{Sec:Application.Knapsack}
We consider a market where different goods
can be bought by its customers.
The prices for the goods can vary over time,
and we assume that the goods are chosen by the customers to
maximize their utility given their respective budget constraints.
Each sample $ (p_t, x_t) $ corresponds to a customer $ t \in [T] $,
where $ p_t = (p_{t0}, p_{t1}, \ldots, p_{tn}) $
contains the budget~$ p_{t0} \geq 0 $ 
and the current prices~$ p_{ti} \geq 0 $ for each good~$ i \in [n] $.
Customer~$t$ is then assumed to solve the following optimization problem
$ \textrm{OPT}(p_t) $:
\begin{equation*}
	\max  \lrset{\sum_{i = 1}^n u_i x_i : 
	\sum_{i = 1}^n p_{ti} x_i \leq p_{t0},\,
	x \in \F^n},
\end{equation*}
where the aggregate utilities~$ u_i \geq 0 $ of good~$i$
to the customers are unknown.
Learning these utilities can, for example, help stores
to take suitable assortment choices.

We consider two different setups:
first, we assume that the goods are divisible,
which means that the condition $ x \in \F^n $ is relaxed
to $ x \in \FR^n $;
this is the \emph{linear knapsack problem}.
In the second case, the goods are indivisible,
so that we solve the \emph{integer knapsack problem}
with the original constraint $ x \in \F^n $.

To simulate the first setup,
we generated 50~random instances,
in each instance considering $ T = 500 $ observations
for $ n = 50 $ goods.
The customers' unknown utilities for the different goods
are drawn as integer numbers from the interval $ [1, 1000] $
according to a uniform distribution
and then normalized in the $1$-norm.
The prices for sample~$t$ are chosen to be
$ p_{ti} = u_i + 100 + r_{ti} $, $ i = 1, \ldots, n $,
where $ r_{ti} $ is an integer
uniformly drawn from the interval $ [-10, 10] $.
Finally, the right-hand side $ p_{t0} $ is again an integer
drawn uniformly from the interval $ [1, \sum_{i = 1}^n p_{ti} - 1] $.
Choosing utilities and weights in a strongly correlated fashion as above
typically leads to harder (integer) knapsack problems,
see \cite{Pisinger2005} for more details.
Note, however, that the focus here is not the hardness of the instances
but rather the non-triviality of the learning problem.

In the following, we study learning the utilities
for the linear knapsack problem
using three different algorithms:
Algorithm~\ref{Alg:MWU} based on MWU,
Algorithm~\ref{Alg:OGD} based on OGD with $ F = \Delta_n $
and the constant learning rate $ \eta_t = D/(G\sqrt{T}) $
as well as Algorithm~\ref{Alg:FTL}
based on heuristic sequential LP solving.

The solution error for the linear knapsack is shown
in our plots over all rounds
in Figure~\ref{Fig:LinearKnapsack}.
They depict the arithmetic mean
of the solution error
over the 50~instances,
together with the first and second standard deviation.
We can see that the average errors converge to~$0$ rather quickly,
such that it is possible to take practically optimal decision
after 500~rounds in all cases --
with MWU and OGD performing very similarly to each other,
which is explainable by the fact
that the algorithms are basically the same,
except for the difference in the projection step.
For LP, the initially high errors
lead to a lower rate of convergence in the later rounds.
\begin{figure}[h]
	\hfill
	\subfloat[Solution error for MWU]{
		\includegraphics[width=0.31\linewidth]{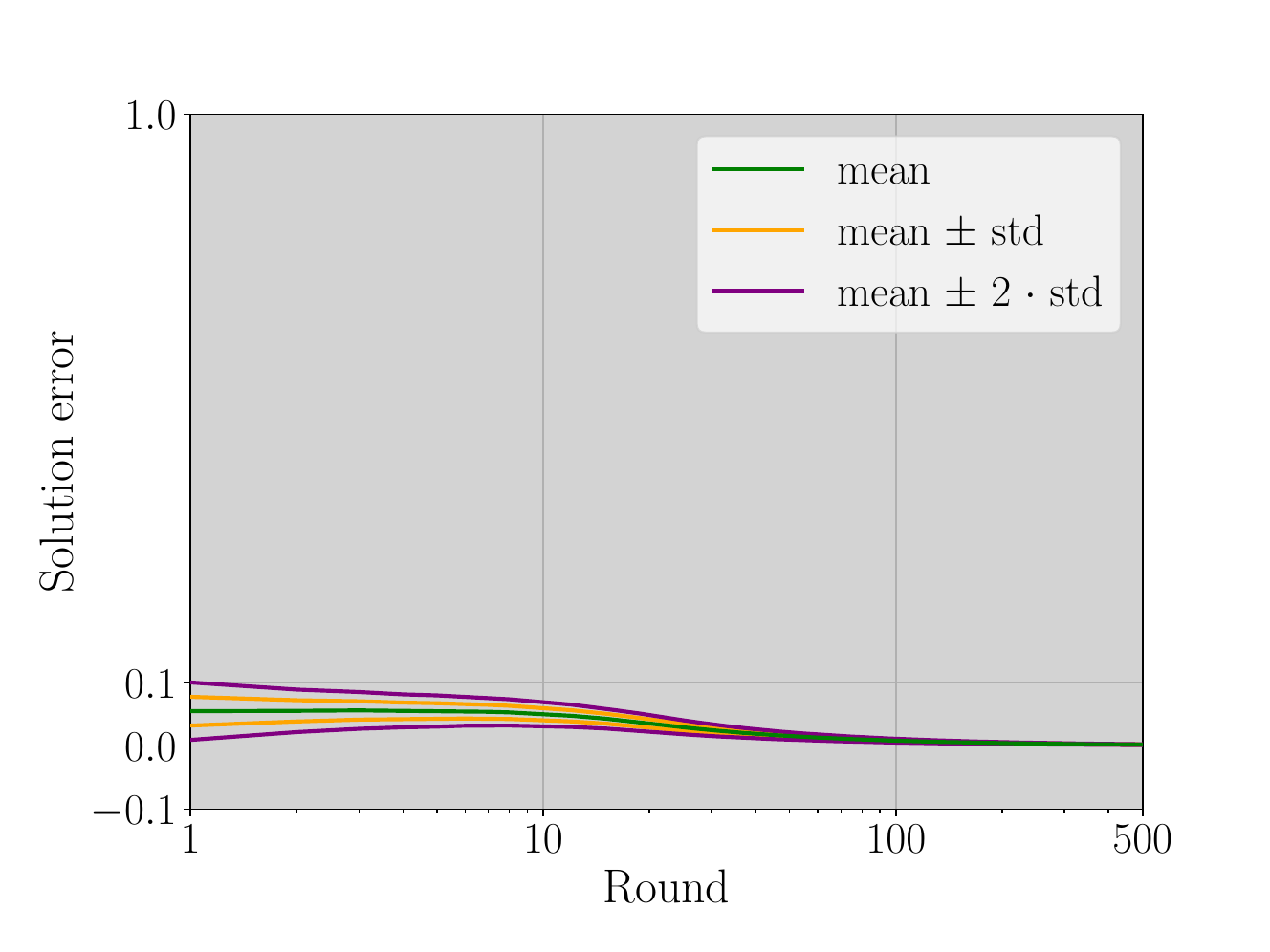}
	}
	\hfill
	\subfloat[Solution error for OGD]{
		\includegraphics[width=0.31\linewidth]{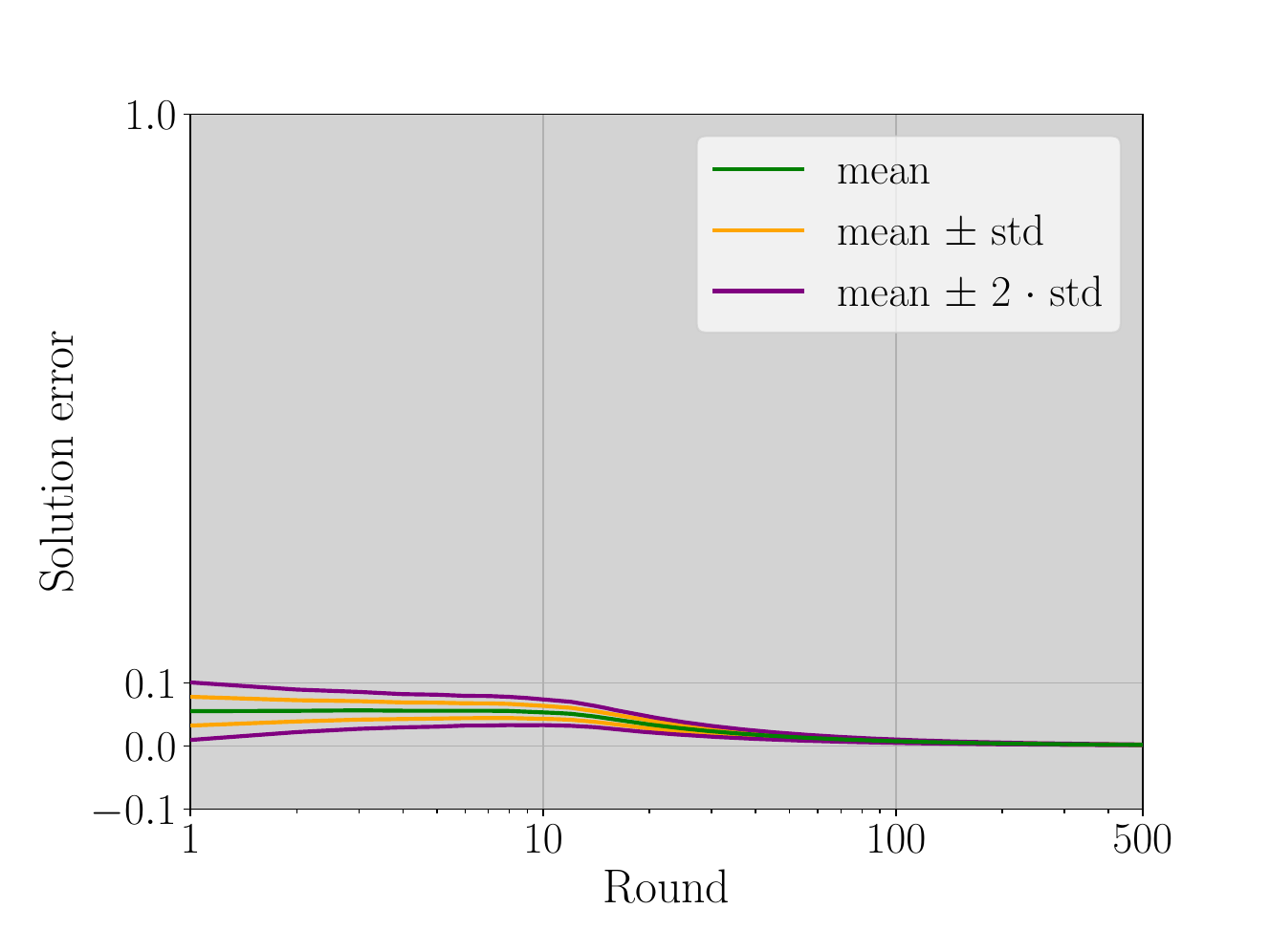}
	}
	\hfill
	\subfloat[Solution error for LP]{
		\includegraphics[width=0.31\linewidth]{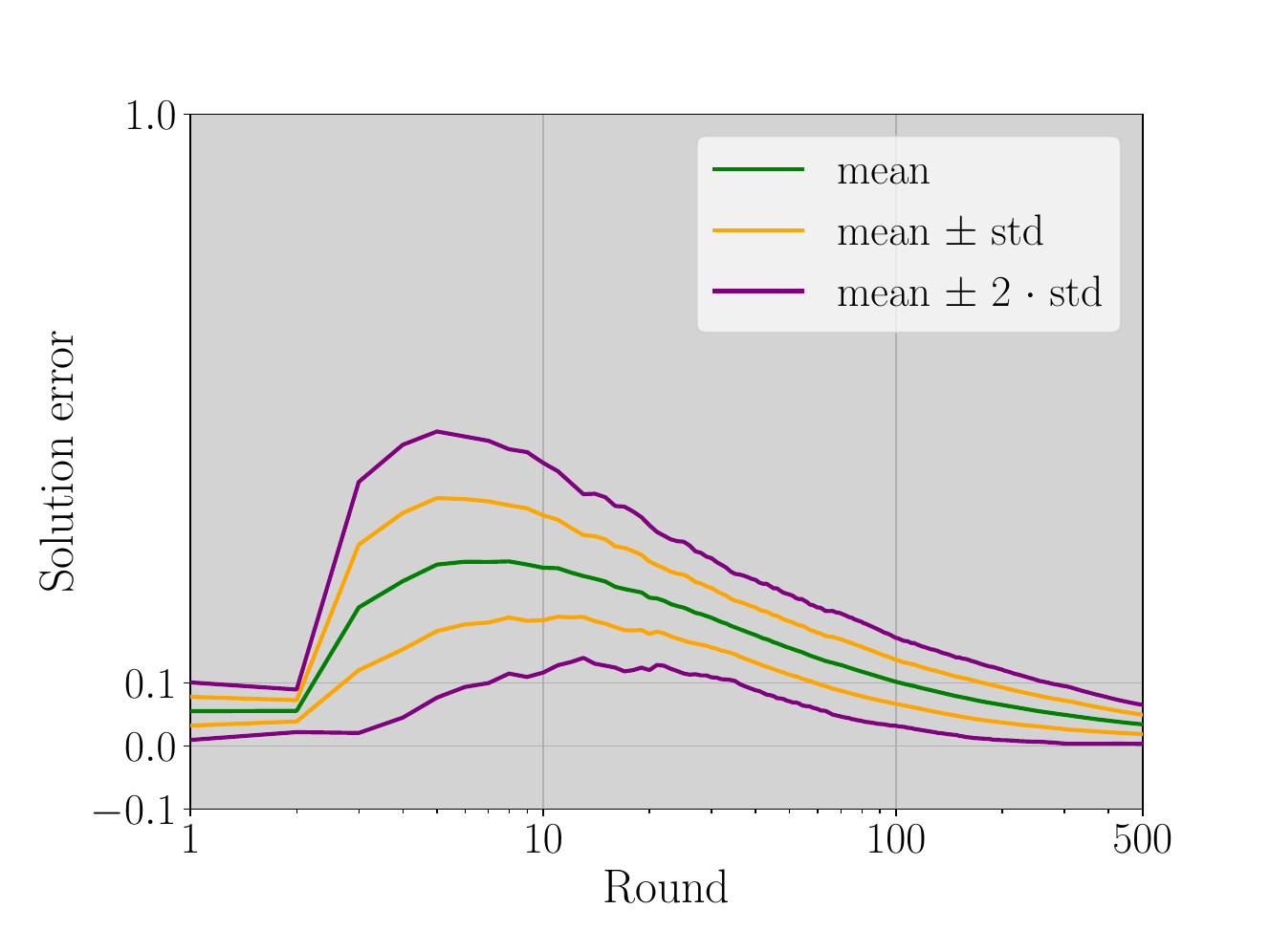}
	}
	\hfill\null
	\caption{The arithmetic means and standard deviations of the different error types in each round (\emph{red}) and averaged up to the current round (\emph{blue}) for the linear knapsack problem with $ n = 50 $ items over $ T = 500 $ rounds for each of the three algorithms, with the arithmetic mean taken over 500~runs, on a doubly symmetric-logarithmic plot}
	\label{Fig:LinearKnapsack}
\end{figure}

We also conducted an experiment for the integer knapsack problem,
using $ n = 1000 $ and $ T = 1000 $.
This time, we considered a single instance
run with MWU as well as two different versions of OGD:
one with a fixed learning rate ($ \eta = D/G\sqrt{T} $) and one with a dynamic learning rate ($\eta_t = D/G\sqrt{t} $).
From Table~\ref{Tab:IntegerKnapsack},
which shows the average errors after 10, 100 and 1000 rounds,
we see that the behaviour of the algorithms
is virtually the same as for the linear knapsack,
which means that the learning task does not become significantly harder
because of the integrality requirement.
\begin{table}[h]
	\begin{center}
		\begin{tabular}{l|rrr|rrr|rrr}
			Algorithm & \multicolumn{3}{c|}{MWU} & \multicolumn{3}{c|}{OGD (fixed)} & \multicolumn{3}{c}{OGD (dynamic)}\\
			\# Rounds		& 10	& 100	& 1000	& 10	& 100	& 1000 & 10	& 100	& 1000\\
			\midrule
			$\varnothing$ obj.\ error	&	0.15 & 0.02 & $<$0.01	&	0.20 & 0.02 & $<$0.01 & 0.05 & 0.01 & $<$0.01\\
			$\varnothing$ sol.\ error	& 	0.05 & 0.01 & $<$0.01	&	0.06 & 0.01 & $<$0.01 & 0.03 & 0.01 & $<$0.01\\
			$\varnothing$ total error		& 	0.20 & 0.03 & $<$0.01	& 0.26 & 0.03 & $<$0.01 & 0.08 & 0.02 & $<$0.01 \\
			\bottomrule
		\end{tabular} 
	\end{center}
	\caption{Average objective, solution and total error for the integer knapsack problem with $ n = 1000 $ items}
	\label{Tab:IntegerKnapsack}
\end{table}

Next, we study the change in the learned objective function over time
at the example of the integer knapsack problem.
In Figure~\ref{mwa:dist},
we compare the convergence behaviour
of the learned objective in the $1$-norm
for the MWU~algorithm.
\begin{figure}[h]
	\hfill
	\subfloat[MWU]{
		\includegraphics[width=0.31\linewidth]{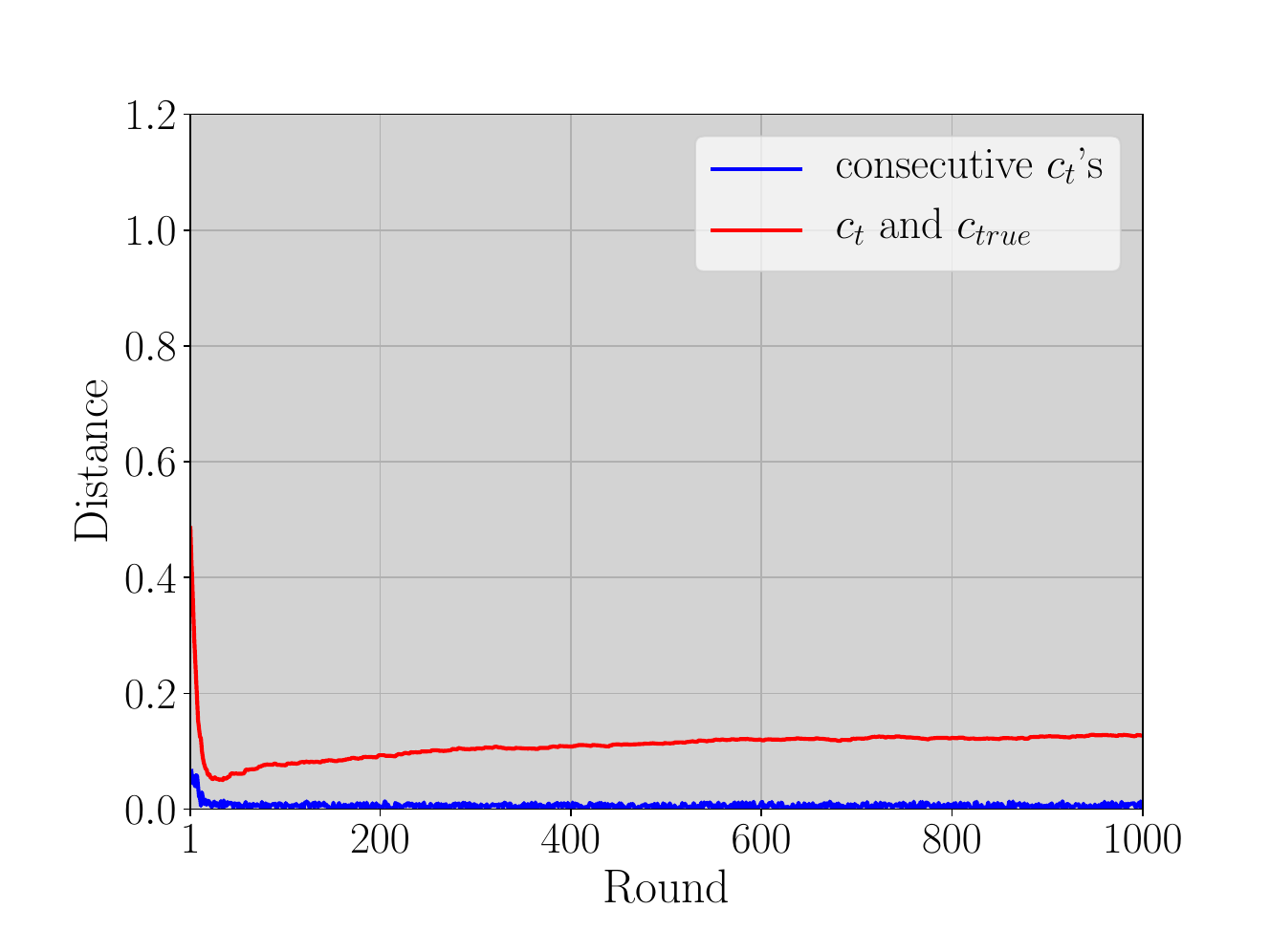}
			\label{mwa:dist}
	}
	\hfill
	\subfloat[MWU]{
	\includegraphics[width=0.31\linewidth]{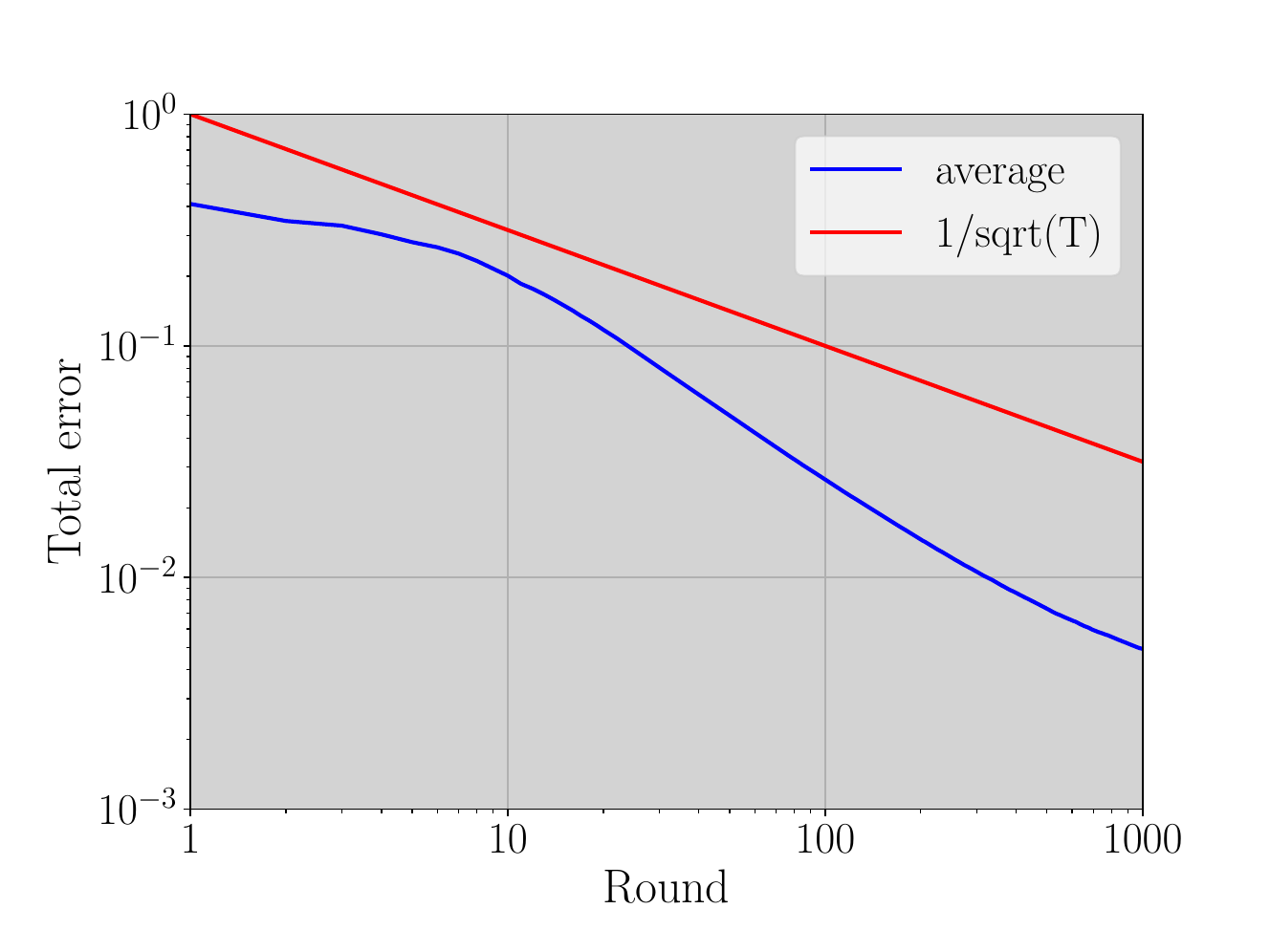}
	\label{mwa:convg}
}
\hfill
\subfloat[OGD, fixed learning rate]{
	\includegraphics[width=0.31\linewidth]{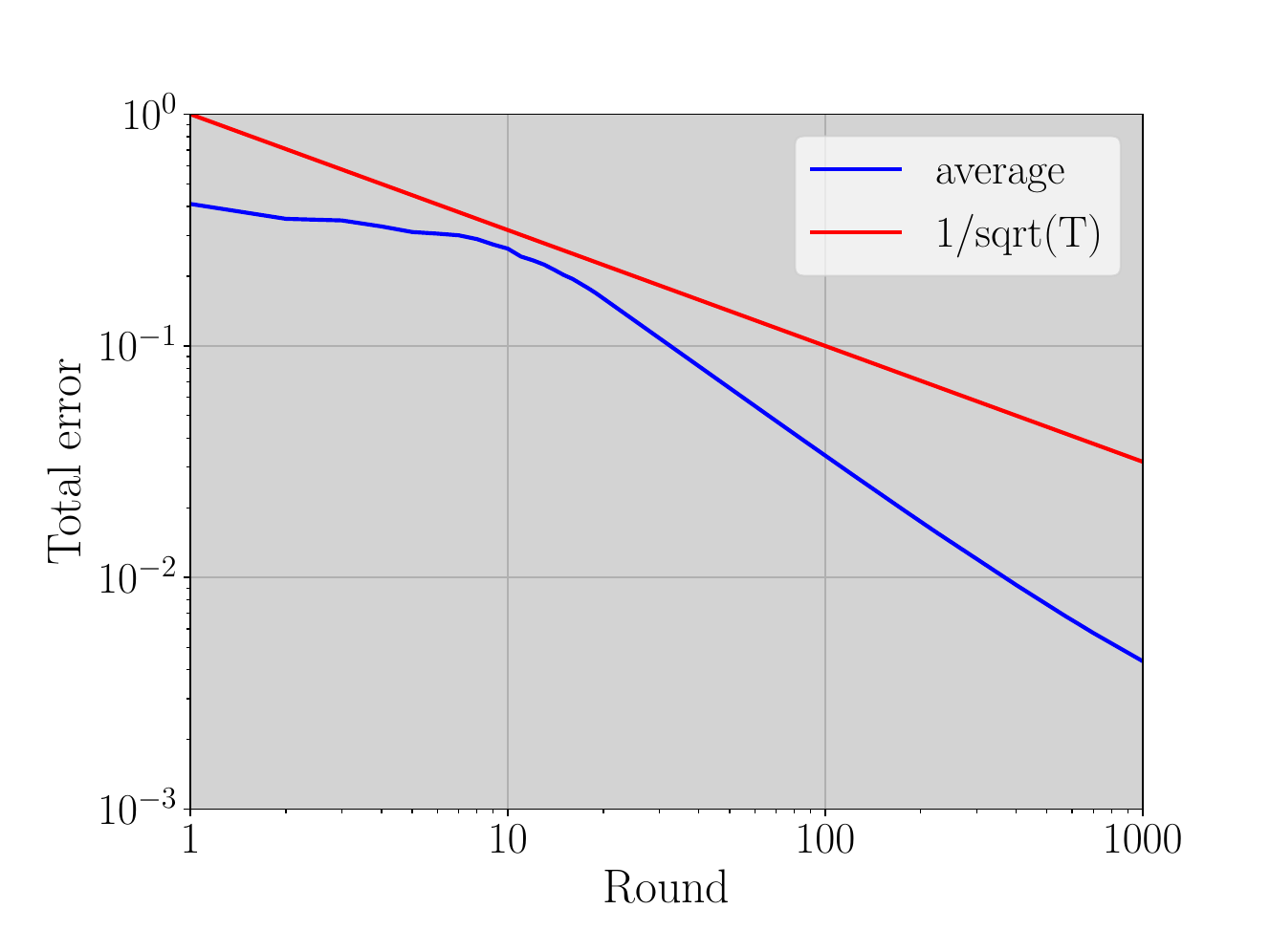}
	\label{ogd:convg}
}
	\hfill\null
	\caption{(a): Distance in the $1$-norm between two consecutive objective objective functions (\emph{blue}) and between the objective function in the current round and the true objective function (\emph{red}) for the integer knapsack problem. (b), (c): Doubly logarithmic plot of the average total error (\emph{blue}) for the integer knapsack problem against the function $ f(T) = 1/\sqrt{T} $ (\emph{red})}
	\label{Fig:IntegerKnapsack-Updates}
\end{figure}
It is visible at first sight that empirically
we do not converge to the true objective function;
rather, the distance to the true objective function
converges to about $0.1$.
The two other algorithms behave similarly
as we show in the appendix.
As discussed in Example~\ref{Ex:AlternativeObjective},
there might be many different objective functions
explaining the same observed solutions.
From the results presented before,
we have already seen that the alternative objectives we find
perform about equally well than the original one.

A related question of interest is
what the actual speed of convergence
of the three algorithms is
in comparison to their proven asymptotic behavior.
Figure~\ref{mwa:convg} and \ref{ogd:convg}
show the MWA and OGD algorithm, respectively,
for the integer knapsack,
depicting the average total error up to a given round
versus the asymptotic upper bound
represented by the function $ f(T) = 1/\sqrt{T} $
on a log-log-scale.
What we observe is that -- over the total observation horizon --
the order of convergence is basically the same as that of~$f$,
where, again, all algorithms perform about equally well.

Finally, we investigate the out-of-sample performances
of several policies for continuing the optimization
if we receive no further feedback
after a given point in time.
\begin{figure}[h]
	\hfill
	\subfloat[$ c_{100} $]{
		\includegraphics[width=0.31\linewidth]{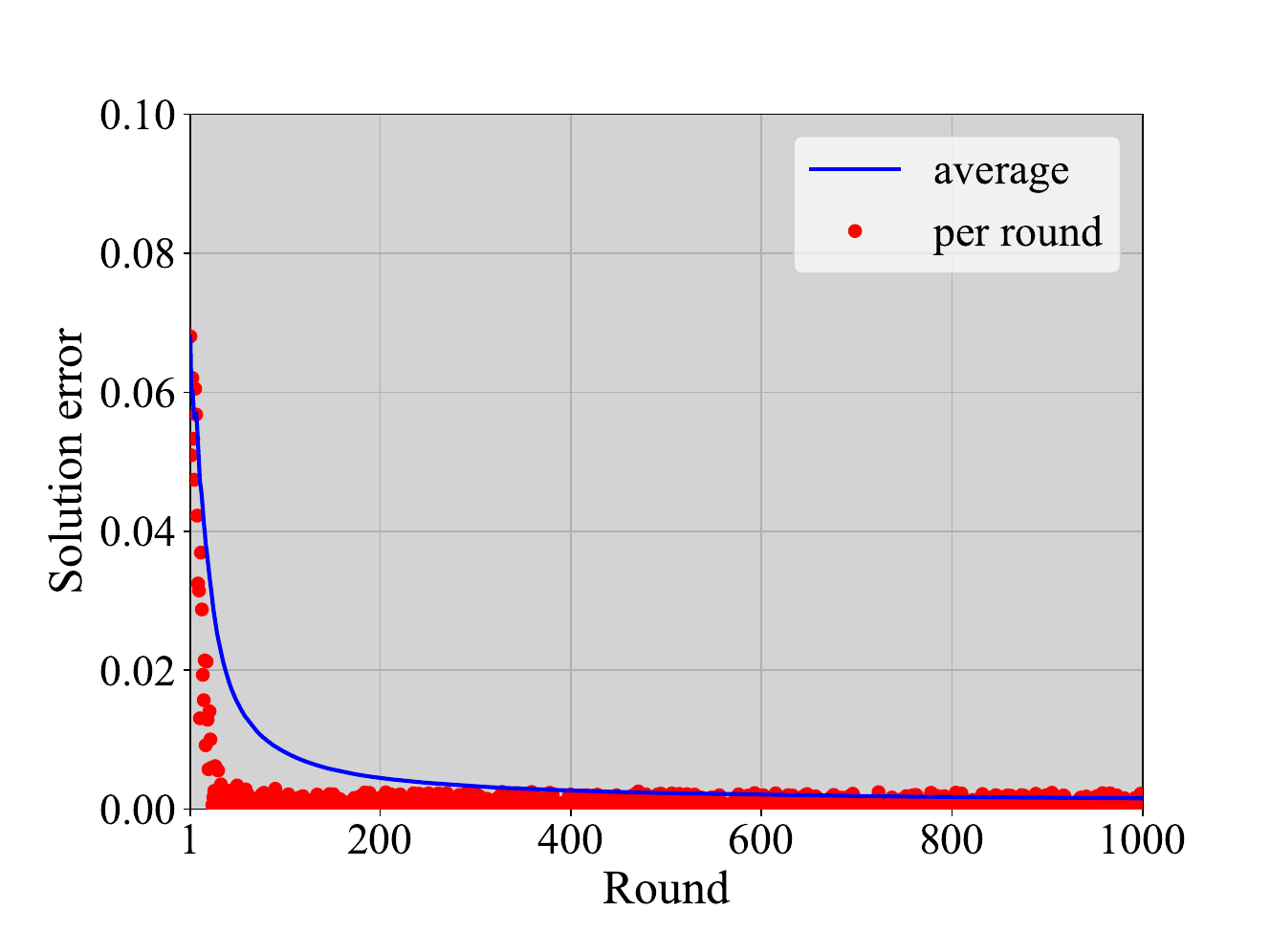}
		\label{Fig:IntegerKnapsack-Policies-cT}
	}
	\hfill
	\subfloat[$ \frac{1}{100} \sum_{t = 1}^{100} c_t $]{
	\includegraphics[width=0.31\linewidth]{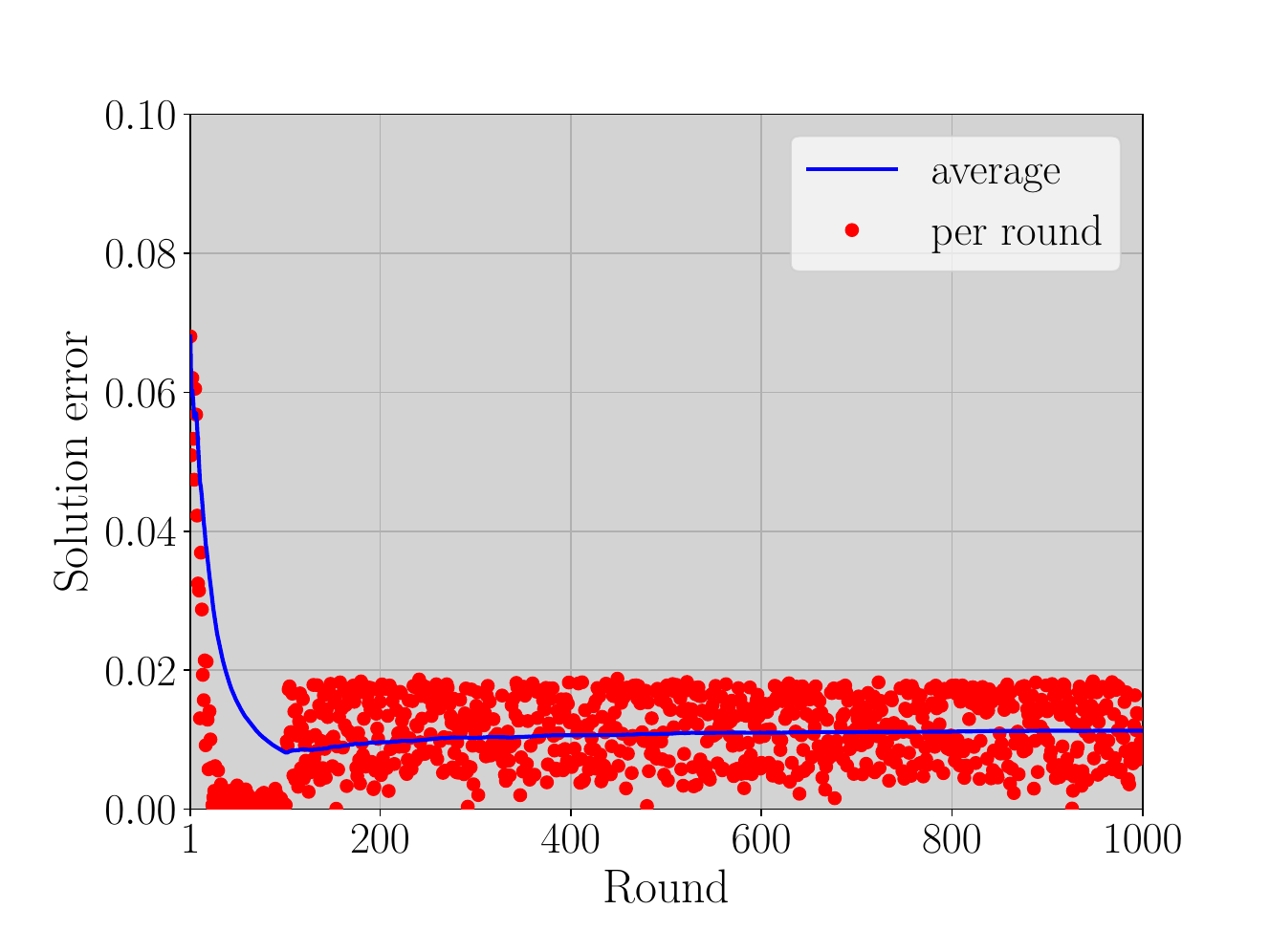}
		\label{Fig:IntegerKnapsack-Policies-average-c_t}
	}
	\subfloat[Best $ c_t $ out of $ t = 1, \ldots, 100 $]{
		\includegraphics[width=0.31\linewidth]{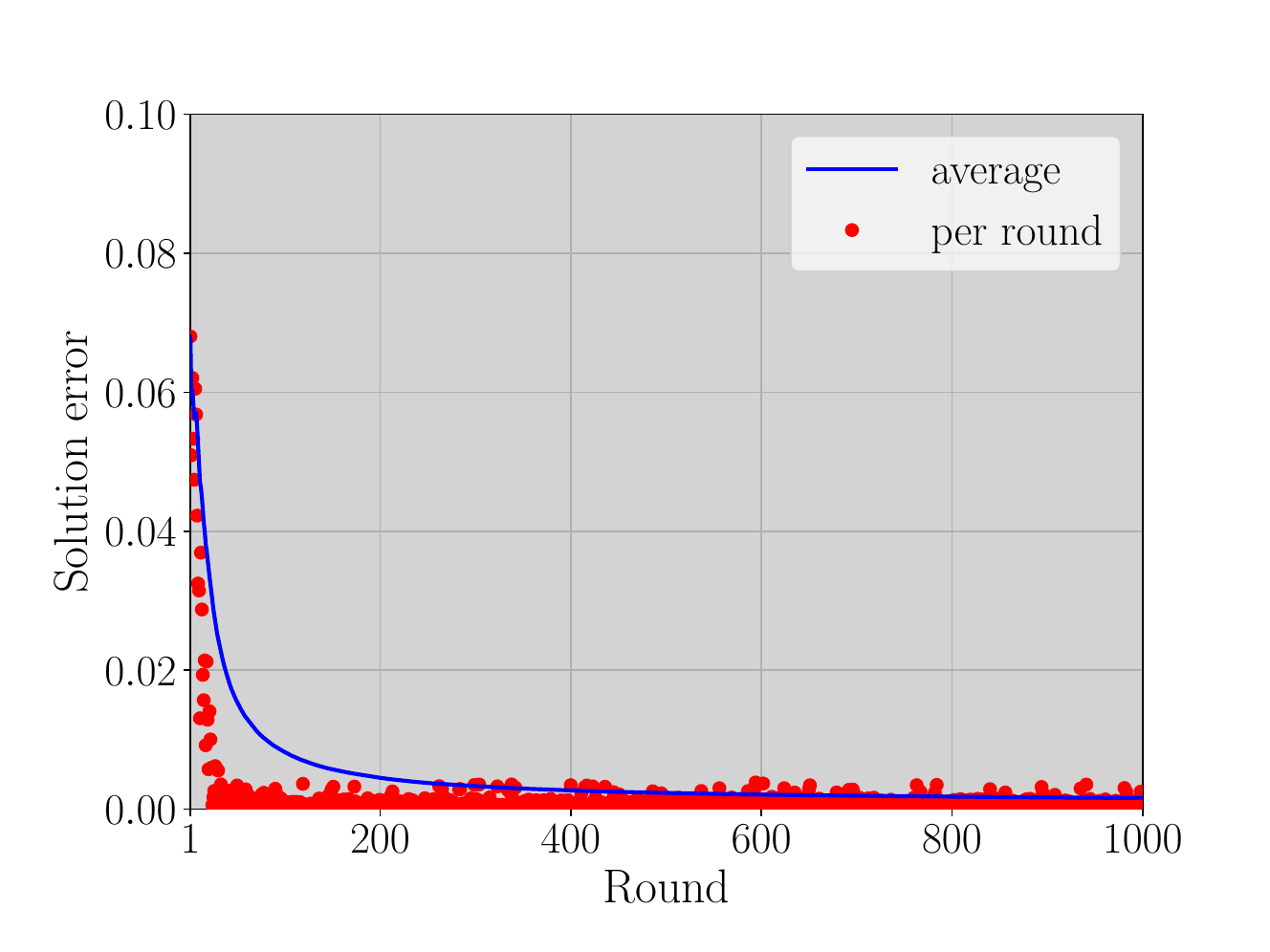}
		\label{Fig:IntegerKnapsack-Policies-Best-c_t}
	}
	\hfill
	\caption{Comparison of the solution error for three policies to continue the optimization if the algorithm (here: MWU) receives no further update after $ t = 100 $ rounds}
	\label{Fig:IntegerKnapsack-Policies}
\end{figure}
To this end, we show in Figure~\ref{Fig:IntegerKnapsack-Policies}
the solution error for the objectives produced by MWU
for the case that there are no further updates after $100$~rounds.
In Figure~\ref{Fig:IntegerKnapsack-Policies-cT},
we use $ c_{100} $ for the remaining $900$~rounds,
in Figure~\ref{Fig:IntegerKnapsack-Policies-average-c_t},
we use $ \frac{1}{100} \sum_{t = 1}^{100} c_t $,
in Figure~\ref{Fig:IntegerKnapsack-Policies-Best-c_t},
it is the $ c_t $ which produced the lowest error
in the corresponding round~$t$,
over all rounds $ t = 1, \ldots, 100 $.
We find that $ c_{100} $ consistently produces solutions
with a cost similar to that of the true objective
and thus generalizes very well to the unseen data in this instance.
The same holds for the best $ c_t $,
which in this case was highly non-unique,
as there are many rounds where the error is~$0$.
Thus, we averaged over all $ c_t $ with error~$0$
and chose the resulting objective.
In contrast, averaging over all~$ c_t $
from the first $100$~rounds performs much worse.
Altogether, we conclude that our approach leads to objective functions
which provide a consistent explanation for the observed decisions
in settings with i.i.d.\ sampled parameters~$ p_t $ empirically,
even for those observations
which the algorithm has not seen before.

\subsection{Learning Travel Times}
\label{Sec:Application.Paths}
In our second computational experiment,
we consider a street network
where the travel times on a segment may vary over the day.
Each driver in the network is assumed to choose a route
that leads from some origin to some destination in the shortest time possible. 
In other words, each driver solves a shortest-path problem
on the same directed graph $ G = (V, A) $.
An observation $ p_t = (p^1_t, p^2_t) $ for $ t \in [T] $
represents a driver in the network
who departs at a given time step
in the observation period
which we also denote by~$t$,
going from the starting point $ p^1_t \in V $
to the end point $ p^2_t \in V $.
The entries of $ x_t $ indicate the path taken by driver~$t$,
which are obtained by solving the following optimization problem $ \textrm{OPT}(p_t) $: 
\begin{mip}
	\objective{\min}{\sum_{a \in A} c_{\text{true}, ta} x_{ta}}\\
	\stconstraint{\sum_{a \in \delta^{+}(v)} x_{ta}
		- \sum_{a \in \delta^{-}(v)} x_{ta}}
		{=}{\choice{1, & \text{if } v = p^1_t\\
			-1, & \text{if } v = p^2_t\\
			0, & \text{otherwise}}}{(\forall v \in V)}\\
	\constraint{x_{ta}}{\in}{\F^{\card{A}}.}{}
\end{mip}
Observing the paths of each of the drivers,
we want to learn the values~$ c_{\text{true}, ta} $
corresponding to the travel times
to traverse arc~$ a \in A $ at time step $t$.
The major difference to our previous experiment is
that here the objective function will be allowed to change over time,
representing, for example, slowdowns due to traffic congestions.

We created instances of the problem
based on a real-world street network,
namely an aggregated version of the city map of Chicago.
It is available as instance \emph{ChicacoSketch}
in Ben Stabler's library of transportation networks (see \cite{bstabler}),
and has 933~nodes and 2950~arcs
(of which we ignore the 387~nodes representing \qm{zones}
as well as their incident arcs).
Each arc~$a$ has a certain free-flow time $ c_{\text{free}, a} $,
which we assume to be the unknown uncongested travel time.
For each driver~$t$, we chose a random pair
of an origin and a destination node.
Furthermore, we consider 5 hours of real time and one driver per 5 minutes
for creating the observations.
Hence, our time horizon is $ T = 60 $, which is also the number of drivers.

We consider three different settings.
In a first step, we try to retain the free-flow times
from the observed paths,
\ie we assume the travel times to be constant:
$ c_{\text{true}, ta} \coloneqq c_{\text{free}, a} $
for all $ a \in A $ and $ t = 1, \ldots, T $.
As a second test, we integrate temporary increases in the travel time
induced by congestions on the most-frequently used arcs.
These increases were modelled in the following way:
we computed all shortest paths in the network between any two nodes
and chose the top 5\% of arcs with the highest number
of shortest paths traversing them as bottlenecks.
For each of these bottlenecks~$a$,
we chose the travel time at time step~$ t $ to be
\begin{equation*}
		c_{\text{true}, ta} \coloneqq \begin{cases}
			c_{\text{free}, a}, & \text{otherwise}, \ \ \ \ \ \ \ \ \ \ \ \ \ \ 
			(\frac{t - 6}{6}) c_{\text{free}, a},  \ \ \ \ \ \text{if } t \in \set{12, \ldots, 18},\\
			2c_{\text{free}, a}, & \text{if } t \in \set{19, \ldots, 29}, \ \ \
			(\frac{42 - t}{6}) c_{\text{free}, a}, \ \ \ \  \text{if } t \in \set{30,\ldots,36}.
		\end{cases} 
\end{equation*}
This means that at one hour of real time,
the congestions on all bottleneck arcs~$a$
start to build up and reach the maximal congestion within 30~minutes,
staying on maximal congestion for one hour
and then ebbing away within 30~minutes.
Finally, we consider the case of abrupt changes in travel time
as they might arise from roads
which are suddenly blocked for some reason.
To this end, we altered the travel times
of the same arcs as before,
but chose the travel time at time step~$t$ as
\begin{equation*}
	c_{\text{true}, ta} \coloneqq \begin{cases}
		1000 c_{\text{free}, a}, & \text{if } t \in \set{12, \ldots, 36}, \quad c_{\text{free}, a}, \quad  \text{otherwise.}
		\end{cases}
\end{equation*}
We then used MWU to learn the dynamically changing travel times.
\begin{figure}[h]
	\hfill
	\subfloat[No congestion]{
		\includegraphics[width=0.31\linewidth]{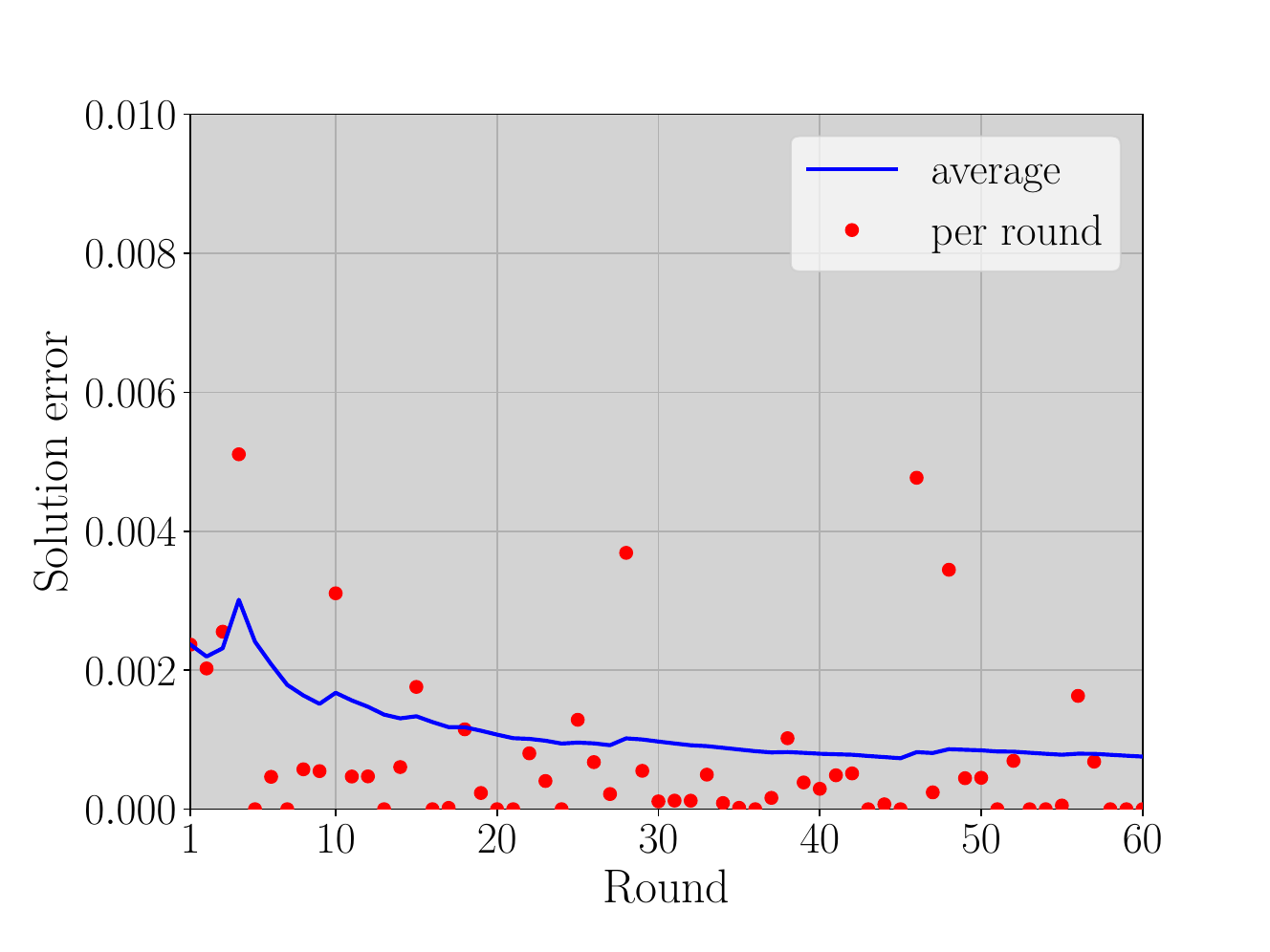}
		\label{Fig:ShortestPathProblem-Original}
	}
	\hfill
	\subfloat[Gradual congestion]{
		\includegraphics[width=0.31\linewidth]{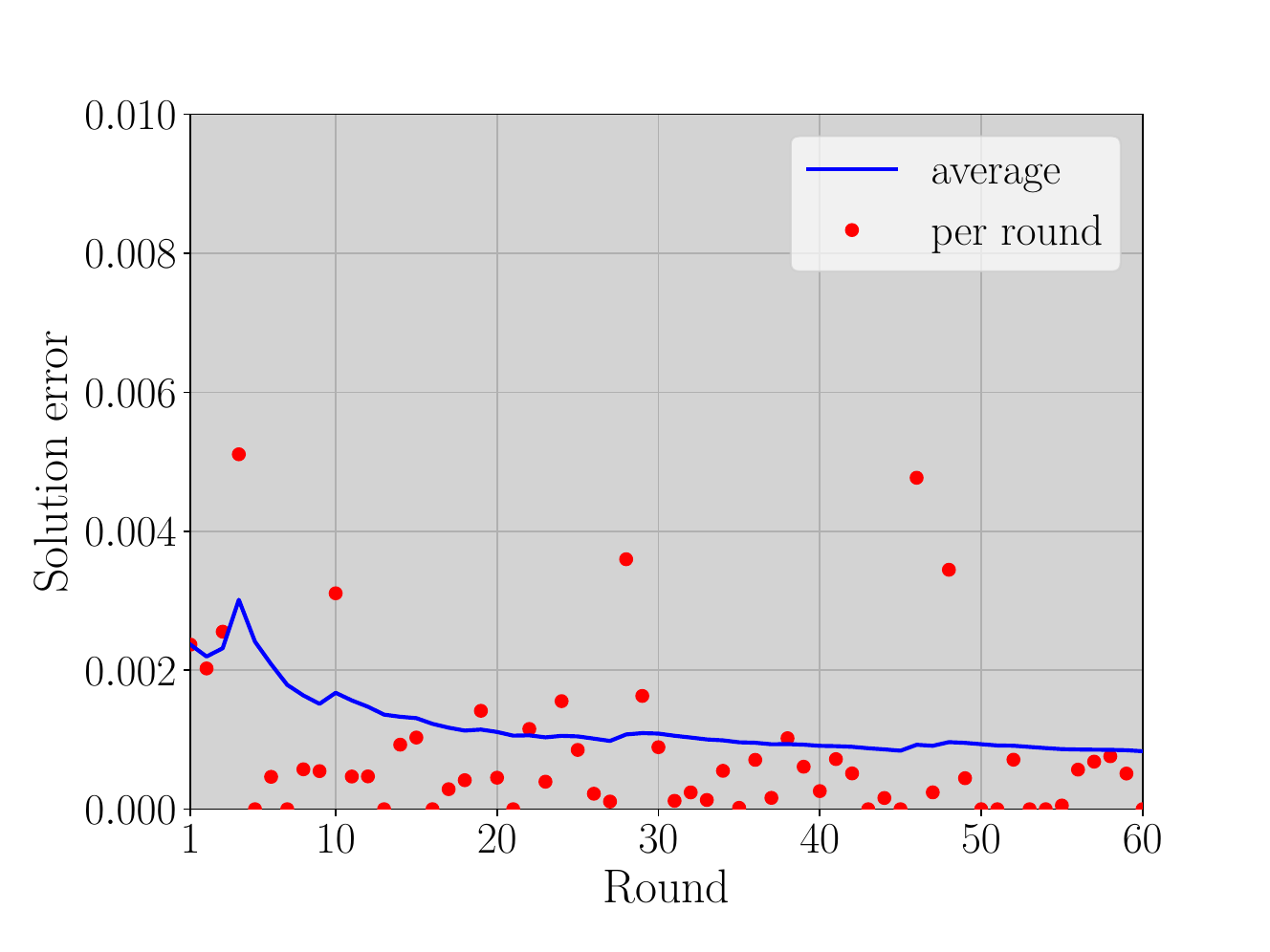}
		\label{Fig:ShortestPathProblem-SC}
	}
	\hfill
	\subfloat[Abrupt congestion]{
		\includegraphics[width=0.31\linewidth]{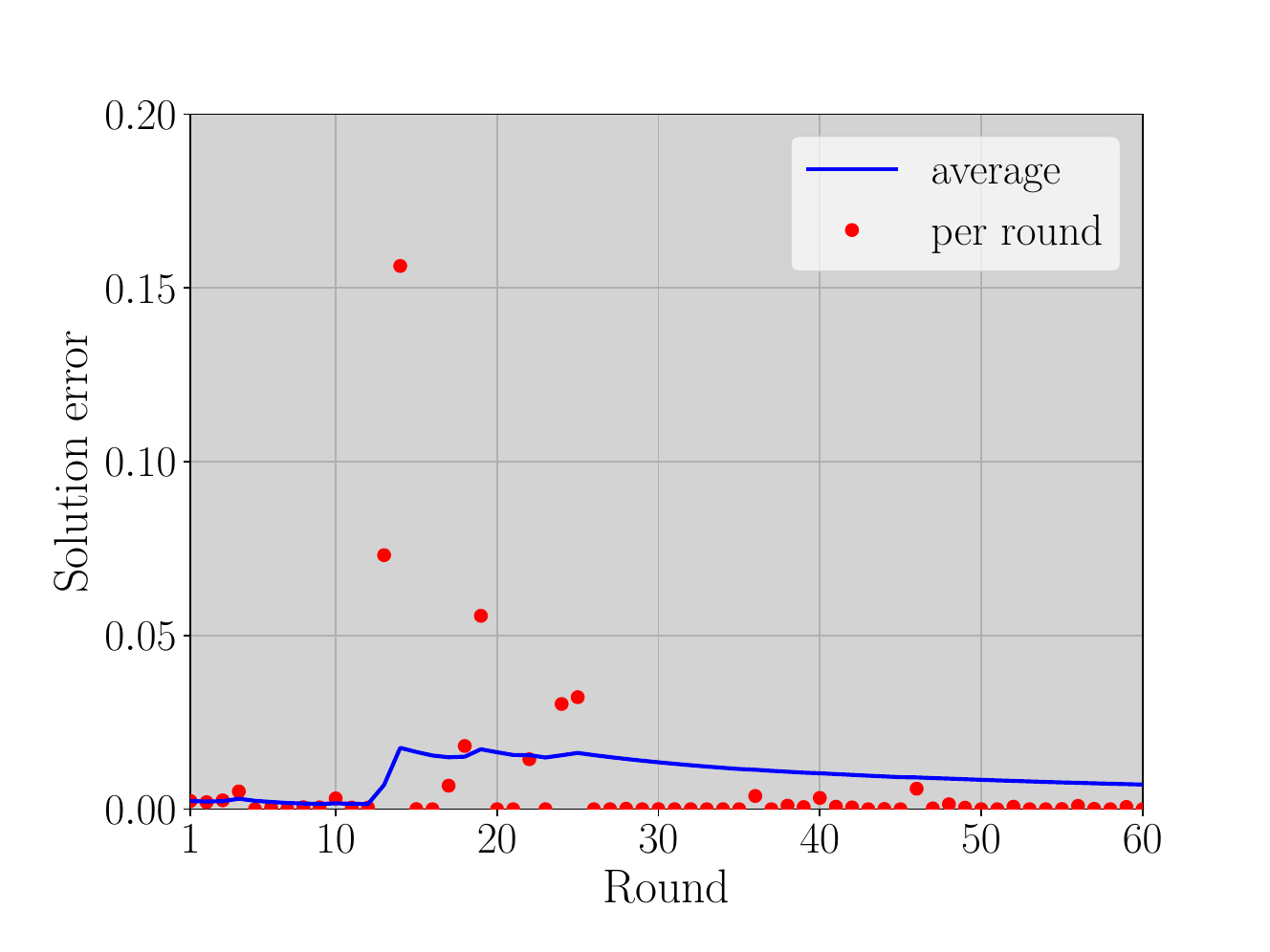}
		\label{Fig:ShortestPathProblem-FC}
	}
	\hfill\null
	\caption{Solution error for each round (\emph{red}) and averaged up to the current round (\emph{blue}) for the shortest-path problem with $ T = 60 $ drivers under the different congestion scenarios for MWU}
	\label{Fig:ShortestPathProblem}
\end{figure}

Figure~\ref{Fig:ShortestPathProblem} depicts the solution error
for the three cases.
From Figure~\ref{Fig:ShortestPathProblem-Original}, we can see
that we achieve a very good convergence of the error
already with as few as 60~observations.
For a gradually building and receding congestion,
we see in Figure~\ref{Fig:ShortestPathProblem-SC}
that the performance of our algorithm
nearly does not deterioriate at all,
which means the learned objective
quickly adapts to the slowly changing travel times.
In the case of abrupt congestions, we see the error spiking sharply
at the point where the congestions begin,
but quickly declining afterwards.
After 60~rounds, the average solution error
is still about 8~times higher
than in the unpertubed case,
which means it takes some time to recover.
This behaviour is expected, as our theoretical results
predict the washing out of any error~$d$ in the average regret
at a rate of $ \OO(d/\sqrt{T}) $,
cf.\ the discussion of Lemma~\ref{Lem:Guarantee_OGD_Dynamic}.
Nevertheless, the solutions we obtain from round~26 on
have a loss which is significantly below the average regret.
To summarize, this experiment shows
that we can learn to take efficient decisions
already with few observations at hand
and that we can quite easily cope with
small continuing changes or big but seldom changes
in the target objective.

\subsection{Learning Optimal Delivery Routes}
\label{Sec:Application.PCTSP}
Finally, we demonstrate the potential of OGD
to learn objective coefficients with mixed signs.
For this purpose, we consider a simple delivery problem
where a company has certain customers which it serves from its depot.
The delivery network is given
by a complete undirected graph $ G = (V \cup \set{v_0}, E) $
with a designated node~$ v_0 $ representing the depot
and the nodes in~$V$ representing the customers.
A delivery route consists of a Hamilton tour comprising the depot
as well as the subset of customers the company decides to serve
in the given time step and in its preferred order.
We assume that each edge $ e \in E $
has a cost~$ c_{te} $ for traversing it,
and that each customer $ v \in V $
brings a revenue $ r_{tv} $ if served, 
both unknown and varying over time steps $ t \in [T] $.
This uncertainty might for example arise
from different traffic scenarios
affecting delivery costs as well as different customer demand scenarios.
Altogether, the company wants to solve
the following profitable tour problem $ \textrm{OPT}(p_t) $,
a prize-collecting variant of the travelling salesman problem,
for each $ t \in [T] $:
\begin{mip}
	\objective{\max}{\sum_{v \in V} r_{tv} y_{tv} - \sum_{e \in E} c_{te} x_{te}}\\
	\stconstraint{\sum_{e \in \delta(v)} x_{te}}{=}{2y_{tv}}{(\forall v \in V \cup \set {v_0})}\\[-0.5\baselineskip]
	\constraint{\sum_{\substack{e = (i, j) \in E:\\i, j\in S}} x_{te} + y_l}{\leq}{1 + \sum_{i \in S \setminus \set{k}} y_i}{\\(\forall S \subset V \cup \set{v_0}, 2 \leq \card{S} \leq \card{V} - 1)\\(\forall k \in S)(\forall l \notin S)}\\
	\constraint{y_{tv}}{\leq}{y_{tv_0}}{(\forall v \in V)}\\
	\constraint{x_t}{\in}{\F^{\card{E}}}{}\\
	\constraint{y_{tv}}{\in}{\F^{\card{V} + 1},}{}
\end{mip}
where $x$ models the chosen edges
and $y$ the chosen customers.

For our computational experiment,
we use instance~\emph{berlin52-gen3-50}
from Gorka Kobeaga's OPLib, see~\cite{OPLib},
where we scale up the customer revenues by a factor of~4
to yield a non-trivial trade-off against the edge costs.
Then, in each time step $ t \in [T] $,
we draw the edge costs uniformly from an interval of $ \pm $ 10\%
and the customer revenues from an interval of $ \pm $ 20\%
around these basic values
which we interpret as the expected costs and revenues respectively.
Thus, we want to learn to distinguish
the more efficient routes from the less efficient routes
and the more profitable customers from the less profitable customers.
We choose~$F$ to be the unit cube
and~$ c_1 $ as the zero vector to initialize the OGD algorithm.

Figure~\ref{Fig:TSP} shows the results of the experiment.
\begin{figure}[h]
	\hfill
	\subfloat[Solution error]{
		\includegraphics[width=0.31\linewidth]{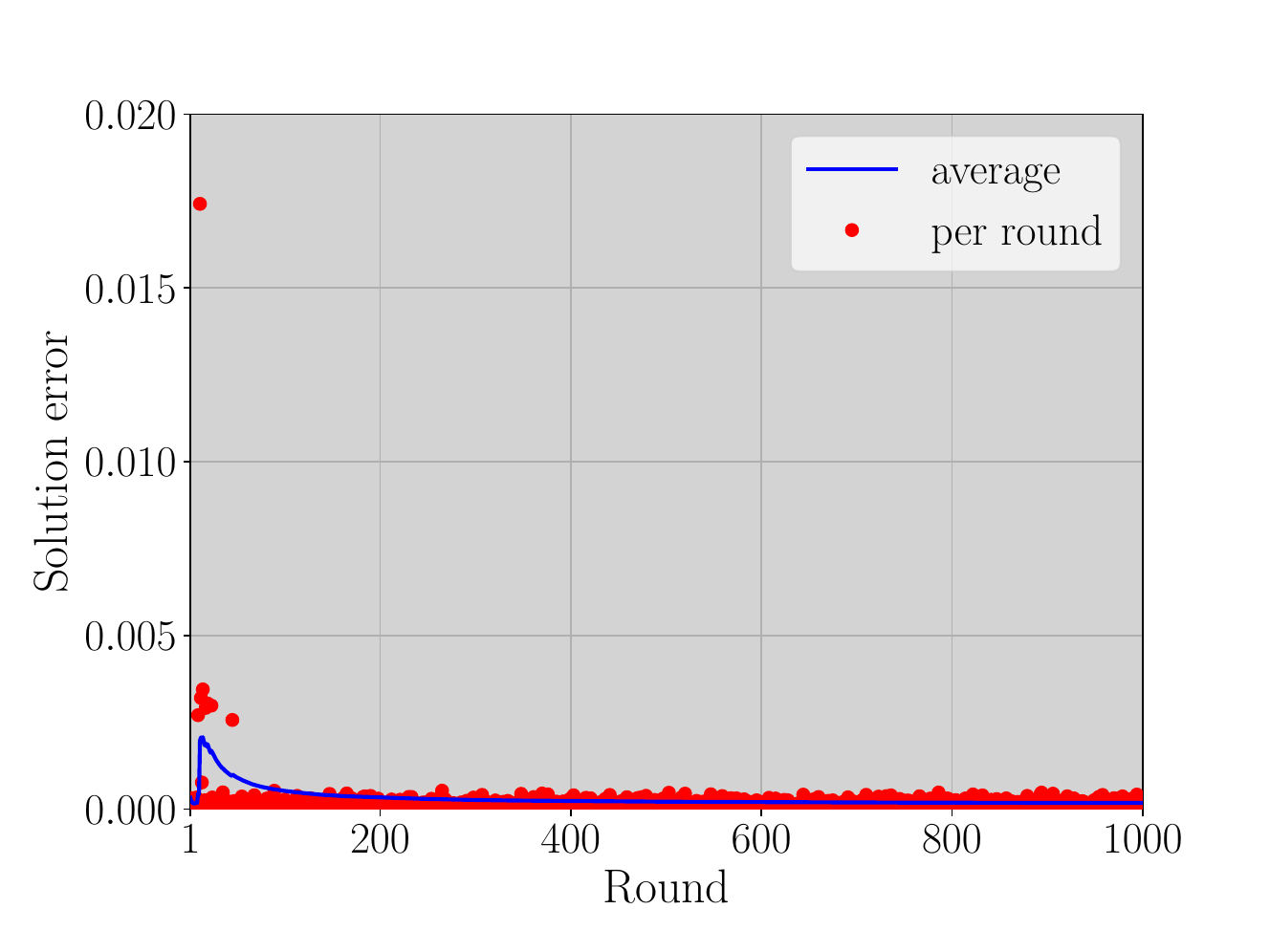}
		\label{Fig:TSP-Solution-Error}
	}
	\hfill
	\subfloat[$1$-norm objective distances]{
		\includegraphics[width=0.31\linewidth]{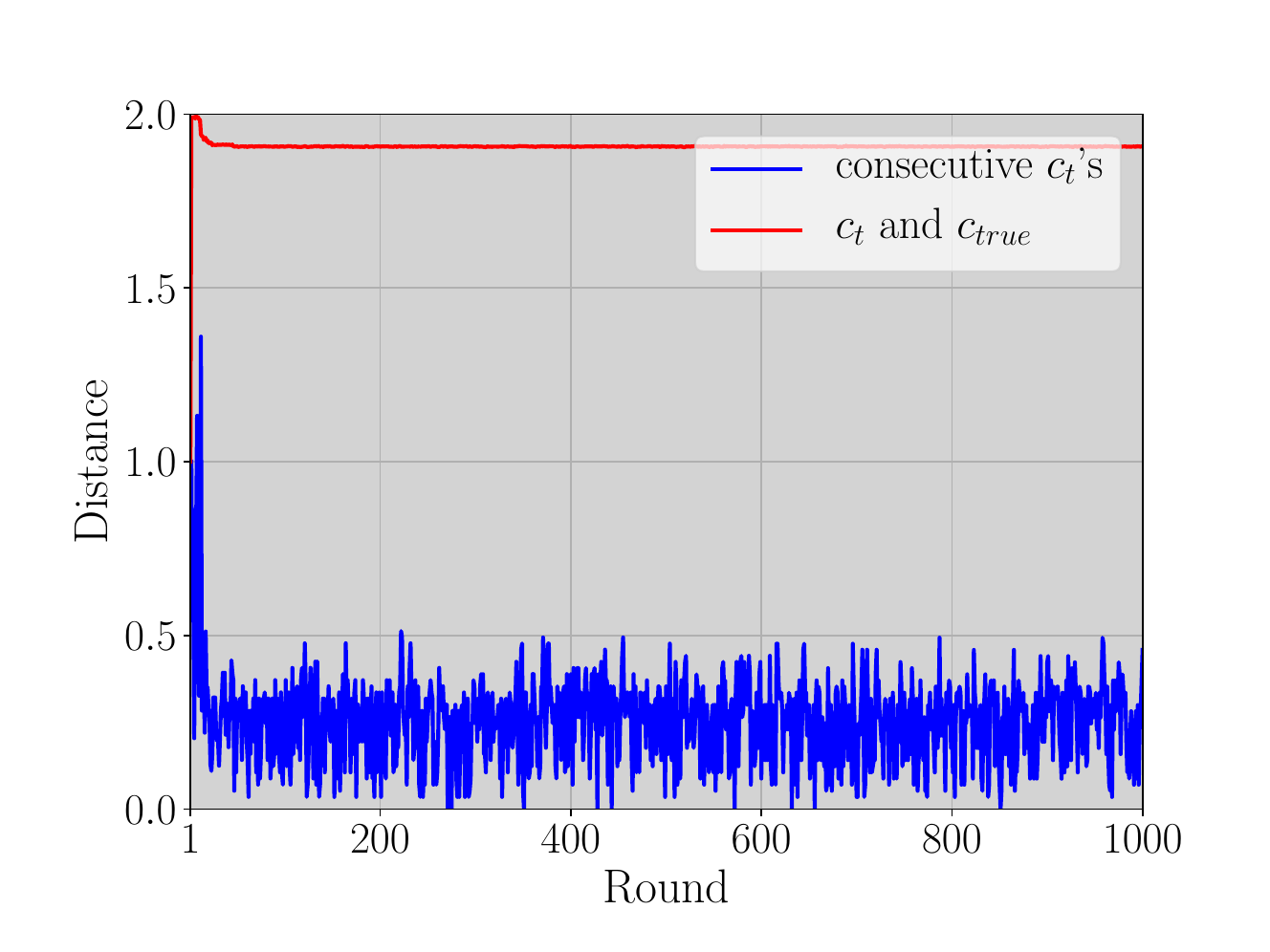}
		\label{Fig:TSP-Distance}
	}
	\hfill
	\subfloat[Convergence]{
		\includegraphics[width=0.31\linewidth]{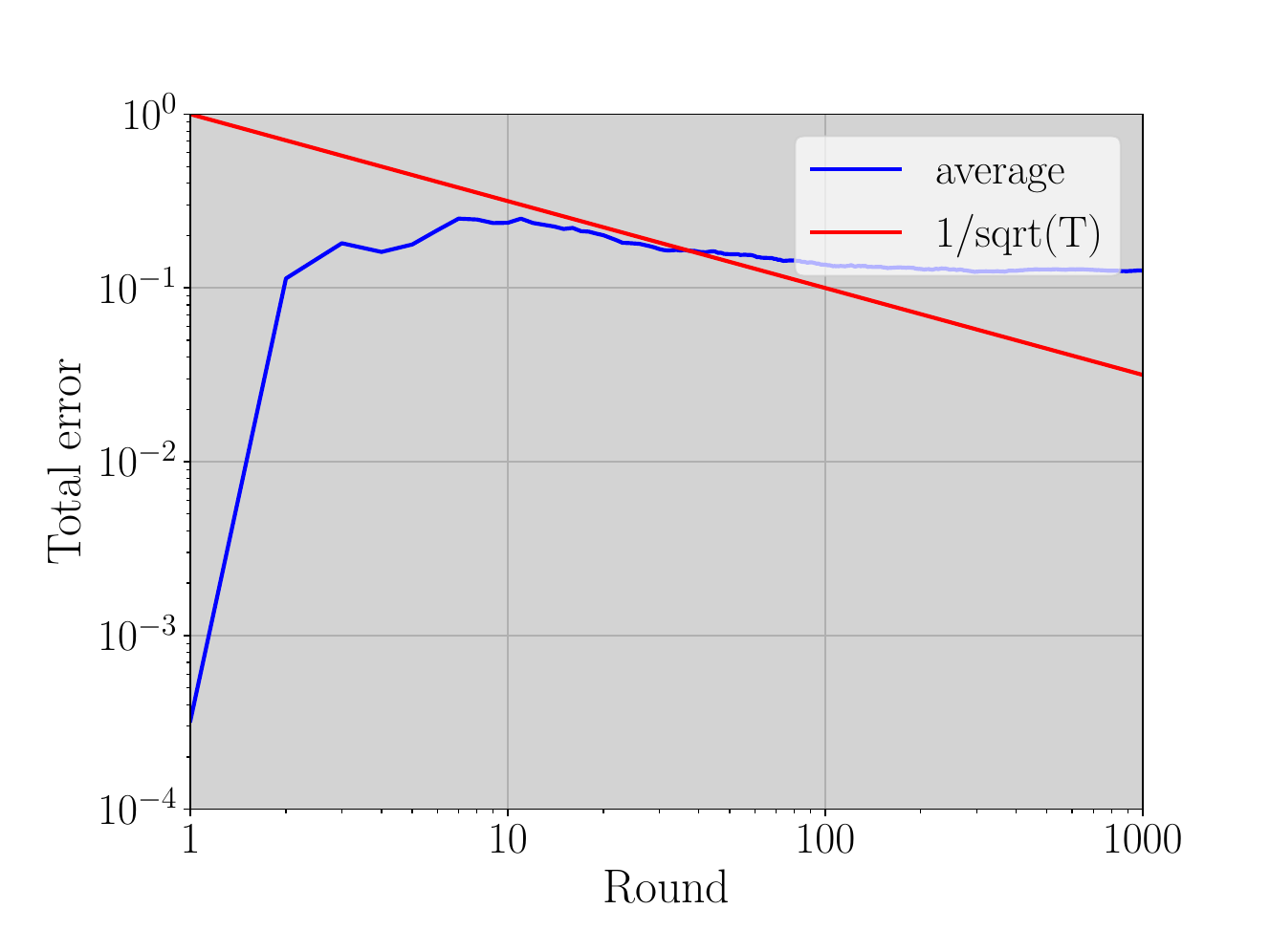}
		\label{Fig:TSP-Convergence}
	}
	\hfill\null
	\caption{Solution error, objective distances and convergence for the prize-collecting travelling salesman problem for OGD}
	\label{Fig:TSP}
\end{figure}
In Figure~\ref{Fig:TSP-Solution-Error},
we see that the solution error falls relatively quickly
after few iterations
as for the other two problems.
However, now the average regret does not converge to zero,
but rather to the variance in the true objective
which is sampled anew in each round.
This is explainable by the dynamic nature
of the objective function to be learned,
\cf Lemma~\ref{Lem:Guarantee_OGD_Dynamic}
and the corresponding discussion.
We obtain a kind of \qm{robust} objective,
which tries to explain the observations
produced by an actually non-constant target objective.




\section{Final Remarks}
\label{Sec:final-remarks}

We saw that algorithms derived from online-learning methods
are capable of learning objective functions
from optimal (and close-to-optimal) observed decisions,
given knowledge of the underlying feasible set in each case.
Especially, our framework provides a new online analogue
of the established concept of inverse optimization,
where the unknown objective may be determined
via multiple samples and over time,
with decisions taken in between the observations.
We were able to prove that our algorithms achieve low errors
with respect to the decisions based on the learned objectives.
Furthermore, they are applicable in more general situations
than previous methods,
which required convexity of the feasible region,
and they are usable in situations where the observations arrive online
as a data stream, allowing to learn objectives which change over time.
In computational experiments, we demonstrated that they quickly converge
in practical settings, learning objectives
which explain the observed decisions very well,
including when testing them out-of-sample.
The practical importance of this approach
is evident when considering that much effort is made
to come to procedures that automate model building from data.
An important question in this respect is to what extent
our framework can be extended to the learning of constraints
(and objective functions simultaneously).

\section*{Acknowledgements}
\label{sec:acknowledgements}

This research was partially supported by
NSF CAREER award CMMI-1452463
and by the Bavarian Ministry of Economic Affairs, Regional Development and Energy
through the Center for Analytics -- Data -- Applications (ADA-Center)
within the framework of \qm{BAYERN DIGITAL~II}.

\bibliographystyle{apalike}
\bibliography{Literature2}

\appendix
\section{Counterexample for the Follow-the-Leader Algorithm}
In the following, we give a counterexample
where the follow-the-leader scheme
presented in Section~3.4
does not yield a strategy with sublinear regret.
\begin{example}
	\label{exp:ggbsp}
	Consider the following series of feasible sets $ X(p_t) $
	over rounds $ t = 1, \ldots, T $ for some time horizon $ T > 0 $:
	\begin{equation*}
		X(p_t) \coloneqq
			\set{\tp{(x_1, x_2)} \in \R_+^2
				\mid x_2 = 1 - \frac{1}{a_t}}, \quad \text{where}\quad a_t \coloneqq \sum_{\tau \in[t]} \left(\frac12\right)^{\tau - 1}.
	\end{equation*}
	They represent line segements
	going from $ \tp{(0, 1)} $ to $ \tp{(a_t, 0)} $
	in round~$t$, approaching the line segment
	from $ \tp{(0, 1)} $ to $ \tp{(2, 0)} $
	with $ T \to \infty $.
	The true objective to learn
	shall be $ \ctrue = \tp{(0, 1)} $
	and $ F = \Delta_2 $.
	In this setting,
	the optimal solution of the adversary
	will always be $ \tp{(0, 1)} $.
	Assuming that the player chooses objective $ c_1 = \tp{(1, 0)} $
	in the first round,
	the solution in that round will be $ \sol{x}_t = \tp{(1, 0)} $, 
	while the adversary will play $ x_t = \tp{(0, 1)} $.
	In the next round, the player will thus have to play an objective
	whose angle to $ (1, 0) $ is at least as high
	as that between $ \tp{(1, 1)} $ and $ \tp{(1, 0)} $,
	\ie a $ c_2 = \tp{(\gamma, 1 - \gamma)} $
	with $ \frac12 \leq \gamma \leq 1 $.
	If, in general,
	the player chooses $ c_t = \tp{(1 / (1 + a_{t - 1}),
		a_{t - 1} / (1 + a_{t - 1}))} $
	in each round~$ t > 1 $,
	the point $ \tp{(0, 1)} $ will always be an optimal solution
	for all previously observed feasible sets~$ X(p_t) $,
	but $ \tp{(a_t, 0)} $ will be picked
	as the optimal solution in the current round.
	The player thus yields a regret of
	\begin{align*}
		\sum_{t \in[T]} \sp{c_t}{(\sol{x}_t - x_t)}
			- \sum_{t \in[T]} \sp{\ctrue}{(\sol{x}_t - x_t)}
			&= \sum_{t \in [T]} \tp{\begin{pmatrix}
				\frac{1}{1 + a_{t - 1}}\\
				\frac{a_{t - 1}}{1 + a_{t - 1}}
				\end{pmatrix}}
				\begin{pmatrix}
					a_t\\-1
				\end{pmatrix}
				- \sum_{t \in [T]} \tp{\begin{pmatrix}
					0\\1
					\end{pmatrix}}
					\begin{pmatrix}
						a_t\\-1
					\end{pmatrix}\\
			&= \sum_{t \in [T]} \frac{a_t - a_{t - 1}}{1 + a_{t - 1}}
				- \sum_{t \in [T]} (-1)\\
			&= \sum_{t \in [T]} \frac{1}{2(1 + a_{t - 1})} + T
			\geq \sum_{t \in [T]} \frac16 + T
			= \frac{7T}{6},
	\end{align*}
	such that neither of objective error, solution error or total error
	converges to zero on average with $ T \to \infty $.
\end{example}

\section{Additional Results for the Knapsack Problem}

Here, we give additional results
for the linear and integer knapsack problem
studied in Section~\ref{Sec:Application.Knapsack}

In Table~\ref{Tab:LinearKnapsack},
we show statistics on the computational results
for the linear knapsack problem.
It shows the arithmetic means and standard deviations
of the average errors after 5, 50 and 500~iterations
for each of the algorithms,
with the arithmetic mean taken over all 50~instances.
\setlength\tabcolsep{4.5pt}
\begin{table}[!htbp]
	\begin{center}
		\begin{tabular}{l|rrr|rrr|rrr}
			Algorithm & \multicolumn{3}{c|}{MWU} & \multicolumn{3}{c|}{OGD}
				& \multicolumn{3}{c}{LP}\\
			\# Rounds		& 5	& 50	& 500	& 5	& 50	& 500	& 5	& 50	& 500\\
			\midrule
			$\varnothing$ objective error	& 0.21 & 0.03 & 0.01	   &	0.23 & 0.04 & $<$0.01 	   &	0.15 & 0.03 & $<$0.01 \\
			$\sigma$ objective error	 & 0.03 & $<$0.01 & $<$0.01	         &	0.04 & $<$0.01 & $<$0.01 	     &0.08 & 0.01 & $<$0.01 \\
			$\varnothing$ solution error	 & 0.06 & 0.02 & $<$0.01	    &	0.06 & 0.01 & $<$0.01    &	0.29 & 0.16 & 0.03\\
			$\sigma$ solution error	 & 0.01 & $<$0.01& $<$0.01	           &	0.01 & $<$0.01 & $<$0.01	  &	0.11 & 0.04 & 0.02 \\
			$\varnothing$ total error		   & 0.27 & 0.05 & 0.01         &	0.28 & 0.05 & 0.01	  &	0.44 & 0.18 & 0.04\\
			$\sigma$ total error         & 0.05 & 0.01 & $<$0.01	       &	0.05 & 0.01 & $<$0.01 	&	0.11 & 0.04 & 0.02 \\
			\bottomrule
		\end{tabular} 
	\end{center}
	\caption{Mean and standard deviation of the average errors for the linear knapsack problem with $ n = 50 $ items and $ T = 500 $ observations for each algorithm, rounded to two decimals, with the arithmetic mean taken over 50~runs on random instances}
	\label{Tab:LinearKnapsack}
\end{table}
As can be seen, after few iterations
we obtain error values close to zero for MWU und OGD,
with standard deviations lowering quickly
over the number of rounds played.
Both perform notably better in early iterations than LP,
which is, however, able to catch up,
as after a couple of rounds
the total error is basically always zero.

In Figure~\ref{Fig:LinearKnapsack11},
the objective error, solution and total error,
are shown for MWU, OGD and LP.
\begin{figure}[h]
	\hfill
	\subfloat[Objective error for MWU]{
		\includegraphics[width=0.31\linewidth]{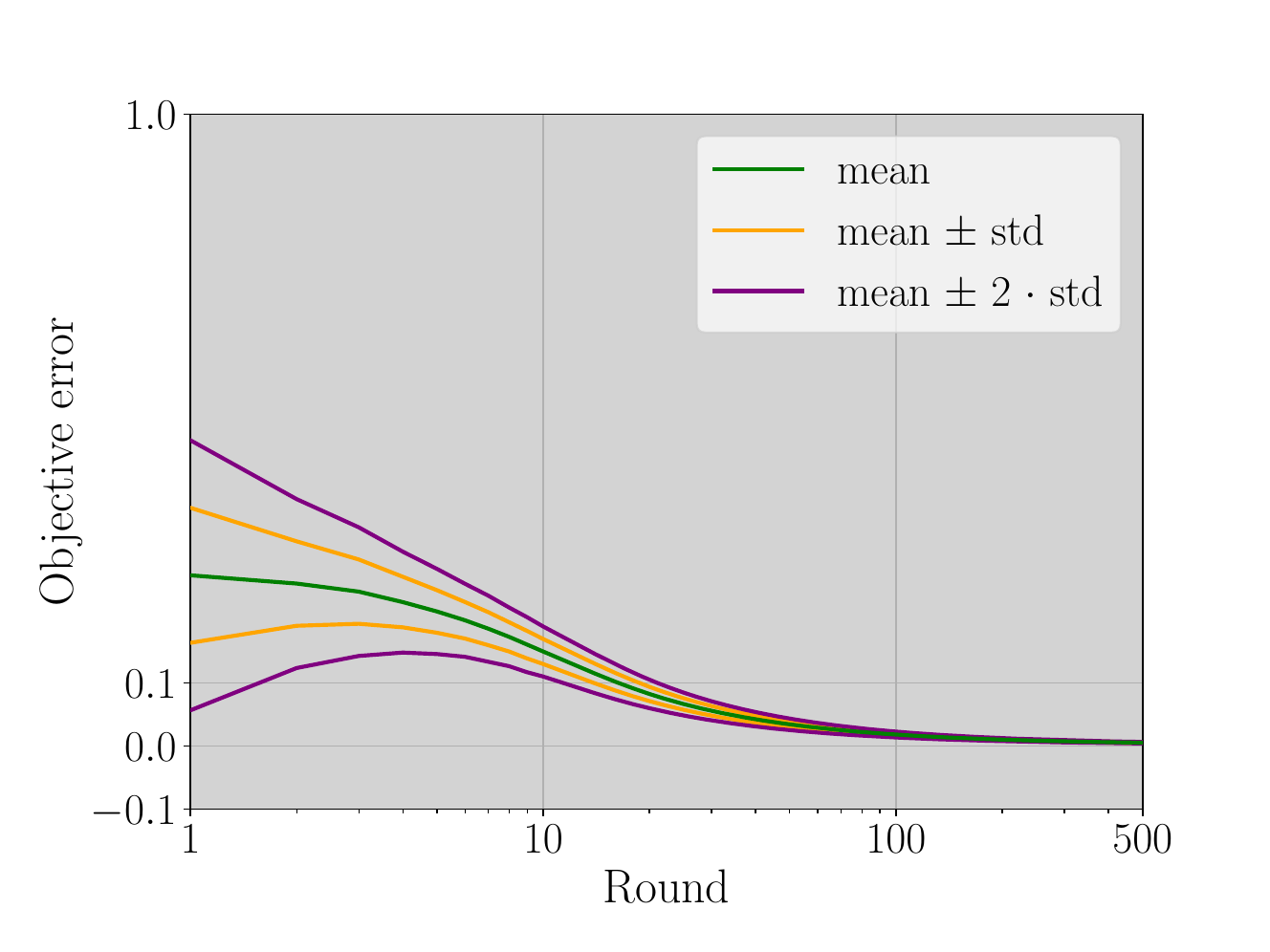}
	}
	\hfill
	\subfloat[Objective error for OGD]{
		\includegraphics[width=0.31\linewidth]{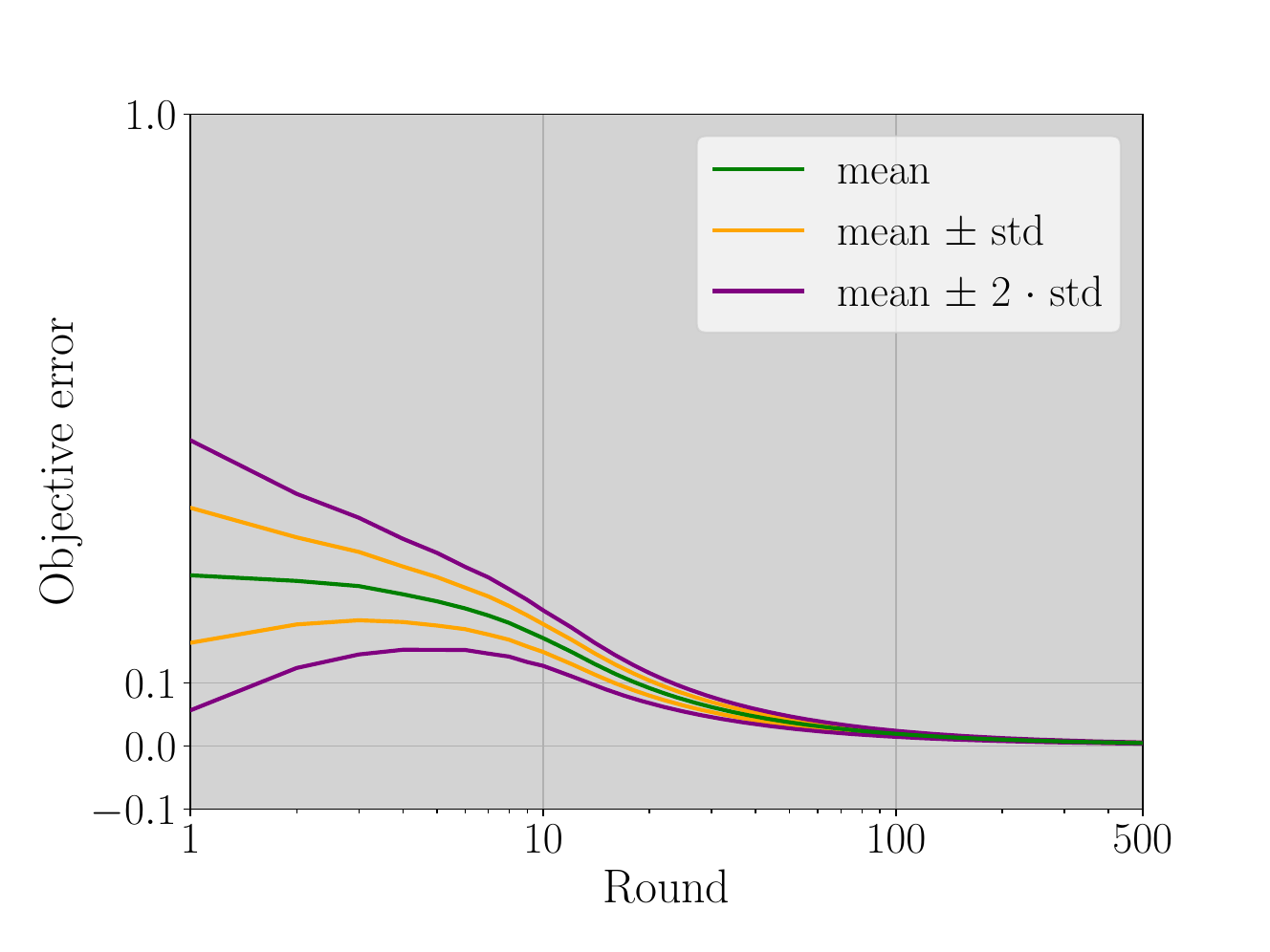}
	}
	\hfill
	\subfloat[Objective error for LP]{
		\includegraphics[width=0.31\linewidth]{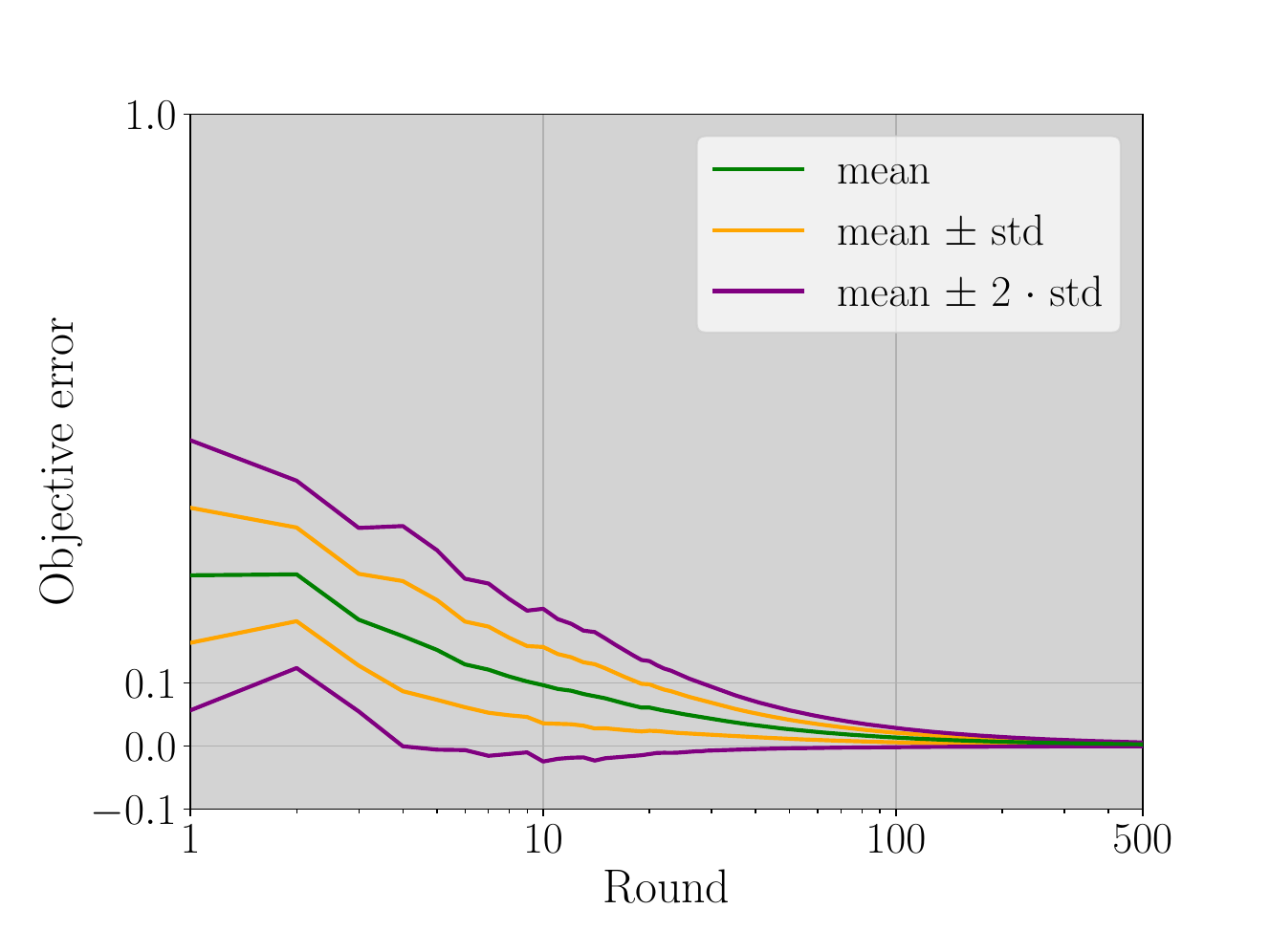}
	}
	\hfill\null
	
	\hfill
	\subfloat[Solution error for MWU]{
		\includegraphics[width=0.31\linewidth]{\graphicpathknp/solution_error_mwa.pdf}
	}
	\hfill
	\subfloat[Solution error for OGD]{
		\includegraphics[width=0.31\linewidth]{\graphicpathknp/solution_error_ogd.pdf}
	}
	\hfill
	\subfloat[Solution error for LP]{
		\includegraphics[width=0.31\linewidth]{\graphicpathknp/solution_error_lp.pdf}
	}
	\hfill\null
	
	\hfill
	\subfloat[Total error for MWU]{
		\includegraphics[width=0.31\linewidth]{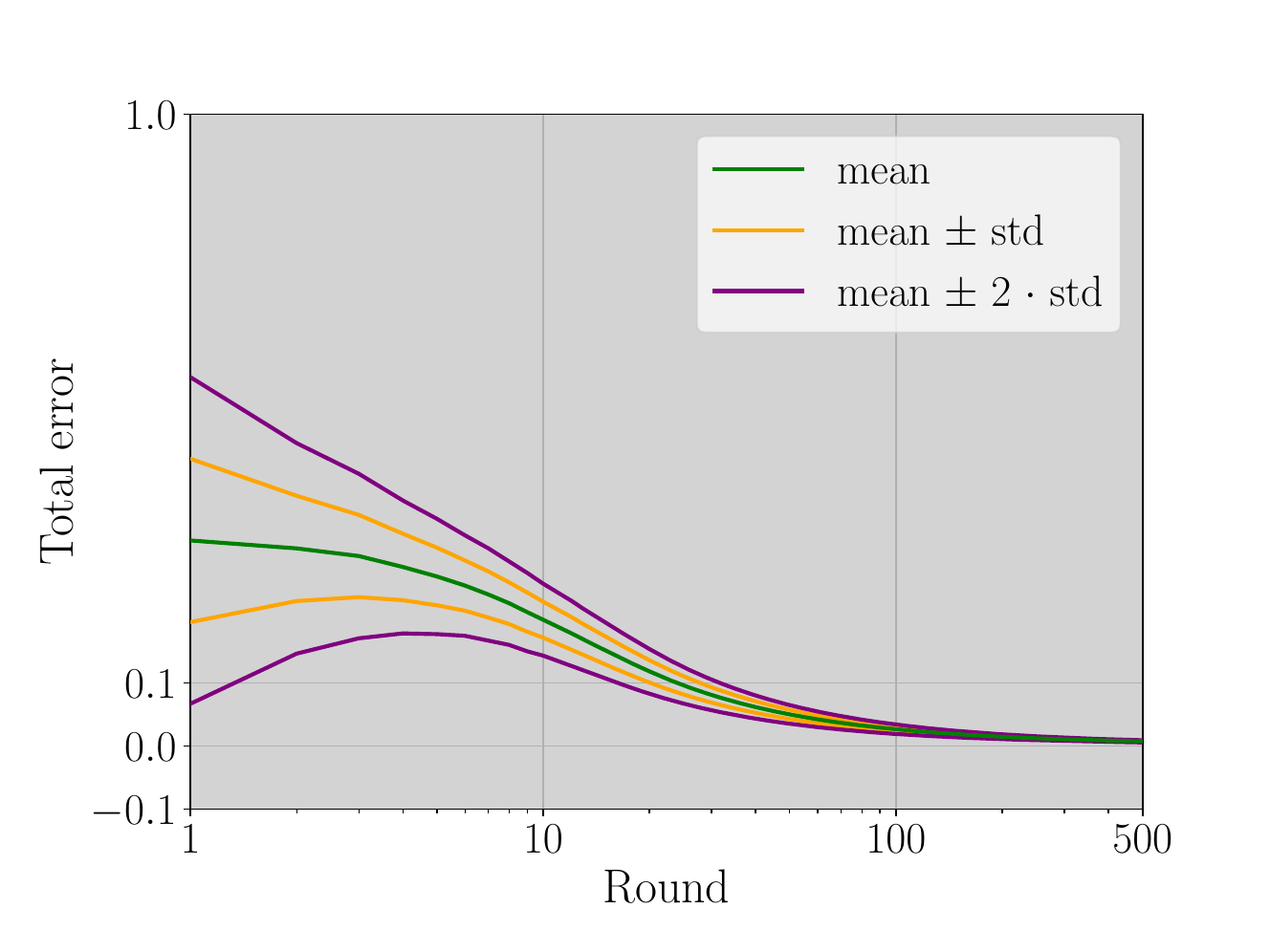}
	}
	\hfill
	\subfloat[Total error for OGD]{
		\includegraphics[width=0.31\linewidth]{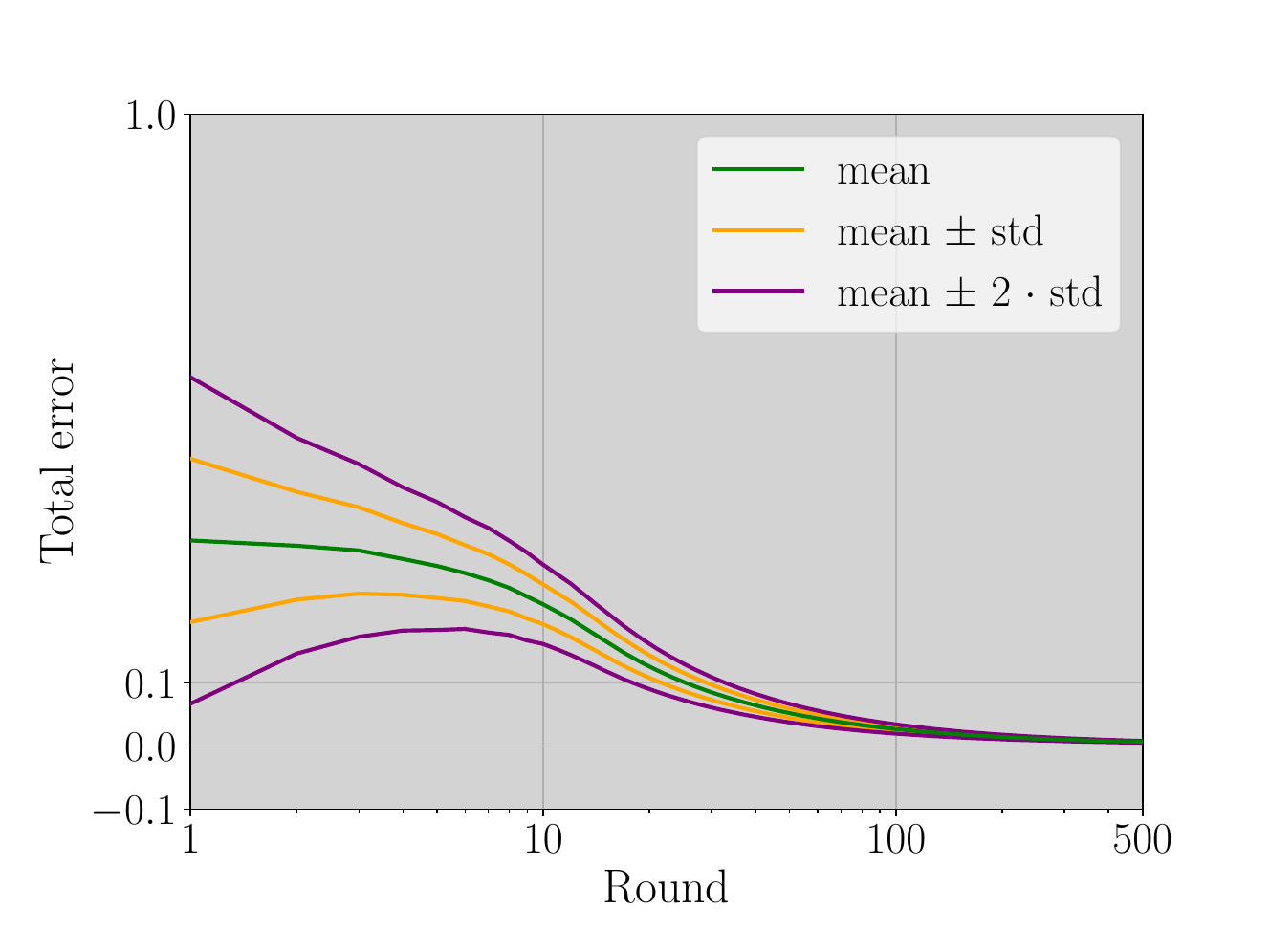}
	}
	\hfill
	\subfloat[Total error for LP]{
		\includegraphics[width=0.31\linewidth]{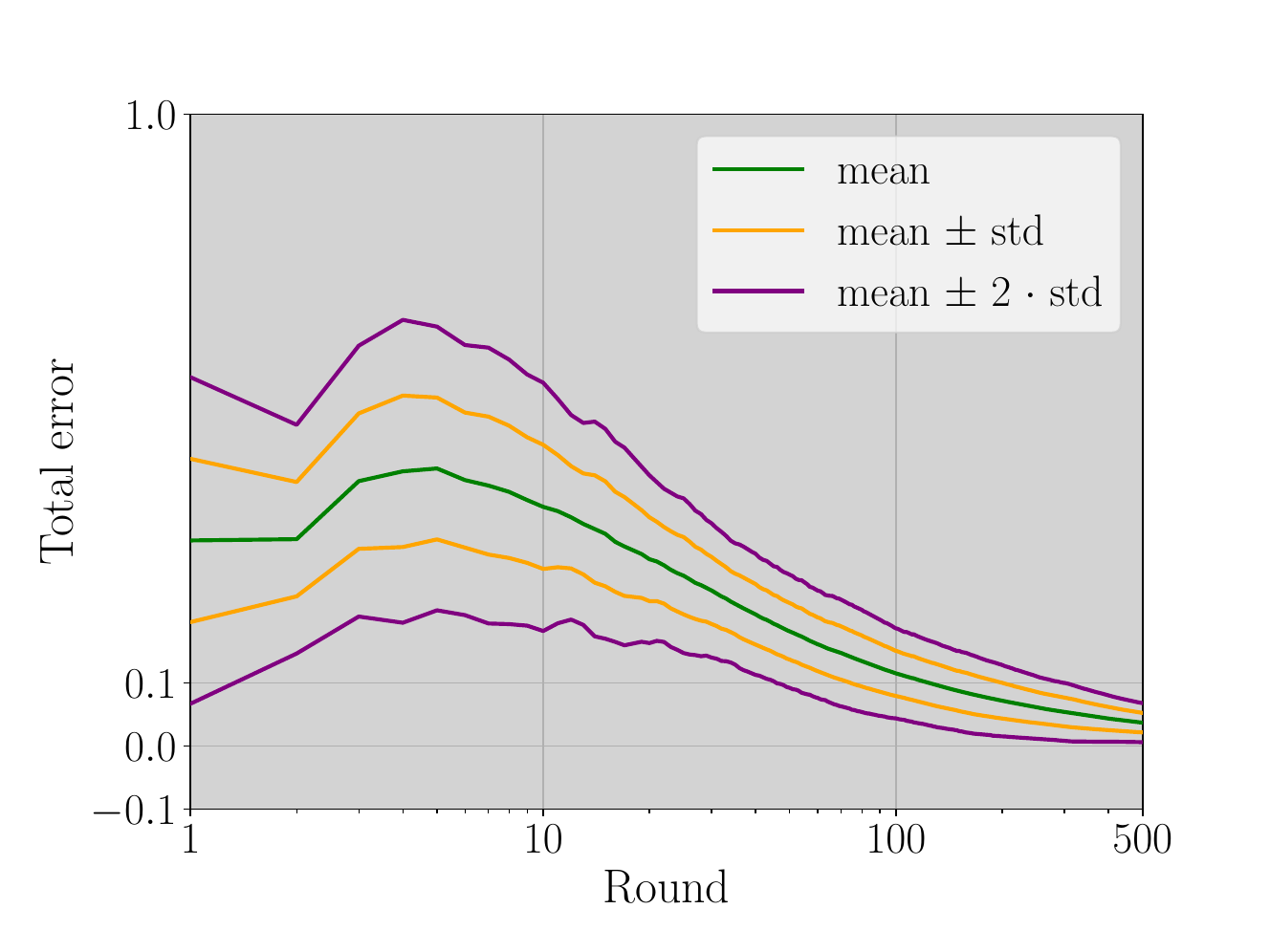}
	}
	\hfill\null
	\caption{The arithmetic means and standard deviations of the different error types in each round (\emph{red}) and averaged up to the current round (\emph{blue}) for the linear knapsack problem with $ n = 50 $ items over $ T = 500 $ rounds for each of the three algorithms, with the arithmetic mean taken over 500~runs, on a doubly symmetric-logarithmic plot}
	\label{Fig:LinearKnapsack11}
\end{figure}
As can be seen, in general after few iterations
the error values reside close to zero,
and the standard deviations lower quickly
with the number of rounds played.
The pictures for MWU and OGD are almost indistinguishable,
while the mean average errors for LP 
are on a somewhat larger scale
due to the tendentially higher errors
in the very first rounds.

Figure~\ref{Fig:IntegerKnapsack1} shows plots
of the solution error, depicting both the error in a given round~$t$
and the average total error up to round~$t$
for the integer knapsack problem.
\begin{figure}[!h]
	\hfill
	\subfloat[MWU]{
		\includegraphics[width=0.31\linewidth]{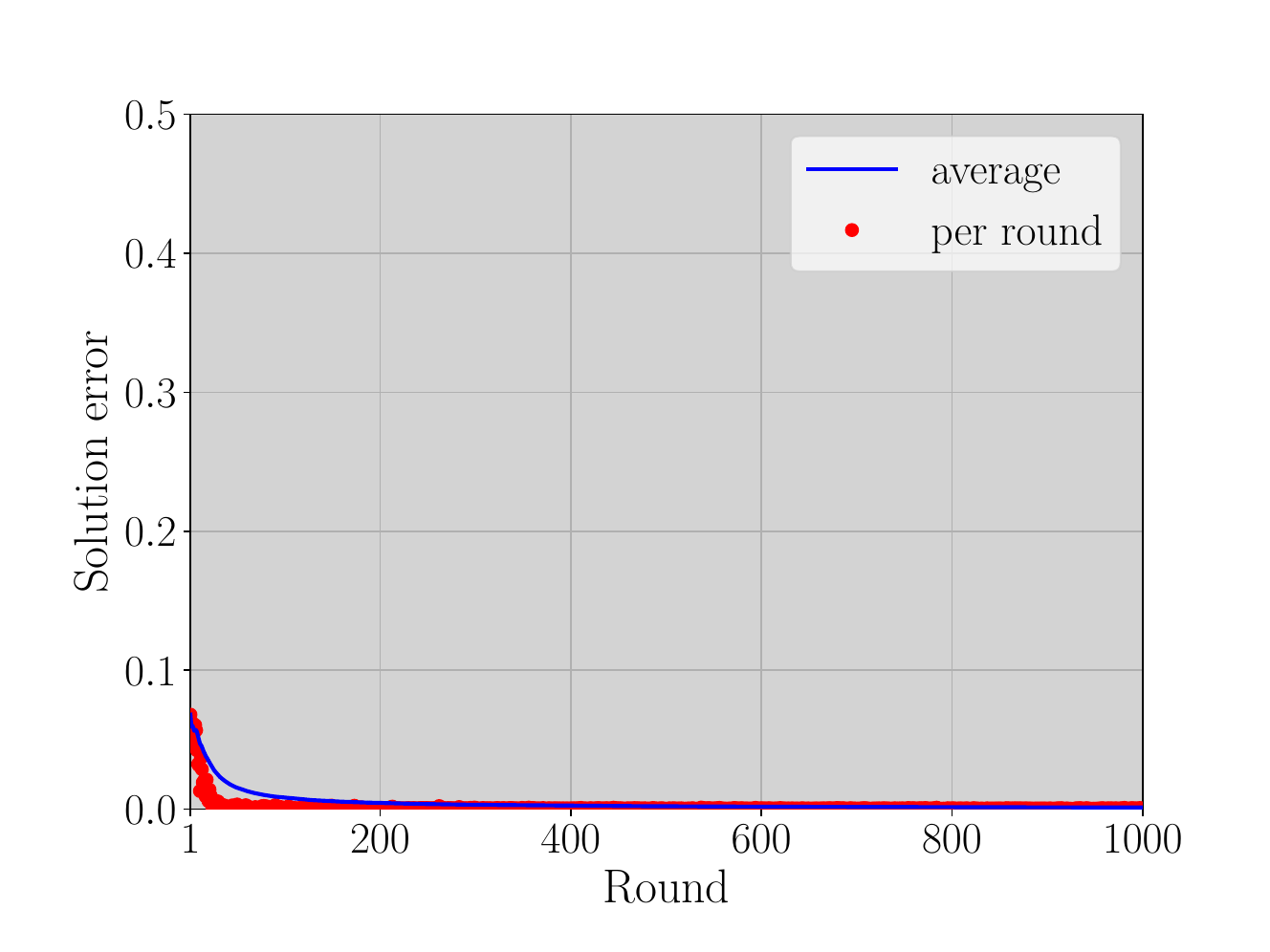}
	}
	\hfill
	\subfloat[OGD, fixed learning rate]{
		\includegraphics[width=0.31\linewidth]{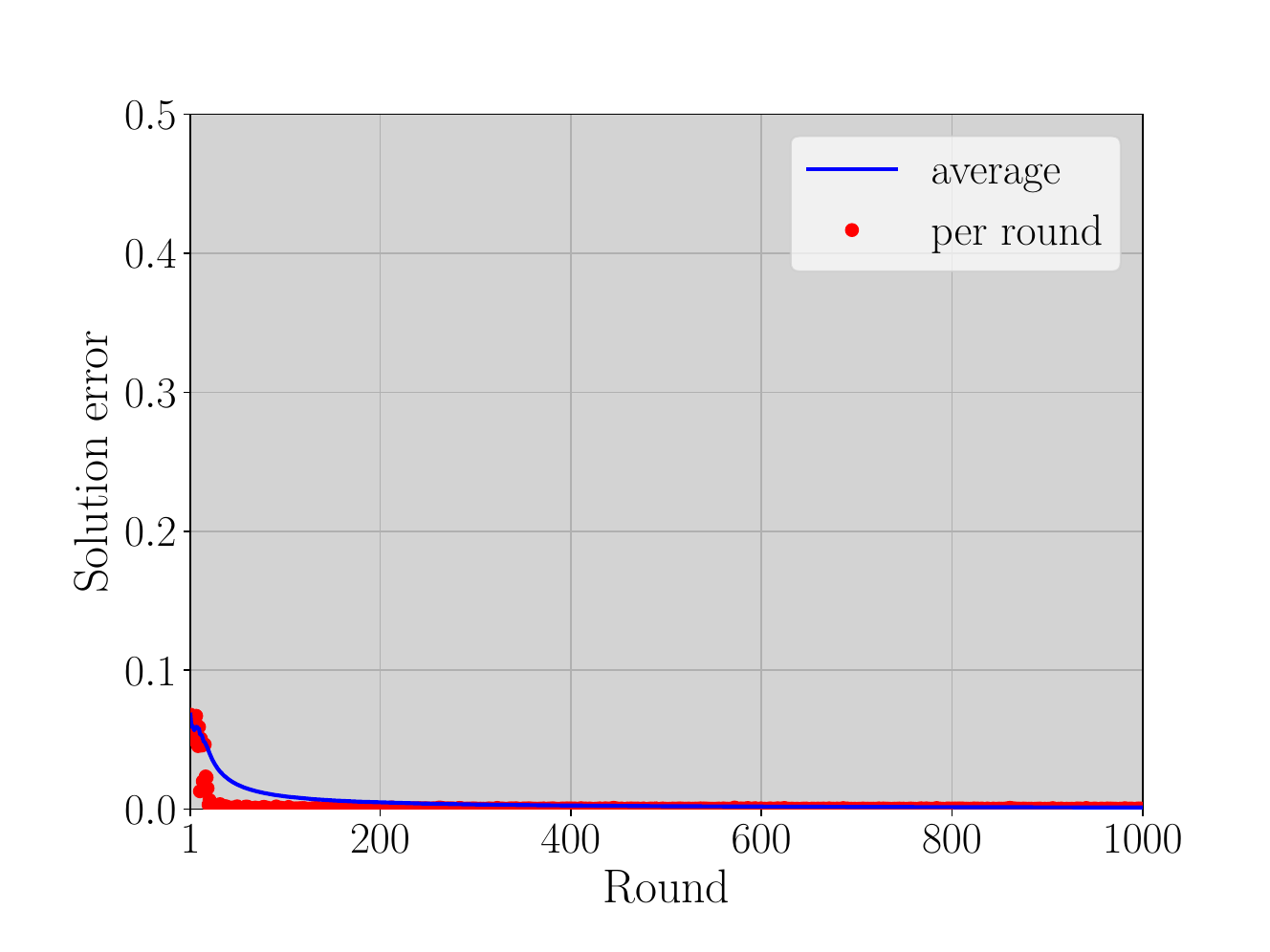}
	}
	\hfill
	\subfloat[OGD, dynamic learning rate]{
		\includegraphics[width=0.31\linewidth]{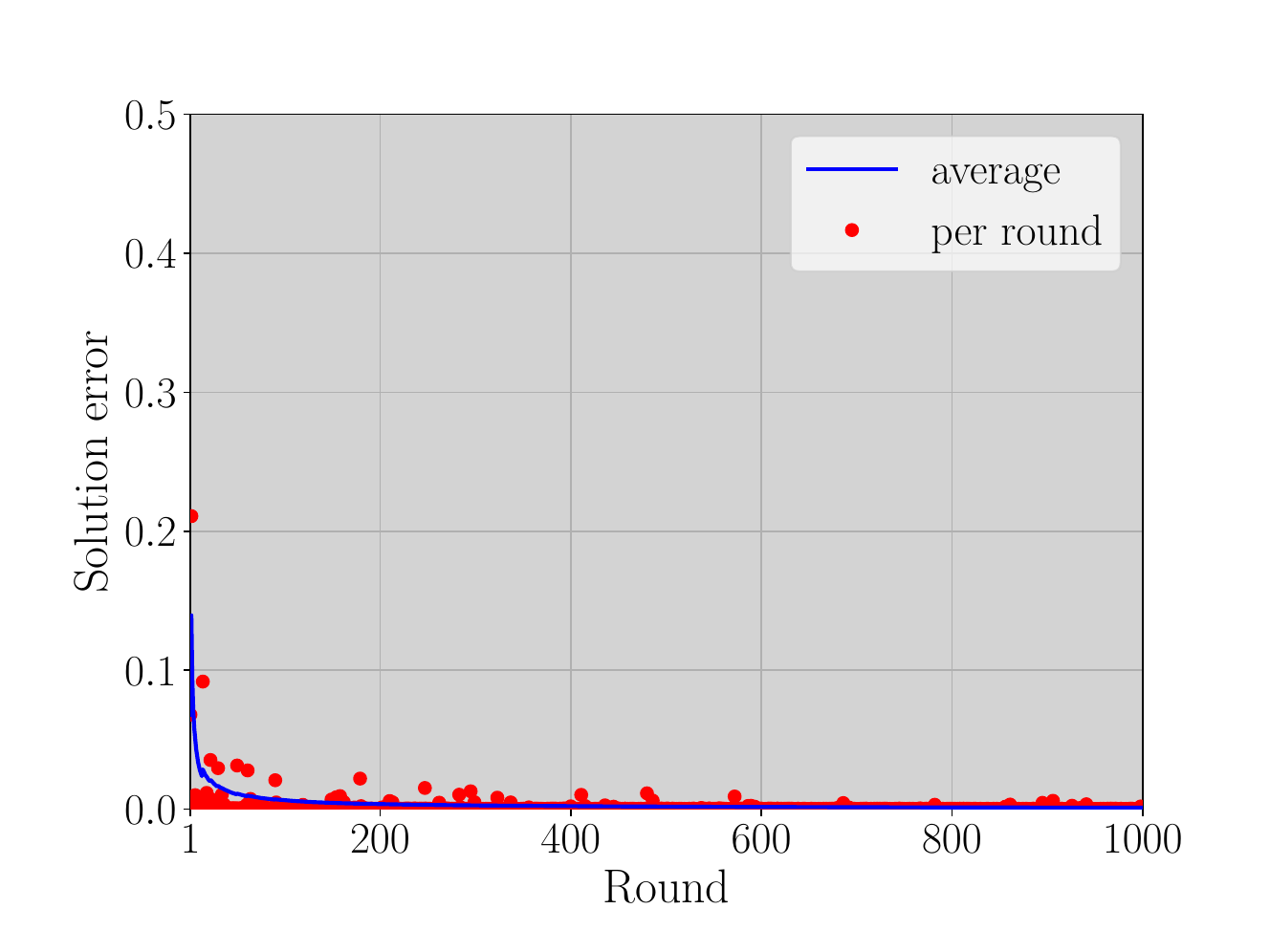}
	}
	\hfill\null
	
	\caption{The solution error for each round (\emph{red}) and averaged up to the current round (\emph{blue}) for the integer knapsack problem with $ n = 1000 $ items over $ T = 1000 $ rounds for MWU and two variants of OGD}
	\label{Fig:IntegerKnapsack1}
\end{figure}
Again, the average errors fall quickly,
and only few of the errors in individual rounds,
mainly in the first rounds, deviate beyond the average.
\begin{figure}[h]
	\hfill
	\subfloat[MWU]{
		\includegraphics[width=0.31\linewidth]{\graphicpathknp/knp_t1000_s1_n1000_IP_mwa_obj_difference.pdf}
	}
	\hfill
	\subfloat[OGD, fixed learning rate]{
		\includegraphics[width=0.31\linewidth]{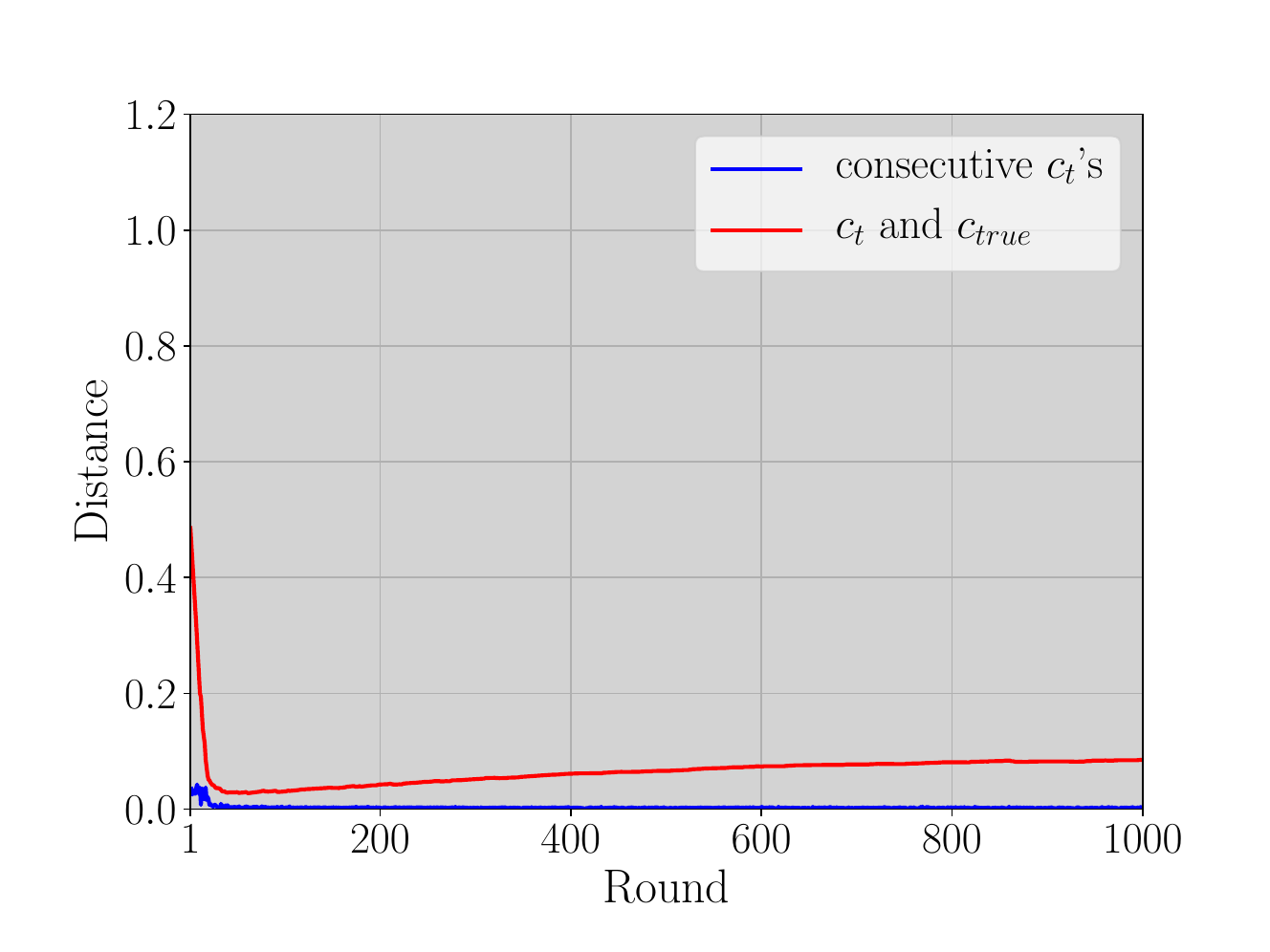}
	}
	\subfloat[OGD, dynamic learning rate]{
		\includegraphics[width=0.31\linewidth]{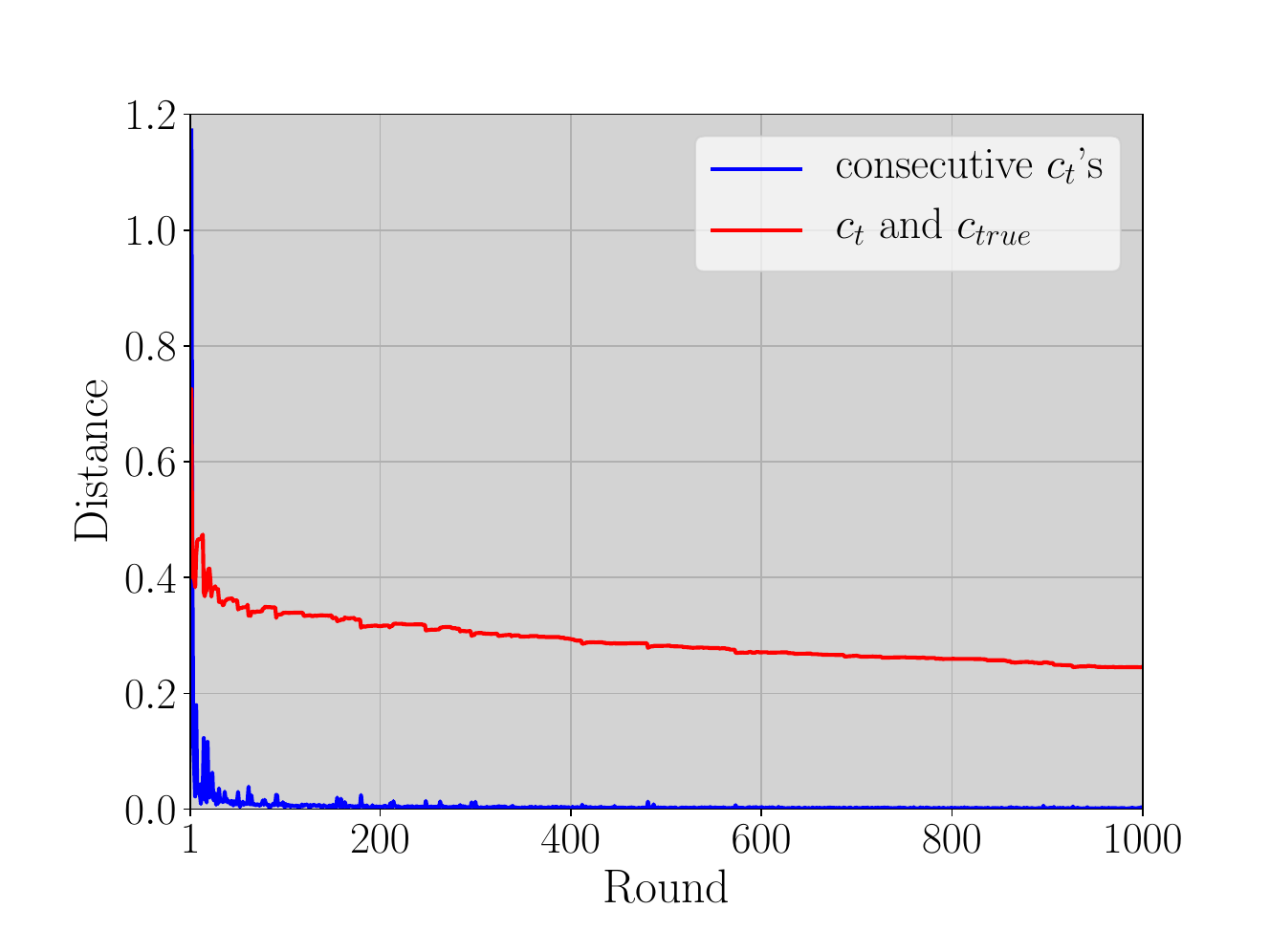}
	}
	\hfill\null
	\caption{Distance in the $1$-norm between two consecutive objective objective functions (\emph{blue}) and between the objective function in the current round and the true objective function (\emph{red}) for the integer knapsack problem}
	\label{Fig:IntegerKnapsack-Updates1}
\end{figure}
\begin{figure}[h]
	\hfill
	\subfloat[MWU]{
		\includegraphics[width=0.31\linewidth]{\graphicpathknp/knp_t1000_s1_n1000_IP_mwa_convergence.pdf}
	}
	\hfill
	\subfloat[OGD, fixed learning rate]{
		\includegraphics[width=0.31\linewidth]{\graphicpathknp/knp_t1000_s1_n1000_IP_ogd_convergence.pdf}
	}
	\subfloat[OGD, dynamic learning rate]{
		\includegraphics[width=0.31\linewidth]{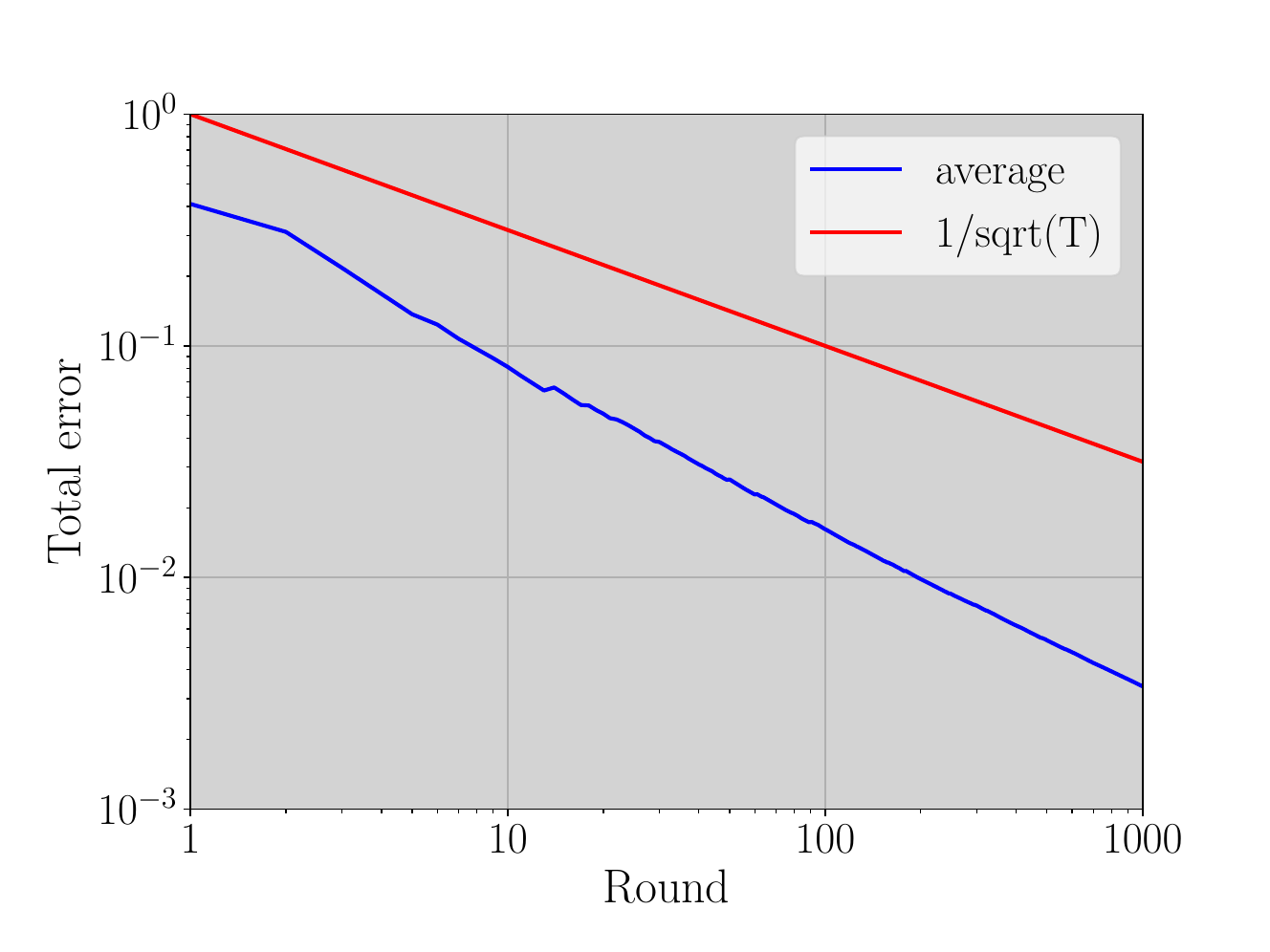}
	}
	\hfill\null
	\caption{Doubly logarithmic plot of the average total error (\emph{blue}) for the integer knapsack problem against the function $ f(T) = 1/\sqrt{T} $ (\emph{red})}
	\label{Fig:IntegerKnapsack-Convergence1}
\end{figure}
In Figure~\ref{Fig:IntegerKnapsack-Updates1},
it is visible at first sight that empirically
the learned objective does not converge to the true objective function.
In the case of MWU, the distance to the true objective function
converges to about~$ 0.1 $.
For OGD with fixed learning rate $ \eta = D/G\sqrt{T} $,
the behaviour is very similar.
When we choose the dynamic learning rate $ \eta_t = D/G\sqrt{t} $,
with a higher scale of the updates at the beginning,
the convergence is slower by a factor two,
which is in line with our treatment
of the choice of learning rates in Section~\ref{Sec:Objective.OGD}.

Finally, Figure~\ref{Fig:IntegerKnapsack-Convergence1}
shows the MWA and two OGD algorithms for the integer knapsack,
depicting the average total error up to a given round
versus the asymptotic upper bound
represented by the function $ f(T) = 1/\sqrt{T} $
on a log-log-scale.
What we observe is that -- over the total observation horizon --
the order of convergence is basically same as that of~$f$,
where, again, all algorithms perform about equally well.

\newpage

\section{Visualization of the Solution for the Learning of Travels Times}

In the following, we give a visualization of the solutions produced
by the original as well as the learned objective functions
for the problem of learning travel times
presented in Section~\ref{Sec:Application.Paths}.
Figure~\ref{Fig:ShortestPathProblem-Paths-Original} shows the graph
corresponding to the Chicago street network,
where we have only plotted the arcs which were actually used
in a any of the shortest paths over 60~rounds,
assuming the constant free-flow objective.
The darker an arc is coloured,
the higher is the number of actually taken paths
which it is part of.
We can clearly recognize several highways
as well as shortcuts through the city center.
In Figure~\ref{Fig:ShortestPathProblem-Paths-Learned},
we show the same figure,
but with the paths learned by MWU without congestions.
We see that the algorithm chooses a couple of routes
which would not be chosen according to the true objective,
expectably at the beginning,
as it first needs to learn which ones are the good arcs
(starting from the assumption that all arcs are of equal travel time).
We clearly observe that the most-frequently used arcs are the same ones
as with the true objective.
The same holds for the cases of gradual and abrupt congestions,
shown in Figures~\ref{Fig:ShortestPathProblem-Paths-SC}
and~\ref{Fig:ShortestPathProblem-Paths-FC} respectively,
with the exception that the most-frequent arcs are not chosen as often,
as they take longer to traverse
or are even blocked for a considerable amount of time.
Instead, a few arcs belonging to detours become more interesting.
This shows that our algorithm quickly adapt even after a major shift
in the learning target.
\begin{figure}[H]
	\hfill
	\subfloat[True objective]{
		\includegraphics[width=0.48\linewidth,trim={12cm 12cm  12cm 12cm},clip]{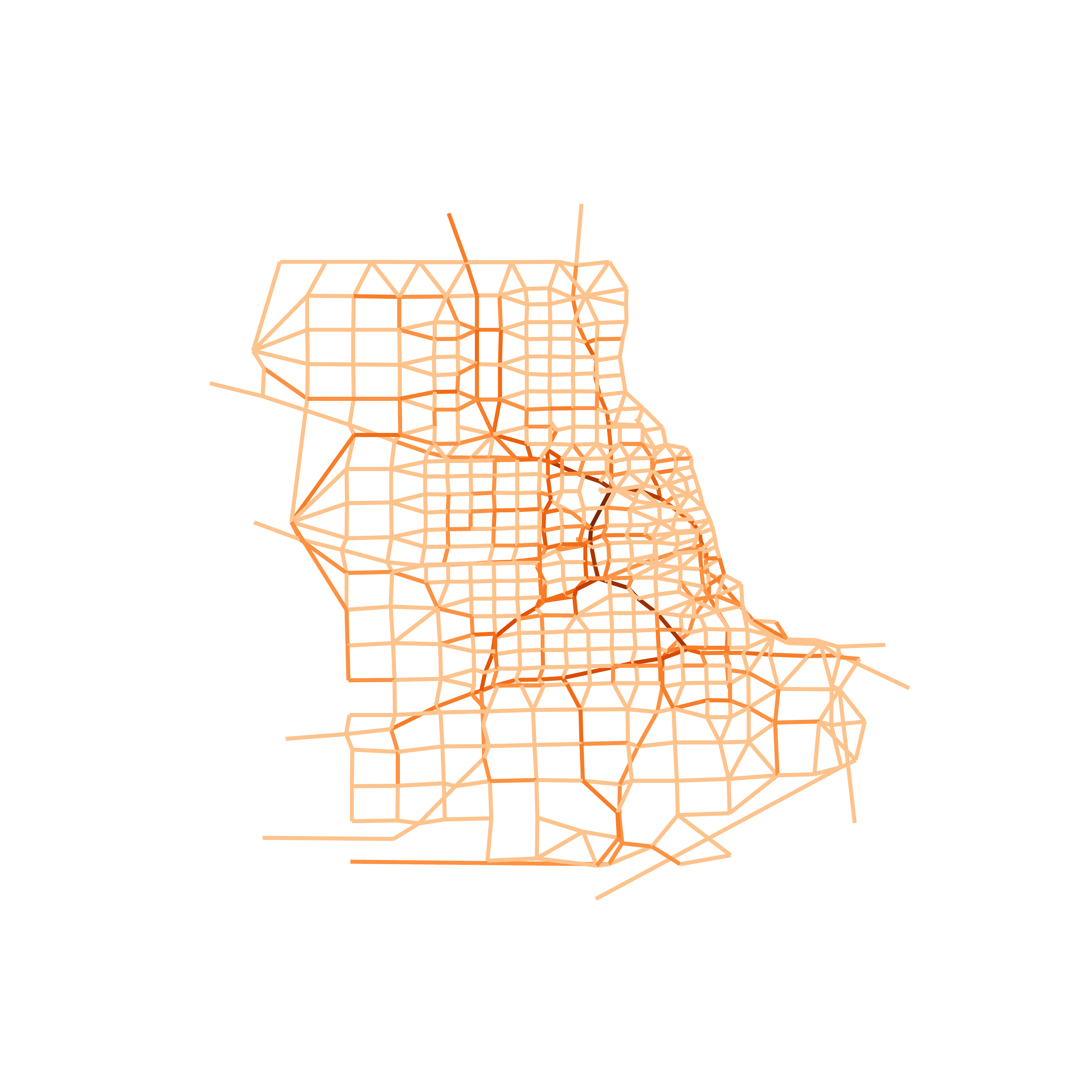}
		\label{Fig:ShortestPathProblem-Paths-Original}
	}	
	\hfill
	\subfloat[No congestion]{
		\includegraphics[width=0.48\linewidth,trim={12cm 12cm  12cm 12cm},clip]{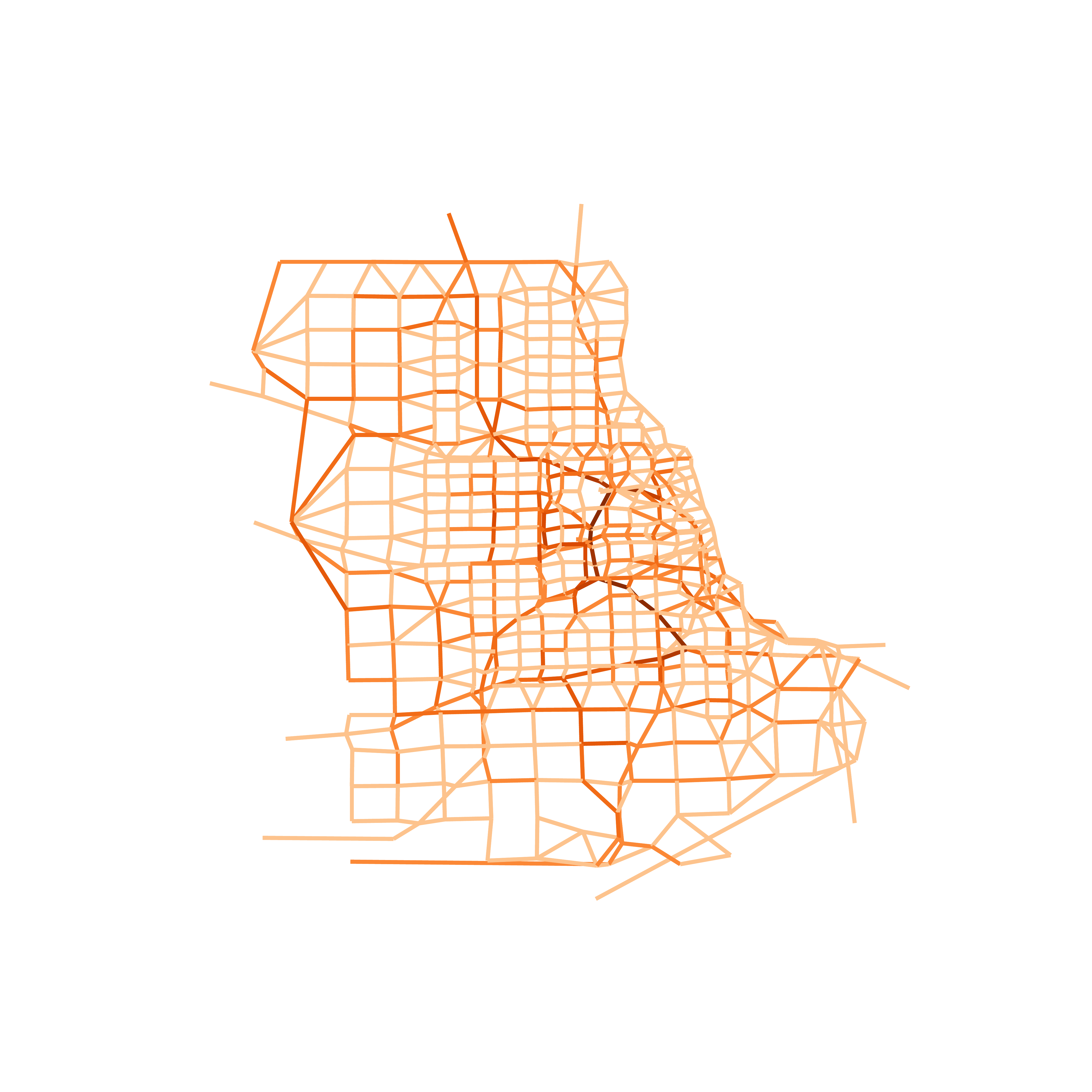}
		\label{Fig:ShortestPathProblem-Paths-Learned}
	}	
	\hfill\null
	
	\hfill
	\subfloat[Gradual congestion]{
		\includegraphics[width=0.48\linewidth,trim={12cm 12cm  12cm 12cm},clip]{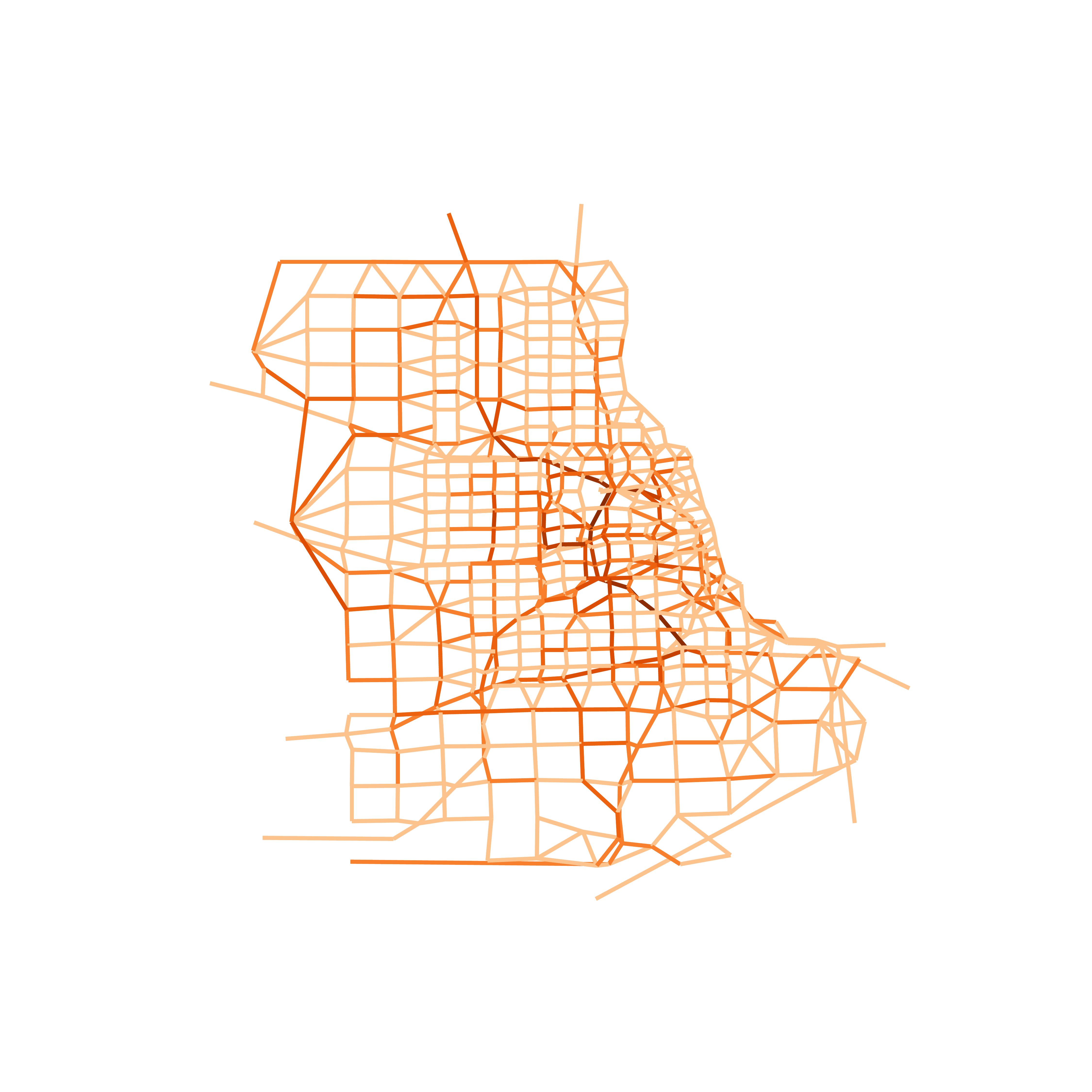}
		\label{Fig:ShortestPathProblem-Paths-SC}
	}
	\hfill
	\subfloat[Abrupt congestion]{
		\includegraphics[width=0.48\linewidth,trim={12cm 12cm  12cm 12cm},clip]{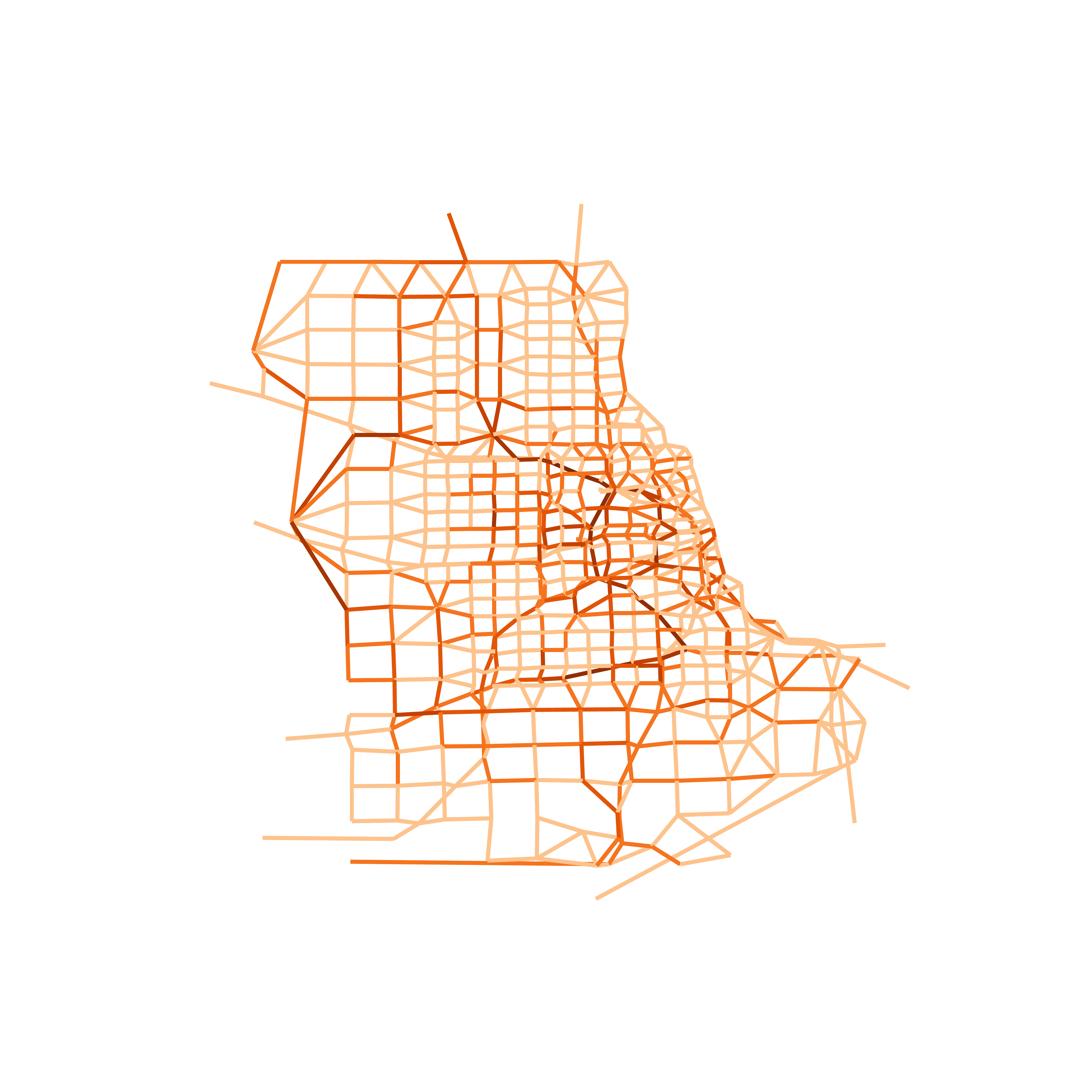}
		\label{Fig:ShortestPathProblem-Paths-FC}
	}
	\hfill\null
	\caption{Actual paths taken in the network for the true objective and the objectives learned by MWU under the three different congestion types}
	\label{Fig:ShortestPathProblem-Paths}
\end{figure}

\end{document}